\documentclass[11pt,reqno]{amsart}

\usepackage{amssymb,amsmath,amsthm,amsfonts}
\usepackage{subfigure, hyperref, multirow, bm, color, stmaryrd, marginnote, graphicx}

\usepackage[margin=1in]{geometry}
\usepackage{etoolbox}
\patchcmd{\section}{\scshape}{\bfseries\Large}{}{}
\makeatletter
\renewcommand{\@secnumfont}{\bfseries}
\makeatother
\makeatletter
\renewcommand{\tocsection}[3]{%
  \indentlabel{\@ifnotempty{#2}{\bfseries\ignorespaces#1 #2\quad}}\bfseries#3}
\renewcommand{\tocsubsection}[3]{%
  \indentlabel{\@ifnotempty{#2}{\ignorespaces#1 #2\quad}}#3}

\newcommand\@dotsep{4.5}
\def\@tocline#1#2#3#4#5#6#7{\relax
  \ifnum #1>\c@tocdepth 
  \else
    \par \addpenalty\@secpenalty\addvspace{#2}%
    \begingroup \hyphenpenalty\@M
    \@ifempty{#4}{%
      \@tempdima\csname r@tocindent\number#1\endcsname\relax
    }{%
      \@tempdima#4\relax
    }%
    \parindent\z@ \leftskip#3\relax \advance\leftskip\@tempdima\relax
    \rightskip\@pnumwidth plus1em \parfillskip-\@pnumwidth
    #5\leavevmode\hskip-\@tempdima{#6}\nobreak
    \leaders\hbox{$\m@th\mkern \@dotsep mu\hbox{.}\mkern \@dotsep mu$}\hfill
    \nobreak
    \hbox to\@pnumwidth{\@tocpagenum{\ifnum#1=1\bfseries\fi#7}}\par
    \nobreak
    \endgroup
  \fi}
\AtBeginDocument{%
\expandafter\renewcommand\csname r@tocindent0\endcsname{0pt}
}
\def\l@subsection{\@tocline{2}{0pt}{2.5pc}{5pc}{}}
\makeatother

\newcommand{\R}{\mathbb{R}}
\newcommand{\Z}{\mathbb{Z}}

\newcommand{\T}{\mathbb{T}}
\newcommand{\A}{\mathcal{A}}
\newcommand{\E}{\mathcal{E}}

\newcommand{\bu}{\bm{u}}

\newcommand{\bx}{\bm{x}}

\newcommand{\X}{\bm{X}}
\newcommand{\be}{\bm{e}}
\newcommand{\bv}{\bm{v}}
\newcommand{\p}{\partial}
\newcommand{\dive}{{\rm{div}}}
\newcommand{\SB}{{\rm SB}}

\newcommand{\floor}[1]{\lfloor #1 \rfloor}

\newcommand{\abs}[1]{\left\lvert #1 \right\rvert}
\newcommand{\norm}[1]{\left\lVert #1 \right\rVert}
\newcommand{\wh}[1]{\widehat{#1}}
\newcommand{\wt}[1]{\widetilde{#1}}
\newcommand{\mc}[1]{\mathcal{#1}}

\newcommand{\paren}[1]{\left(#1\right)}
\newcommand{\PD}[2]{\frac{\partial#1}{\partial#2}}
\newcommand{\PDD}[3]{\frac{\partial^{#1}{#2}}{\partial{#3}^{#1}}}

\newtheorem{theorem}{Theorem}[section]
\newtheorem{lemma}[theorem]{Lemma}
\newtheorem{proposition}[theorem]{Proposition}

\theoremstyle{definition}
\newtheorem{remark}[theorem]{Remark}

\allowdisplaybreaks

\begin{document}
\title[Accuracy of SBT in approximating force exerted by thin fiber on viscous fluid]{Accuracy of slender body theory in approximating force exerted by thin fiber on viscous fluid}
\author{Yoichiro Mori}
\address{Department of Mathematics, University of Pennsylvania, Philadelphia, PA 19104}
\email{y1mori@sas.upenn.edu}

\author{Laurel Ohm}
\address{School of Mathematics, University of Minnesota, Minneapolis, MN 55455}
\email[Corresponding author]{ohmxx039@umn.edu}

\thanks{This research was supported in part by NSF DMS-1907583 and the Simons Foundation Math+X grant awarded to Y.M., and a UMN Doctoral Dissertation Fellowship, awarded to L.O.}

\maketitle

\begin{abstract}
We consider the accuracy of slender body theory in approximating the force exerted by a thin fiber on the surrounding viscous fluid when the fiber velocity is prescribed. We term this the \emph{slender body inverse problem}, as it is known that slender body theory converges to a well-posed PDE solution when the force is prescribed and the fiber velocity is unknown. From a PDE perspective, the slender body inverse problem is simply the Dirichlet problem for the Stokes equations, but from an approximation perspective, nonlocal slender body theory exhibits high wavenumber instabilities which complicate analysis. Here we consider two methods for regularizing the slender body approximation: spectral truncation and the $\delta$-regularization of Tornberg and Shelley (2004). For a straight, periodic fiber with constant radius $\epsilon>0$, we explicitly calculate the spectrum of the operator mapping fiber velocity to force for both the PDE and the approximations. We show that the spectrum of the original slender body approximation agrees closely with the PDE solution at low wavenumbers but differs at high frequencies, allowing us to define a truncated approximation with a wavenumber cutoff $\sim1/\epsilon$. For both the truncated and $\delta$-regularized approximations, we obtain similar convergence results to the PDE solution as $\epsilon\to0$: a fiber velocity with $H^1$ regularity gives $O(\epsilon)$ convergence, while a fiber velocity with at least $H^2$ regularity yields $O(\epsilon^2)$ convergence. Moreover, we determine the dependence of the $\delta$-regularized error estimate on the regularization parameter $\delta$. 
\end{abstract}

\tableofcontents

\section{Introduction}
Slender body theory seeks to describe the dynamics of thin filaments immersed in a 3D Stokes fluid \cite{batchelor1970slender,  johnson1980improved, keller1976slender, lighthill1976flagellar, tornberg2004simulating}. A central problem in slender body theory is to find the velocity of the surrounding fluid $\bm{u}$ and of the fiber $\bm{u}^\SB_{\rm C}$ given a 1D force density $\bm{f}$ along the fiber centerline, for which different approximation methods have been proposed \cite{nguyen2014computing, tornberg2004simulating}. In previous work \cite{rigid_SBT, closed_loop, free_ends}, the authors and D. Spirn investigated the accuracy of a canonical approximation to this problem due to Keller--Rubinow and Johnson \cite{johnson1980improved, keller1976slender}. To do so, we introduced the \emph{slender body PDE}, proved its well-posedness, and derived an error estimate between the slender body PDE and the slender body approximation in terms of the fiber radius $\epsilon$ as $\epsilon\to0$. \\

In this paper, we address the validity of using slender body approximation for the \emph{slender body inverse problem}. Here the velocity of the fiber $\bm{u}^\SB_{\rm C}$ is given and we aim to find the 1D force density $\bm{f}$. This problems arises, for example, when one tries to infer the force generated by a filamentous microorganism given its shape dynamics \cite{nguyen2014computing}. This problem is also relevant when an inextensibility constraint must be imposed on a moving filament; the filament tension that ensures inextensibility can be found by solving a slender body inverse problem \cite{tornberg2004simulating}. From a PDE perspective, the slender body inverse problem is simply the Dirichlet problem for the Stokes equations with a restricted class of admissible boundary data, for which the existence, uniqueness, and regularity of solutions are well understood. However, from the perspective of slender body approximation, the solvability of the inverse problem is less clear. In particular, an analysis by G\"otz \cite{gotz2000interactions} showed that, in a simplified scenario, the operator mapping $\bm{f}$ to $\bu^\SB_{\rm C}$ is not necessarily invertible for the Keller--Rubinow--Johnson slender body approximation. We will recall this analysis and explore the issue in greater detail here. This non-invertibility means that the Keller--Rubinow--Johnson slender body approximation is not suitable for the inverse problem, and we are thus led to consider regularizations of their original expression. In particular, we will consider both spectral truncation and what we will term the \emph{$\delta$-regularization} approach of Tornberg and Shelley \cite{tornberg2004simulating}. Neither the validity of such approximations nor the choice of parameter $\delta$ has been clear. Our previous work on the slender body PDE provides the necessary framework to study these issues.\\

Here we carry out a complete spectral analysis of the slender body inverse problem in the case of a straight, periodic filament with constant radius $\epsilon>0$. This scenario provides intuition for more complicated fiber geometries and also corresponds to the model problem studied by G\"otz and others \cite{gotz2000interactions, shelley2000stokesian, tornberg2004simulating}. Moreover, in this model scenario, we can explicitly calculate the eigenvalues of the operator mapping the fiber velocity to the force density $\bm{f}$ for both the slender body PDE and the slender body approximation. This involves finding closed-form solutions to a family of inhomogeneous Bessel ODEs. We show that the eigenvalues of the slender body approximation agree closely with the slender body PDE at low wavenumbers, but deviate wildly beyond a threshold wavenumber of $O(1/\epsilon)$. This allows us to define a truncated slender body approximation, where frequencies above this threshold wavenumber are cut off, and for which we can derive a rigorous error bound with respect to the slender body PDE. If the prescribed fiber velocity $\bu$ belongs to $H^1$, we obtain an error bound for $\bm{f}$ proportional to $\epsilon$. If $\bu\in H^2$ or smoother, we obtain $\epsilon^2$ convergence. One of our main technical tools in deriving these estimates is a new bound on the ratio between modified Bessel functions. By similar methods, we show that the $\delta$-regularized approximation yields the same order of convergence to the slender body PDE without requiring truncation. In addition, we determine the dependence of this error estimate on the regularization parameter $\delta$, which provides a guideline for the choice of $\delta$. Before proving these results for the Stokes setting, we perform an analogous analysis for the Laplace setting, where the spectral calculations and estimates needed are much simpler. The Stokes setting relies on the same fundamental ideas but is much more technically involved. \\

The spectral picture obtained in this paper is of significance beyond its direct implications for the slender body inverse problem. In particular, our spectral calculation suggests the high mode relaxation behavior of thin elastic filaments in a Stokesian fluid, which will inform analysis and numerical analysis of slender body problems as well as algorithmic development for the numerical simulation of such problems.

\subsection{Slender body theory and slender body PDE}\label{sec:SBTandPDE}
Let $\X : \T\equiv \R / 2\Z \to \R^3$ be the coordinates of a closed curve in $\R^3$, parameterized by arclength $s$, and for $\epsilon>0$ define 
\[ \Sigma_\epsilon = \{ \bx \in \R^3 : \text{ dist}(\bx,\X(s)) < \epsilon \}, \]
a closed loop slender body with constant radius $\epsilon$ (see Figure \ref{fig:geom}). It is of interest in many practical applications \cite{dreyfus2005microscopic, hamalainen2011papermaking, lauga2009hydrodynamics, pak2011high, smith2011mathematical, spagnolie2011comparative, young2010dynamics} to describe the flow about $\Sigma_\epsilon$ when the fiber is immersed in a highly viscous fluid. The goal of slender body theory \cite{batchelor1970slender,  johnson1980improved, keller1976slender, lighthill1976flagellar} is to approximate the Stokes flow about $\Sigma_\epsilon$ when $\epsilon$ is small. \\

Slender body theory approximates the fluid velocity $\bu^\SB(\bx)$ at any $\bx$ away from $\X(s)$ as the flow in $\R^3$ due to a force density $\bm{f}(s)$, $s\in \T$, along the 1D fiber centerline. In particular, the fluid velocity is approximated by 
\begin{equation}\label{stokes_SB}
\begin{aligned}
8\pi \bu^{\SB}(\bx) &=\int_\T \bigg( \mc{S}(\bm{R})+\frac{\epsilon^2}{2}\mc{D}(\bm{R}) \bigg)\bm{f}(s') \, ds', \quad 
\bm{R}=\bm{x}-\bm{X}(s'); \\
\mc{S}(\bm{R})&=\frac{{\bf I}}{\abs{\bm{R}}}+\frac{\bm{R}\bm{R}^{\rm T}}{\abs{\bm{R}}^3}, \; 
\mc{D}(\bm{R})=\frac{{\bf I}}{\abs{\bm{R}}^3}-\frac{3\bm{R}\bm{R}^{\rm T}}{\abs{\bm{R}}^5},
\end{aligned}
\end{equation}
where $\frac{1}{8\pi}\mc{S}(\bm{R})$ is the Stokeslet, the free space Green's function for the Stokes equations in $\R^3$, and $\frac{1}{8\pi}\mc{D}(\bm{R})=\frac{1}{16\pi}\Delta\mc{S}(\bm{R})$ is the doublet, a higher order correction to the velocity approximation. The doublet term serves to enforce a \emph{fiber integrity condition}: to leading order in $\epsilon$, the velocity is constant across each cross section (fixed $s$) of $\Sigma_\epsilon$. \\


Now, the expression \eqref{stokes_SB} is singular at $\bx=\X(s)$ and thus only valid away from the filament centerline. A common method for obtaining an expression for the velocity of the fiber itself is to perform a matched asymptotic expansion about $\epsilon=0$ \cite{gotz2000interactions, keller1976slender, johnson1980improved, shelley2000stokesian, tornberg2004simulating}. This yields the following expression for the fiber velocity in the periodic setting: 
\begin{equation}\label{SBT_expr}
\begin{aligned}
\bu^\SB_{\rm C}(s) &= \mc{L}_\epsilon^\SB[\bm{f}](s) := \bm{\Lambda}[\bm{f}](s) + \bm{K}[\bm{f}](s),\\
\bm{\Lambda}[\bm{f}](s) &:=  \frac{1}{8\pi}\big[({\bf I}- 3\be_t\be_t^{\rm T})-2({\bf I}+\be_t\be_t^{\rm T}) \log(\pi\epsilon/8) \big]{\bm f}(s) \\
\bm{K}[\bm{f}](s) &:= \frac{1}{8\pi}\int_{\T} \left[ \left(\frac{{\bf I}}{|\bm{R}_0|}+ \frac{\bm{R}_0\bm{R}_0^{\rm T}}{|\bm{R}_0|^3}\right){\bm f}(s') - \frac{\pi}{2}\frac{{\bf I}+\be_t(s)\be_t(s)^{\rm T} }{|\sin (\pi(s-s')/2)|} {\bm f}(s)\right] \, ds'.
\end{aligned}
\end{equation}
Here $\be_t(s)$ is the unit tangent vector to $\X(s)$ and $\bm{R}_0(s,s') = \X(s) - \X(s')$. Note that the integral operator $\bm{K}$ only makes sense as a difference of the two terms.  \\

%
A natural question to then ask of slender body theory is how well does the expression \eqref{SBT_expr} approximate the actual motion of a thin fiber in Stokes flow? \\  

In \cite{closed_loop,free_ends}, Mori et al. develop the \emph{slender body PDE} \eqref{SB_PDE} to answer this question when the force density $\bm{f}$ is given and the slender body velocity is unknown. Let $p=p(\bx)$ denote the scalar-valued fluid pressure field, $\bm{\sigma}= \nabla\bu+(\nabla\bu)^{\rm T} - p{\bf I}$ the fluid stress tensor, and $\bm{\sigma}\bm{n}\big|_{\p\Sigma_\epsilon}$ the surface stress on $\p\Sigma_\epsilon$, where $\bm{n}(\bx)$ is the unit normal to $\bx\in \p\Sigma_\epsilon$. The fluid velocity field $\bu(\bx)$ about $\Sigma_\epsilon$ is described by the solution to the boundary value problem 
\begin{equation}\label{SB_PDE}
\begin{aligned}
-\Delta \bu +\nabla p &= 0, \quad \dive \, \bu = 0 \quad \text{ in } \Omega_\epsilon = \R^3\backslash \overline{\Sigma_\epsilon} \\
\int_0^{2\pi} (\bm{\sigma}\bm{n}) \mc{J}_\epsilon(s,\theta) \, d\theta &= \bm{f}(s) \hspace{2cm} \text{ on } \p\Sigma_\epsilon \\
\bu\big|_{\p\Sigma_\epsilon} &= \bu(s), \hspace{1.85cm} \text{ unknown but independent of }\theta \\
\abs{\bu} \to 0 & \text{ as } \abs{\bx} \to \infty.
\end{aligned}
\end{equation}
Here $\mathcal{J}_{\epsilon}(s,\theta)$ is the surface element on $\p\Sigma_\epsilon$. The velocity $\bu\big|_{\p\Sigma_\epsilon}$ of the slender body itself is unspecified but constrained to belong to the set 
\[ \A_\epsilon = \{ \bv \in X(\Omega_\epsilon) \, : \, \bv\big|_{\p\Sigma_\epsilon} = \bv(s) \}, \]
where $X(\Omega_\epsilon)$ is a suitable function space. As mentioned, this constant-on-cross-sections constraint for the slender body velocity is known as the fiber integrity condition and ensures that cross sections of $\Sigma_\epsilon$ maintain their circular shape and do not deform. Due to the fiber integrity condition, the slender body velocity $\bu(s)$ may be considered both as a function on $\p\Sigma_\epsilon$ and on $\T$. \\

\begin{figure}[!h]
\centering
\includegraphics[scale=0.6]{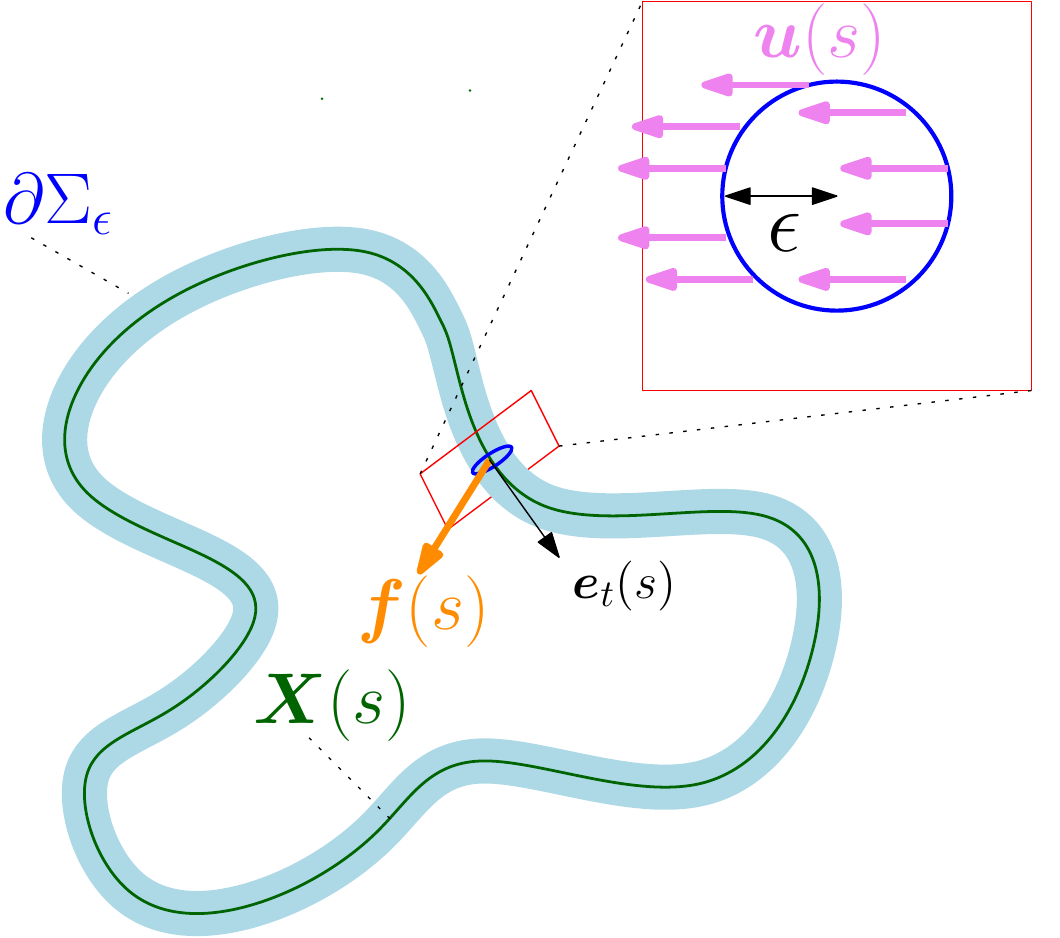} \hspace{1.5cm}
\includegraphics[scale=0.6]{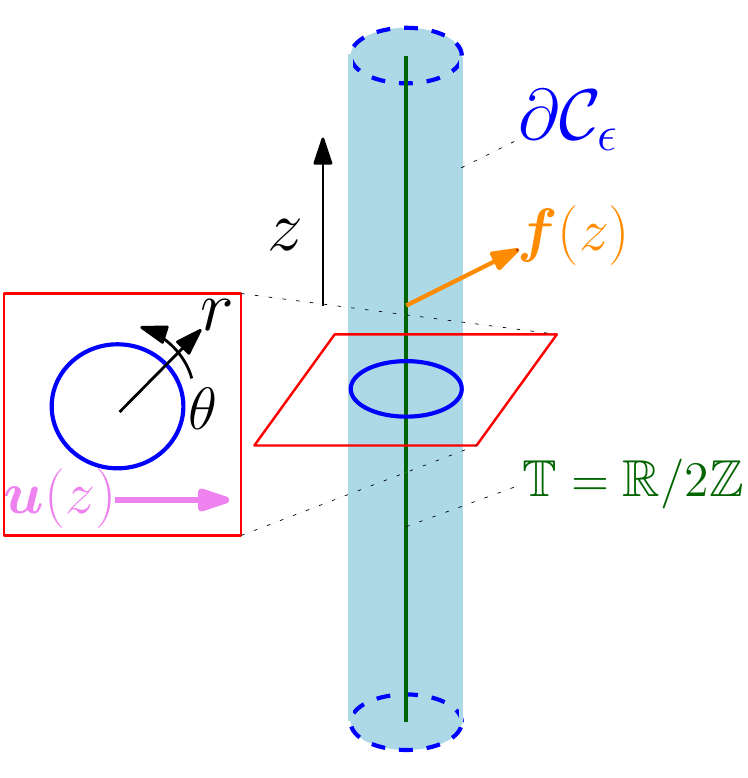}\\
\caption{In general, the slender body PDE \eqref{SB_PDE} is defined about a closed fiber $\Sigma_\epsilon$ (left). We are interested in the map $\mc{L}_\epsilon$ \eqref{Lmap}, defined along the 1D curve $\X(s)$. Here we will consider in detail the model case of a straight, periodic fiber $\Sigma_\epsilon=\mc{C}_\epsilon$ (right), since this will allow us to explicitly calculate the spectrum of $\mc{L}_\epsilon$. }
\label{fig:geom}
\end{figure}

Given a force density $\bm{f}(s)$, $s\in \T$, we may thus define the map 
\begin{equation}\label{Lmap}
 \mc{L}_\epsilon: \bm{f}(s) \mapsto \bu(s) 
 \end{equation}
by solving \eqref{SB_PDE} and measuring $\bu\big|_{\p\Sigma_\epsilon}$. In \cite{closed_loop,free_ends}, Mori et al. show that this map is well posed $L^2(\T)\to L^2(\T)$ with an $\epsilon$-dependent energy estimate. In particular, for $\bm{f}\in L^2(\T)$, we have
\begin{equation}\label{stokes_est}
\norm{ \mc{L}_\epsilon[\bm{f}]}_{L^2(\T)} \le C\abs{\log\epsilon}\norm{\bm{f}}_{L^2(\T)}.
\end{equation}
In fact, for fixed $\epsilon$, standard well-posedness theory for the Stokes equations along with the analysis in \cite{closed_loop,free_ends} can be used to show that this map is well-posed from $H^{-1/2}(\T)\to H^{1/2}(\T)$. However, well-posedness in these function spaces does not yield a similar $\epsilon$-dependent energy estimate to \eqref{stokes_est}, since the proof of $\abs{\log\epsilon}^{1/2}$ dependence relies on an $L^2$ bound for the trace of $\A_\epsilon$ functions which does not have an $H^{1/2}$ counterpart. Note, however, that the map $\mc{L}_\epsilon$ is then compact $L^2(\T)\to L^2(\T)$. \\

In \cite{closed_loop,free_ends}, Mori et al. study the slender body approximation $\mc{L}_\epsilon^\SB$ to the operator $\mc{L}_\epsilon$, where $\mc{L}_\epsilon^\SB[\bm{f}](s)$ is given by the formula \eqref{SBT_expr}. Under mild regularity assumptions on $\Sigma_\epsilon$ and on $\bm{f}$, it is shown that 
\begin{equation}\label{forward_err}
 \norm{ \mc{L}_\epsilon^\SB[\bm{f}]-\mc{L}_\epsilon[\bm{f}]}_{L^2(\T)} \le C\epsilon\abs{\log\epsilon}^{3/2} \norm{\bm{f}}_{C^1(\T)} .
 \end{equation}
 
However, what if we are instead given the velocity $\bu^\SB_{\rm C}(s)=\overline \bu(s)$ of the slender body and wish to solve for the force density along $\Sigma_\epsilon$? It seems natural to try to use \eqref{SBT_expr} to approximate $\mc{L}_\epsilon^{-1}[\overline \bu](s)$, since the sum $\bm{\Lambda}+\bm{K}$ looks similar to a second-kind Fredholm integral equation for $\bm{f}(s)$. However, as noted in various papers \cite{gotz2000interactions,shelley2000stokesian,tornberg2004simulating}, this is not the case, as the operator $\bm{K}$ is not compact. Moreover, in simple geometries, the spectrum of $\bm{\Lambda}+\bm{K}$ changes sign at finite wavenumber. Thus equation \eqref{SBT_expr} alone is not suitable for approximating solutions to the slender body inverse problem. \\ 

In this paper we consider the validity of using two different regularizations of the slender body approximation \eqref{SBT_expr} to approximate $\mc{L}_\epsilon^{-1}$. We consider spectral truncation and the $\delta$-regularization technique of Tornberg and Shelley \cite{tornberg2004simulating}, both of which correct for the invertibility issues in \eqref{SBT_expr}. We aim to compare these regularizations to the slender body PDE solution via a detailed study of the spectrum of $\mc{L}_\epsilon^{-1}$.  \\

For a general periodic slender body $\Sigma_\epsilon$, we may think of the map $\mc{L}_\epsilon^{-1}$ as follows. Given $\theta$-independent velocity data $\overline \bu(s)\in H^{1/2}(\T)$, let $(\bu,p)$ be the unique solution to the corresponding Stokes Dirichlet boundary value problem. The existence, uniqueness, and regularity properties of solutions to the Stokes Dirichlet problem are well-studied (see \cite{galdi2011introduction, boyer2012mathematical} for an in-depth treatment). Using the solution $(\bu,p)$, we may compute the surface stress $\bm{\sigma}\bm{n}$ on $\p\Sigma_\epsilon$. Due to the standard regularity properties of the Stokes equations, we have $\bm{\sigma}\bm{n}\in H^{-1/2}(\p\Sigma_\epsilon)$. The Dirichlet-to-Neumann operator $\mc{L}^{-1}$ mapping the boundary value $\overline{\bu}(s)$ to the corresponding surface stress $\bm{\sigma}\bm{n}$ may then be defined:
\begin{equation}\label{usual_DTN}
\begin{aligned}
\mc{L}^{-1} : H^{1/2}(\Gamma_\epsilon) &\to H^{-1/2}(\Gamma_\epsilon); \\
\mc{L}^{-1}[\overline{\bu}](s) &= \bm{\sigma} {\bm n}.
\end{aligned}
\end{equation}

The inverse slender body PDE map $\mc{L}_\epsilon^{-1}$ may be considered as the weighted $\theta$-integral of the Dirichlet-to-Neumann map \eqref{usual_DTN} over each fiber cross section, where the weight is given by the surface element $\mc{J}_\epsilon(s,\theta)$ on $\p\Sigma_\epsilon$. We have 
\begin{equation}\label{avg_DTN}
\mc{L}_\epsilon^{-1} : H^{1/2}(\T) \to H^{-1/2}(\T); \quad \mc{L}_\epsilon^{-1} [\overline \bu](s) = \bm{f}(s),
\end{equation}
where the formula for $\bm{f}(s)$ is as in \eqref{SB_PDE}. Here, given any $\overline\varphi\in H^{1/2}(\T)$, we understand this $\bm{f}\in H^{-1/2}(\T)$ by the dual pairing
\[ \langle \bm{f}, \overline\varphi \rangle_{H^{-1/2}(\T),H^{1/2}(\T)} = \int_{\Omega_\epsilon} \bm{\sigma} : \nabla \varphi \, d\bx, \]
where $\varphi(\bx)$ denotes the $H^1$ extension of the boundary value $\overline\varphi(s)$ into $\Omega_\epsilon$. \\

%
We may note some general properties of the spectrum of $\mc{L}_\epsilon$. In addition to being compact $L^2(\T)\to L^2(\T)$, the operator $\mc{L}_\epsilon$ is self adjoint, which can be seen as follows. Let $\bm{f}$, $\bm{g}\in L^2(\T)$ and let $\bu$, $\bu_g$ be the solutions to the slender body PDE \eqref{SB_PDE} with $\bm{f}$ and $\bm{g}$ as force data, respectively. Then, using Definition 2.1 in \cite{closed_loop} of a weak solution to the slender body PDE, we have 
\begin{align*}
\int_{\T} \big( \mc{L}_\epsilon[ \bm{f}](s) \big) \bm{g}(s) \, ds &= \int_{\Omega_\epsilon} 2\, \E(\bu):\E(\bu_g) \, d\bx = \int_{\T} \bm{f}(s) \big(\mc{L}_\epsilon[\bm{g}](s) \big) \, ds.
\end{align*}

Here $\E(\bu)=\frac{1}{2}(\nabla \bu+(\nabla\bu)^{\rm T})$ denotes the symmetric gradient. By the spectral theorem for compact, self-adjoint operators, the operator $\mc{L}_\epsilon$ admits a countable basis of orthonormal eigenvectors $\{\psi_j \}$ with corresponding eigenvalues $\{\mu_j\}\subset \R$ satisfying $\mu_j\to 0$. Thus $\mc{L}_\epsilon^{-1}$ admits real-valued eigenvalues $\lambda_j \to \infty$ along with a basis of orthonormal eigenvectors $\{\phi_j\}$. \\

To find the eigenvalues of $\mc{L}_\epsilon^{-1}$, we look for $\lambda$ satisfying
 \begin{equation}\label{eigenval_prob}
\begin{aligned}
-\Delta \bu +\nabla p &= 0, \quad \dive \, \bu = 0 \hspace{.5cm} \text{ in } \Omega_{\epsilon} \\
 \int_0^{2\pi} (\bm{\sigma} {\bm n}) \, \mathcal{J}_{\epsilon}(s,\theta) \, d\theta &= \lambda\bu(s) \hspace{1.85cm} \text{ on }\Gamma_\epsilon \\
\abs{\bu} \to 0 & \text{ as } \abs{\bx} \to \infty.
\end{aligned}
\end{equation}

%

Since \eqref{SB_PDE} is uniquely solvable, \eqref{eigenval_prob} has no zero eigenvalue, i.e. $\mc{L}_\epsilon^{-1}$ is positive definite. 
In general, the behavior of the eigenvalues of $\mc{L}_\epsilon^{-1}$ can be calculated via the Rayleigh quotient 
\begin{equation}\label{rayleigh}
\lambda_k(\Omega_\epsilon) = \min_{\bu\in \A_\epsilon; \, \dive\, \bu =0} \bigg\{ \frac{\int_{\Omega_\epsilon}2\abs{\E(\bu)}^2 d\bx}{\int_{\T}\abs{\bu(s)}^2 ds} \, : \, \int_{\T} \bu(s)\cdot\bm{\phi}_j(s) = 0, \, j = 0,\dots,k-1 \bigg\}
\end{equation}
where $\bm{\phi}_j(s)$ is the $j^{\text{th}}$ eigenvector for $\mc{L}_\epsilon^{-1}$. From now on we consider only the $k\neq 0$ modes, as the $k=0$ mode corresponds to the fundamental solution in 2D, which leads to logarithmic growth of the velocity field at spatial infinity. \\

It is difficult to say much more in general about the eigenvalue problems \eqref{eigenval_prob} and \eqref{SBlap_BVP}. For the remainder of this paper, we will consider the map $\mc{L}_\epsilon^{-1}$ when $\Sigma_\epsilon$ is a straight, periodic cylinder (see Figure \ref{fig:geom}), which we denote in cylindrical coordinates $(r,\theta,z)$ by 
\[ \mc{C}_\epsilon = \{(r,\theta,z) \, : \, 0\le r\le \epsilon, \, 0\le \theta<2\pi, \,  z \in \R/2\Z \}. \]
This nonphysical geometry is nevertheless useful for developing intuition for the eigenvalue problems \eqref{eigenval_prob} and \eqref{SBlap_BVP} in more general settings. Indeed, the case $\Sigma_\epsilon=\mc{C}_\epsilon$ is especially interesting because we know the eigenfunctions explicitly. Noting that the eigenvalue problem \eqref{eigenval_prob} may be decomposed into velocity fields in the directions purely tangential ($\be_z$) and normal ($\be_x$ and $\be_y$) to the fiber centerline, the eigenfunctions in each of these directions are given by $e^{i\pi k z}\be_j$, $j=x,y,z$, $\abs{k}=1,2,3,\dots$. We are then able to calculate the full spectrum of $\mc{L}_\epsilon^{-1}$, which we state in Section \ref{main_results}.  \\

This explicit spectral information is especially useful because we can directly compare the eigenvalues of $\mc{L}_\epsilon^{-1}$ to the eigenvalues of the slender body approximation \eqref{SBT_expr}. In particular, the spectrum of the forward operator $\mc{L}^\SB_\epsilon$ about $\mc{C}_\epsilon$ has been studied in detail in \cite{gotz2000interactions,shelley2000stokesian,tornberg2004simulating}. We recall the form of the eigenvalues of $\mc{L}^\SB_\epsilon$ in Proposition \ref{KRop_eigs}. Note that since we are actually interested in approximating $\mc{L}_\epsilon^{-1}$, we denote the spectrum of $\mc{L}^\SB_\epsilon$ by $1/\lambda_k$. Further note that the eigenvalue problem for \eqref{SBT_expr} can be decomposed into tangential ($\be_z$) and normal ($\be_x$ and $\be_y$) directions, yielding two different eigenvalue expressions. 
\begin{proposition}\label{KRop_eigs}
%
The eigenvalues of the Stokes slender body approximation $\mc{L}^\SB_\epsilon$ \eqref{SBT_expr} along $\mc{C}_\epsilon$ in the tangential ($\be_z$) and normal directions ($\be_x$ and $\be_y$), respectively, are given by 
\begin{align}
\label{KR_eigs_St}
\frac{1}{\lambda^{\SB,{\rm t}}_k} &= -\frac{1}{4\pi} \big(1 + 2\log(\pi\epsilon\abs{k}/2) + 2\gamma \big) \\
\label{KR_eigs_Sn}
\frac{1}{\lambda^{\SB,{\rm n}}_k} &= \frac{1}{8\pi} \big( 1 - 2\log(\pi\epsilon\abs{k}/2) - 2\gamma \big) .
\end{align}
Here $\gamma\approx 0.5772$ is the Euler gamma.
\end{proposition}
We briefly reiterate the derivation of Proposition \ref{KRop_eigs} in Appendix \ref{sec:gotz}. \\

Immediately from Proposition \ref{KRop_eigs} we note some strong disagreement between the spectrum of $\mc{L}^\SB_\epsilon$ about $\mc{C}_\epsilon$ and the general spectral behavior of $\mc{L}_\epsilon$. In particular, in both cases $q={\rm t}$, ${\rm n}$, we see that $1/\lambda^{\SB,q}_k$ begins positive and crosses 0 for some $\abs{k}\sim 1/\epsilon$, which means that $\lambda^{\SB,q}_k$ blows up and jumps to negative values at finite $k$. \\

For this reason, we consider two forms of regularization for $\mc{L}^\SB_\epsilon$ \eqref{SBT_expr}. First, we consider truncation of the Fourier series for $\mc{L}^\SB_\epsilon$. 
%
Due to the distinct behavior in the tangential ($\be_z$) and normal ($\be_x$ and $\be_y$) directions, we distinguish between the $\be_z$ and $\be_x$, $\be_y$ directions in defining a truncated operator. For vector-valued $\bv$, we write $\bv= v_x\be_x+v_y\be_y +v_z\be_z$. We denote the Fourier coefficients of $\bv$ on $\T$ by $\wh \bv_k = \wh v_{x,k}\be_x+ \wh v_{y,k}\be_y + \wh v_{z,k}\be_z$. For $\bv\in L^2(\T)$, we then define the truncated slender body approximation $(\mc{L}_\epsilon^\SB)^{-1}_{NM}$ by
\begin{equation}\label{LambdaNM}
(\mc{L}_\epsilon^\SB)^{-1}_{NM} [\bv](z) := \sum_{\abs{k}=1}^N \lambda^{\rm t}_k \wh v_{z,k} e^{i \pi k z} \be_z + \sum_{\abs{k}=1}^M \lambda^{\rm n}_k (\wh v_{x,k} \be_x + \wh v_{y,k}\be_y) e^{i \pi k z} 
\end{equation}
for some $N$, $M<\infty$. In \eqref{LambdaN} and \eqref{LambdaNM}, the idea is to choose $N$ and $M$ well below the threshold wavenumber for blowup in $\lambda^{\SB,q}_k$, $q={\rm l}$, ${\rm t}$, ${\rm n}$. We are then able to show an error estimate between the truncated slender body approximation and the slender body PDE operator $\mc{L}_\epsilon^{-1}$, which we state in Section \ref{main_results}.  \\

Second, we consider what we will term the \emph{$\delta$-regularized slender body approximation}, proposed by Tornberg and Shelley \cite{tornberg2004simulating} and Shelley and Ueda \cite{shelley2000stokesian}. The idea behind $\delta$-regularization is to replace the integral kernel of \eqref{SBT_expr}, which is singular at $s=s'$, with a smooth kernel that is instead proportional to $\epsilon$ at $s=s'$. In the case $\Sigma_\epsilon=\mc{C}_\epsilon$, this is accomplished by using the relation 
\begin{equation}\label{periodic}
 \int_{-1}^1 \bigg(\frac{\pi}{\abs{2\sin(\pi z/2)}} - \frac{1}{\abs{z}} \bigg) \, dz = -2\log(\pi/4), 
 \end{equation} 
and adding a regularization $\delta\epsilon$ to the denominator of the integral terms of \eqref{SBT_expr}, where $\delta>0$ is a constant. With respect to the Cartesian basis $\be_x$, $\be_y$, $\be_z$, the $\delta$-regularized slender body approximation about $\mc{C}_\epsilon$ is given by
\begin{equation}\label{KR_reg}
\begin{aligned}
\bu^\SB_{\rm C}(z) &= \mc{L}_\epsilon^\delta[f](z) := \bm{\Lambda}_\delta[\bm{f}](z) + \bm{K}_\delta[\bm{f}](z), \\
\bm{\Lambda}_\delta[\bm{f}](z) &:=\frac{1}{8\pi}\bigg(({\bf I} -3\be_z\be_z^{\rm T}) + 2\log(\delta)({\bf I}+\be_z\be_z^{\rm T})\bigg)\bm{f}(z) \\
\bm{K}_\delta[\bm{f}](z) &= ({\bf I}+\be_z\be_z^{\rm T}) \int_{-1}^1 \frac{\bm{f}(z')}{\sqrt{(z-z')^2+\delta^2\epsilon^2} } \, dz'.
\end{aligned}
\end{equation}
Since $\mc{C}_\epsilon$ is periodic, we consider here the periodization of the integral operator $\bm{K}_\delta$. The form of \eqref{KR_reg} also uses that the second integral term in \eqref{SBT_expr} can be integrated up to $O(\epsilon^2)$ errors to nearly cancel the logarithmic term in $\bm{\Lambda}$, leaving only $\log(\delta)$. The idea is to then choose $\delta$ such that $\mc{L}_\epsilon^\delta$ is better suited for inversion and potentially provides a closer approximation of the slender body PDE operator than \eqref{SBT_expr} and \eqref{gotz2}. In particular, as verified in Lemma \ref{diff_REGS}, $\delta> \sqrt{e}$ is sufficient for removing the invertibility issues of \eqref{SBT_expr} in the case $\Sigma_\epsilon=\mc{C}_\epsilon$, yielding an actual second-kind Fredholm integral equation.    
We may then obtain an error estimate between the $\delta$-regularized slender body approximation and the slender body PDE operator $\mc{L}_\epsilon^{-1}$. We state our main results in Section \ref{main_results}.

\begin{remark}
The formula \eqref{KR_reg} differs from the expression derived from the Method of Regularized Stokeslets in \cite{cortez2012slender} only by the form of the logarithmic term in $\bm{\Lambda}_\delta$, which uses $-\log(\sqrt{\delta^2+1}/\delta)$ in place of $\log\delta$. Here we show that $\log\delta$ yields the correct low wavenumber behavior, indicating that \eqref{KR_reg} has the correct form of $\bm{\Lambda}_\delta$.  
\end{remark}


\subsection{Main results}\label{main_results}
Here we state the main results comparing the slender body PDE operator $\mc{L}_\epsilon^{-1}$ to its approximations $(\mc{L}_\epsilon^\SB)_{NM}^{-1}$ and $(\mc{L}^\delta_\epsilon)^{-1}$. Each of the results presented here has a simpler analogue in the Laplace setting, which we will introduce in Section \ref{sec:SBLap}. We include the Laplace analysis later as a simpler blueprint of the arguments needed to show the following results for the Stokes setting. \\
A crucial component of our analysis is the ability to calculate closed-form expressions for the eigenvalues of $\mc{L}_\epsilon^{-1}$ about $\mc{C}_\epsilon$.  
Recalling that the eigenvalue problem \eqref{eigenval_prob} decouples into the tangential ($\be_z$) and normal ($\be_x$ and $\be_y$) directions, we obtain the following eigenvalue expressions.
\begin{proposition}[Stokes PDE spectrum]\label{Stokes_spec}
Let $(\mc{L}_\epsilon)^{-1}$ be as in \eqref{avg_DTN} for $\Sigma_\epsilon=\mc{C}_\epsilon$. The eigenvalues $\lambda^{\rm t}_k$ and $\lambda^{\rm n}_k$ of $(\mc{L}_\epsilon)^{-1}$ in the tangential and normal directions, respectively, are given by 
\begin{align}
\label{eigsT}
\lambda^{\rm t}_k &= \frac{ 4\pi^2\epsilon \abs{k} K_1^2}{2K_0K_1 + \pi\epsilon \abs{k} \big( K_0^2 - K_1^2 \big) } \\
\label{eigsN}
\lambda^{\rm n}_k &= 
2\pi^2\epsilon \abs{k}\frac{4K_1^2K_2+\pi\epsilon \abs{k} K_1(K_1^2-K_0K_2)}{2K_0K_1K_2 + \pi\epsilon \abs{k} \big(K_1^2(K_0+K_2)-2K_0^2K_2 \big)}
\end{align}
where each $K_j=K_j(\pi\epsilon \abs{k})$, $j=0,1,2$, is a $j^{\text th}$ order modified Bessel function of the second kind. Furthermore, the growth of $\lambda^{\rm t}_k$ and $\lambda^{\rm n}_k$ satisfies the bounds 
\begin{align}
\label{Tgrowth}
4\pi^2\epsilon \abs{k} &< \lambda^{\rm t}_k < 4\pi^2\epsilon \abs{k} + 2\pi \\
\label{Ngrowth}
3\pi^2\epsilon \abs{k} &< \lambda^{\rm n}_k < 3\pi^2\epsilon \abs{k} + 3\pi 
\end{align}
for $\abs{k}=1,2,3,\dots$.
\end{proposition}
The calculation of $\lambda^{\rm t}_k$ and $\lambda^{\rm n}_k$ as well as the proofs of the growth bounds \eqref{Tgrowth} and
\eqref{Ngrowth} appear in Appendix \ref{app_eval}. \\

Now, clearly both $\lambda^{\rm t}_k$ and $\lambda^{\rm n}_k$ strongly disagree with the corresponding eigenvalues $\lambda^{\SB,{\rm t}}_k$, $\lambda^{\SB,{\rm n}}_k$ of $\mc{L}^\SB_\epsilon$ for large $\abs{k}$, since $\lambda^{\rm t}_k$ and $\lambda^{\rm n}_k$ grow linearly in $\abs{k}$ by \eqref{Tgrowth} and \eqref{Ngrowth}, but $\lambda^{\SB,{\rm t}}_k$ and $\lambda^{\SB,{\rm n}}_k$ blow up and jump to negative values as $\abs{k}\to \frac{2}{\pi\epsilon}e^{-\gamma-1/2}$ and $\abs{k}\to \frac{2}{\pi\epsilon}e^{-\gamma+1/2}$, respectively. However, we can make use of the close agreement between the spectra of the PDE spectrum and the slender body approximation at low wavenumbers. \\

In particular, given Dirichlet data with at least $H^1(\T)$ regularity, we can show the following error estimate for the difference between the slender body PDE operator $\mc{L}_\epsilon^{-1}$ and the truncated slender body approximation $(\mc{L}_\epsilon^\SB)^{-1}_{NM}$.

\begin{theorem}[Truncated Stokes approximation error estimate]\label{thm:errS}
Let $\mc{L}_\epsilon^{-1}$ be as in \eqref{avg_DTN} for the Stokes slender body PDE on $\p\mc{C}_\epsilon$, and let $(\mc{L}_\epsilon^\SB)^{-1}_{NM}$ be the truncated slender body approximation defined in \eqref{LambdaNM} with $N=C_1/\epsilon$, $C_1\le \frac{1}{4\pi}$, and $M=C_2/\epsilon$, $C_2\le \frac{73}{100 \pi}$. Then for any slender body velocity $\bu \in H^1(\T)$, we have
\begin{equation}
\norm{\mc{L}_\epsilon^{-1} [\bu] - (\mc{L}_\epsilon^\SB)^{-1}_{NM}[\bu] }_{L^2(\T)} \le C\epsilon \norm{\bu}_{H^1(\T)}.
\end{equation}
Furthermore, if $\bu\in H^2(\T)$, we have 
\begin{equation}
\norm{\mc{L}_\epsilon^{-1} [\bu] - (\mc{L}_\epsilon^\SB)^{-1}_{NM}[\bu] }_{L^2(\T)} \le C\epsilon^2 \norm{\bu}_{H^2(\T)}.
\end{equation}
\end{theorem}

For the $\delta$-regularized slender body approximation $(\mc{L}_\epsilon^\delta)^{-1}$, a similar error estimate holds. 
\begin{theorem}[$\delta$-regularized Stokes approximation error estimate]\label{thm:errREGS}
Let $\mc{L}_\epsilon^{-1}$ be as in \eqref{avg_DTN} for the slender body Stokes PDE on $\p\mc{C}_\epsilon$, and let $(\mc{L}_\epsilon^\delta)^{-1}$ be the $\delta$-regularized slender body Stokes approximation \eqref{KR_reg}. Then, given $\delta>\sqrt{e}$, for any fiber velocity $\bu\in H^1(\T)$ we have 
\begin{equation}\label{L2estREG_S}
\norm{\mc{L}_\epsilon^{-1} [\bu] - (\mc{L}_\epsilon^\delta)^{-1} [\bu]}_{L^2(\T)} \le C_\delta \, \epsilon \norm{\bu}_{H^1(\T)}.
\end{equation}
Furthermore, if $\bu\in H^2(\T)$, we obtain
\begin{equation}\label{L2estREG2_S}
\norm{\mc{L}_\epsilon^{-1} [\bu] - (\mc{L}_\epsilon^\delta)^{-1} [\bu]}_{L^2(\T)} \le C_\delta \, \epsilon^2 \norm{\bu}_{H^2(\T)}.
\end{equation}
Here $C_\delta=C_1\delta^2(1+\log\delta) + \frac{C_2}{-1+2\log\delta}$ for $C_1$, $C_2$ independent of $\delta$.
\end{theorem}
Note that we obtain the same order of convergence as $\epsilon \to 0$ as in Theorem \ref{thm:errS}, but Theorem \ref{thm:errREGS} holds without restricting the wavenumber $k$. In addition, the estimates \eqref{L2estREG_S} and \eqref{L2estREG2_S} provide some guidance in choosing the regularization parameter $\delta$. In particular, there exists an optimal regularization $\sqrt{e}< \delta <\infty$ which minimizes the constant $C_\delta$, given by $\delta^2(-1+2\log\delta)^2(\frac{3}{2}+\log\delta)=C_2/C_1$. For $0.05 \le C_2/C_1 \le 10$, this yields an optimal $\delta$ in the range $1.72 \le \delta \le 2.5$. Tracking the value of constants in the proof of Theorem \ref{thm:errREGS} gives $C_2/C_1\approx 0.1$ in the $H^1$ estimate and $C_2/C_1\approx 0.085$ in the $H^2$ estimate, but again we make no claims of optimality for these $C_1$ and $C_2$.  \\

Furthermore, as an application of Theorem \ref{thm:errS}, we can show the $\epsilon$-dependence in a well-posedness estimate for the mapping $\mc{L}_\epsilon^{-1}: H^1(\T)\to L^2(\T)$. 
\begin{theorem}[Stokes well-posedness estimate]\label{thm:wellpoS}
Let $\mc{L}_\epsilon^{-1}$ be as in \eqref{avg_DTN} for the Stokes slender body PDE on $\p\mc{C}_\epsilon$. For $\bu\in H^1(\T)$, we have that $\mc{L}_\epsilon^{-1} [\bu]$ satisfies the bound
\begin{equation}\label{wellpoS}
\norm{\mc{L}_\epsilon^{-1} [\bu]}_{L^2(\T)} \le \frac{C}{\abs{\log\epsilon}}\norm{\bu}_{H^1(\T)}. 
\end{equation}
\end{theorem}
This $\abs{\log\epsilon}^{-1}$ dependence is consistent with the $\abs{\log\epsilon}$ dependence obtained in the well-posedness estimate \eqref{stokes_est} for the forward map $\mc{L}_\epsilon$. \\

In order to prove Theorem \ref{thm:errS}, we will rely on the following two lemmas bounding the difference between eigenvalues of the slender body PDE operator $\mc{L}_\epsilon^{-1}$ with the eigenvalues of the slender body approximation $(\mc{L}_\epsilon^\SB)^{-1}$ at low wavenumber.

%
\begin{lemma}[Difference in tangential Stokes eigenvalues]\label{tan_eig_diff}
Let $\lambda^{\rm t}_k$ denote the tangential eigenvalues of the slender body PDE operator $\mc{L}_\epsilon^{-1}$ \eqref{avg_DTN}, given by \eqref{eigsT}, and let $\lambda^{\SB,{\rm t}}_k$ denote the tangential eigenvalues of the slender body approximation $(\mc{L}_\epsilon^\SB)^{-1}$ \eqref{gotz2}, given by \eqref{KR_eigs_St}.   
For $\abs{k}\le \frac{1}{4\pi\epsilon}$, the difference $\lambda^{\rm t}_k - \lambda^{\SB,{\rm t}}_k$ satisfies
\begin{equation}\label{tan_diff}
\abs{\lambda^{\rm t}_k - \lambda^{\SB,{\rm t}}_k} \le C \epsilon^2 k^2.
\end{equation}
\end{lemma}

\begin{lemma}[Difference in normal Stokes eigenvalues]\label{nor_eig_diff}
Let $\lambda^{\rm n}_k$ denote the normal direction eigenvalues of the slender body PDE operator $\mc{L}_\epsilon^{-1}$ \eqref{avg_DTN}, given by \eqref{eigsN}, and let $\lambda^{\SB,{\rm n}}_k$ denote the normal direction eigenvalues of the slender body approximation $(\mc{L}_\epsilon^\SB)^{-1}$ \eqref{gotz2}, given by \eqref{KR_eigs_Sn}. 
For $\abs{k}\le \frac{73}{100\pi\epsilon}$, the difference $\lambda^{\rm n}_k - \lambda^{\SB,{\rm n}}_k$ satisfies
\begin{equation}\label{nor_diff}
\abs{\lambda^{\rm n}_k - \lambda^{\SB,{\rm n}}_k} \le C \epsilon^2 k^2.
\end{equation}
\end{lemma}
Note that the upper bound for $\abs{k}$ is much larger in the normal direction, which is the reason for the different cutoff values $N$ and $M$ in Theorem \ref{thm:errS}. The proofs of Lemmas \ref{tan_eig_diff} and \ref{nor_eig_diff} appear in Sections \ref{sec:tan_diff} and \ref{sec:nor_diff}, respectively. The proofs of Theorems \ref{thm:errS} and \ref{thm:wellpoS} follow in Section \ref{Stokes_err}.\\

Similarly, to show Theorem \ref{thm:errREGS} for the $\delta$-regularized Stokes approximation \eqref{KR_reg}, we will need the following lemma.
\begin{lemma}[Difference in $\delta$-regularized Stokes eigenvalues]\label{diff_REGS}
Let $\delta>\sqrt{e}$. The eigenvalues $\lambda^{\delta,{\rm t}}_k$, $\lambda^{\delta,{\rm n}}_k$ of the $\delta$-regularized Stokes slender body approximation $(\mc{L}_\epsilon^\delta)^{-1}$ in the tangential and normal directions, respectively, are given by
\begin{align}
\label{lamREG_tS}
\lambda^{\delta,{\rm t}}_k &= \frac{4\pi}{ -1+ 2\log \delta + 2 K_0(\delta \pi \epsilon \abs{k})} , \\ 
\label{lamREG_nS}
\lambda^{\delta,{\rm n}}_k &= \frac{8\pi}{1+ 2\log\delta + 2K_0(\delta \pi\epsilon \abs{k})}, \qquad \abs{k}=1,2,3,\dots.
 \end{align}
For each $k$, $\lambda^{\delta,{\rm t}}_k$ and $\lambda^{\delta,{\rm n}}_k$ satisfy
\begin{align}
\label{REG_diff0}
\abs{\lambda_k^{\rm t} - \lambda^{\delta,{\rm t}}_k} \le 4\pi\bigg(\frac{1}{2}+ \frac{1}{-1+2\log\delta}+ \pi\epsilon \abs{k}\bigg), \\
\label{REG_diff00}
\abs{\lambda_k^{\rm n} - \lambda^{\delta,{\rm n}}_k} \le 3\pi\bigg(1+ \frac{8}{3(1+2\log\delta)}+ \pi\epsilon \abs{k}\bigg),
\end{align}
where $\lambda_k^{\rm t}$, $\lambda_k^{\rm n}$ are the tangential and normal eigenvalues \eqref{eigsT} and \eqref{eigsN}, respectively, of the Stokes slender body PDE. Furthermore, for $\abs{k}\le \frac{1}{4\pi\epsilon}$, $\lambda^{\delta,{\rm t}}_k$ satisfies the refined bound 
\begin{equation}\label{REG_diff1}
\abs{\lambda_k^{\rm t} - \lambda^{\delta,{\rm t}}_k} \le C\delta^2(1+\log\delta)\epsilon^2k^2,
\end{equation}
while for $k\le \frac{2}{3\pi\epsilon}$, $\lambda^{\delta,{\rm n}}_k$ satisfies 
\begin{equation}\label{REG_diff2}
\abs{\lambda_k^{\rm n} - \lambda^{\delta,{\rm n}}_k} \le C\delta^2(1+\log\delta)\epsilon^2k^2.
\end{equation}
\end{lemma}
Note that $\lambda^{\rm t}_k$ is guaranteed to be positive and bounded as $\abs{k}\to\infty$ as long as $\delta> \sqrt{e}$, which verifies the value reported by Tornberg and Shelley \cite{tornberg2004simulating} as being sufficient for positivity. Furthermore, the low wavenumber behavior of $\lambda^{\rm t}_k$ and $\lambda^{\rm n}_k$ determines the rate of convergence in Theorem \ref{thm:errREGS}, but the high wavenumber bound \eqref{REG_diff1} affects the $\delta$-dependence of $C_\delta$. The optimal choice of regularization parameter $\delta$ has to balance the differing high and low wavenumber behavior. 
The proofs of Lemma \ref{diff_REGS} and Theorem \ref{thm:errREGS} appear in Sections \ref{sec:REGS} and \ref{sec:REGSerr}, respectively. \\


Finally, we point out some implications of the above spectral results for a dynamic filament problem. 
Consider a dynamic thin closed filament that moves through a Stokes fluid and assume that the filament resists bending and is inextensible.
This problem has been considered in \cite{tornberg2004simulating} for open filaments using the slender body approximation discussed above.
The formulation of this problem using the slender body PDE is as follows. 
Let $\bm{X}(s,t)$ be the center line position of this filament. The equations that govern the movement of the filament are given by
\begin{equation}\label{SB_dynamic_PDE}
\begin{aligned}
-\Delta \bu +\nabla p &= 0, \quad \dive \, \bu = 0 \hspace{1.8cm} \text{ in } \Omega_{\epsilon,t} = \R^3\backslash \overline{\Sigma_{\epsilon,t}} \\
\int_0^{2\pi} (\bm{\sigma}\bm{n}) \mc{J}_\epsilon(s,\theta) \, d\theta &= -\frac{\p^4\bm{X}}{\p s^4}+\PD{}{s}\paren{\tau\PD{\bm{X}}{s}} \quad \text{ on } \p\Sigma_{\epsilon,t} \\
\bu\big|_{\p\Sigma_\epsilon} &= \bu(s), \hspace{3.2cm} \text{ unknown but independent of }\theta \\
\abs{\bu} \to 0 & \text{ as } \abs{\bx} \to \infty,\\
\PD{\bm{X}}{t}&=\bm{u}(s), \; \abs{\PD{\bm{X}}{s}}=1.
\end{aligned}
\end{equation}
In the above, $\tau$ is the tension that enforces the inextensibility condition $\abs{\p \bm{X}/\p s}=1$, which must be solved as part of the problem.
We have added a subscript $t$ to $\Sigma_{\epsilon}$ and $\Omega_\epsilon$ to emphasize that these domains are now time dependent.
We note that this problem is physically natural in the sense that it satisfies the following energy dissipation identity:
\begin{equation}
\frac{d}{dt}\int_{\T} \frac{1}{2}\abs{\PDD{2}{\bm{X}}{s}}^2ds=-\int_{\Omega_{\epsilon,t}} 2\, \mc{E}(\bm{u}):\mc{E}(\bm{u}) \,d\bm{x}.
\end{equation}
Let us consider the linearization around the steady state $\bm{X}=\bm{X}_\star(s)=(\cos(s),\sin(s),0)^{\rm T}$. At this steady state, $\bm{u}=0$
and the tension is given by $\tau=1$. Substituting $\bm{X}=\bm{X}_\star+\bm{Y}, \tau=1+\wt{\tau}$ into \eqref{SB_dynamic_PDE} and retaining 
terms that are linear in $\bm{Y}$ and $\wt{\tau}$, we obtain
\begin{equation}\label{SB_dynamic_PDE_linear}
\begin{aligned}
-\Delta \bu +\nabla p &= 0, \quad \dive \, \bu = 0 \hspace{2.3cm} \text{ in } \Omega_{\epsilon,\star} = \R^3\backslash \overline{\Sigma_{\epsilon,\star}} \\
\int_0^{2\pi} (\bm{\sigma}\bm{n}) \mc{J}_\epsilon(s,\theta) \, d\theta &= -\frac{\p^4\bm{Y}}{\p s^4}+\PDD{2}{\bm{Y}}{s}+\PD{\wt{\tau}}{s}\PD{\bm{X}_\star}{s} \quad \text{ on } \p\Sigma_{\epsilon,\star} \\
\bu\big|_{\p\Sigma_{\epsilon,\star}} &= \bu(s), \hspace{3.85cm} \text{unknown but independent of }\theta \\
\abs{\bu} \to 0 & \text{ as } \abs{\bx} \to \infty,\\
\PD{\bm{Y}}{t}&=\bm{u}(s), \; \PD{\bm{X_\star}}{s}\cdot \PD{\bm{Y}}{s}=0.
\end{aligned}
\end{equation}
In the above, $\Sigma_{\epsilon, \star}$ denotes the region $\Sigma_{\epsilon}$ with $\bm{X}=\bm{X}_\star$, and likewise for $\Omega_{\epsilon,\star}$.
Let $\mc{L}_{\epsilon,\star}$ be the operator $\mc{L}_\epsilon$ for the geometry $\bm{X}=\bm{X}_\star$. The above equation can be written as
\begin{equation}\label{Ycirceqn}
\PD{\bm{Y}}{t}=\mc{L}_{\epsilon,\star} \paren{-\frac{\p^4\bm{Y}}{\p s^4}+\PDD{2}{\bm{Y}}{s}+\PD{\wt{\tau}}{s}\PD{\bm{X}_\star}{s}}, \; \PD{\bm{X_\star}}{s}\cdot \PD{\bm{Y}}{s}=0.
\end{equation}
The analysis of this linearized equation is beyond the scope of this manuscript. However, our spectral results allow for a complete analysis of the following analogue of this problem, in which we replace $\bm{X}_\star$ with the periodic straight line filament considered in this manuscript. In this case, the above linear equation reduces to
\begin{equation}\label{Ystreqn}
\PD{\bm{Y}}{t}=\mc{L}_{\epsilon} \paren{-\frac{\p^4\bm{Y}}{\p s^4}+\PDD{2}{\bm{Y}}{s}+\PD{\wt{\tau}}{s}\bm{e}_z}, \quad \bm{e}_z\cdot \PD{\bm{Y}}{s}=0, 
\; \int_\T \bm{Y}(s)ds=0, \; \int_\T \wt{\tau} ds=0.
\end{equation}
The last two conditions have been added to remove translation modes and to allow for unique determination of $\wt{\tau}$.
It is immediate from Proposition \ref{Stokes_spec} that the eigenvectors and eigenvalues of the dynamics are given by
\begin{equation}
\text{eigenvalues: } \nu_k=\frac{-k^4+k^2}{\lambda_k^{\rm n}}, \; \text{eigenvectors: } e^{i\pi k z} \bm{e}_{x,y}.
\end{equation} 
Given Proposition \ref{Stokes_spec} we see that $\nu_k$ behaves like $-k^4\abs{\log(\epsilon \abs{k})}$ when $1\ll \abs{k}\ll 1/\epsilon$ and $-\abs{k}^3/\epsilon$ when $\abs{k}\ll 1/\epsilon$.
It is expected that the spectrum of \eqref{Ycirceqn} and \eqref{Ystreqn} are similar when $\abs{k}\gg 1$.
More generally, suppose $\bm{X}$ is a sufficiently smooth curve whose curvature radius is order $1$ with respect to $\epsilon$. 
The linearization of \eqref{SB_dynamic_PDE} around such a configuration $\bm{X}$ is expected to have similar spectral properties.\\

The {\em small scale decomposition}, first introduced in \cite{beale1993growth,hou1994removing} for the Hele-Shaw and water wave problems, has been crucial for the analysis and scientific computing of various interfacial fluid and fluid structure interaction problems. The above considerations suggest that a suitable small scale decomposition for problem \eqref{SB_dynamic_PDE} will result in a principal evolution operator which behaves like a fourth order diffusion for wave numbers less than $1/\epsilon$, and a third order diffusion with diffusion coefficient $1/\epsilon$ for wave numbers greater than $1/\epsilon$. The parabolic nature of the principal operator implies that $\bm{X}$ will be smooth for $t>0$. We also see that an explicit time-stepping scheme will be prohibitively expensive, which led the authors of \cite{tornberg2004simulating} to develop an implicit scheme. Indeed, if $\Delta s$ is the spatial discretization of $s$ and $\Delta t$ is the time discretization, a stable explicit time marching scheme will require $\Delta t \lesssim (\Delta s)^4$ if $\Delta s \gtrsim \epsilon$ and $\Delta t \lesssim \epsilon(\Delta s)^3$ if $\Delta s \lesssim \epsilon$. \\


The remainder of this paper is structured as follows. Before we prove any of the results stated above for the Stokes setting, we devote Section \ref{sec:SBLap} to a similar analysis for the analogous Laplace slender body PDE. The proofs of the analogous Laplace results are much neater and provide intuition for the Stokes analysis. Sections \ref{sec:lap_calc} -- \ref{sec:REGerr} contain proofs of the analogous results of Section \ref{main_results} for the Laplace setting. As a preliminary step, in Section \ref{bessel_bds} we derive a novel bound relating to modified Bessel functions which will be used throughout the paper. We also recall some Bessel function bounds valid for small argument. 
Section \ref{sec:stokes} then concerns the Stokes setting. In Sections \ref{sec:tan_diff} and \ref{sec:nor_diff}, we prove Lemmas \ref{tan_eig_diff} and \ref{nor_eig_diff}, respectively; in Section \ref{Stokes_err} we prove Theorems \ref{thm:errS} and \ref{thm:wellpoS}, and in Sections \ref{sec:REGS} and \ref{sec:REGSerr} we show Lemma \ref{diff_REGS} and Theorem \ref{thm:errREGS}, respectively. We recall the derivation of the slender body approximation spectrum (Proposition \ref{KRop_eigs}) in Appendix \ref{sec:gotz}, while the form of the Stokes eigenvalues \eqref{eigsT} and \eqref{eigsN} are calculated in Appendix \ref{app_eval}. 

\section{Slender body Laplace PDE}\label{sec:SBLap}
Before we prove the main results of Section \ref{main_results} for the Stokes setting, we consider the Laplace analogue to slender body theory. Here the proofs are much neater and provide a simpler proof-of-concept to inform the Stokes analysis. 
Indeed, much of the existing spectral analysis related to slender body theory (see \cite{gotz2000interactions}) has been carried out for what is essentially the slender body approximation to the Laplace equation about a fiber with a straight centerline. \\

The Laplace analogue to the operator $\mc{L}_\epsilon^{-1}$ maps a scalar-valued $\overline u\in H^{1/2}(\T)$ to a scalar-valued density $f\in H^{-1/2}(\T)$ along the slender body. Given Dirichlet data $\overline u\in H^{1/2}(\T)$ along $\p\Sigma_\epsilon$, we may define the map $\mc{L}_\epsilon^{-1}$ analogously to \eqref{avg_DTN}:
\begin{equation}\label{lap_map}
\mc{L}_\epsilon^{-1} : H^{1/2}(\T) \to H^{-1/2}(\T); \quad \mc{L}_\epsilon^{-1} [\overline u](s) = \int_0^{2\pi} (\nabla u\cdot \bm{n})\, \mc{J}_\epsilon(s,\theta)\, d\theta =: f(s),
\end{equation}
where $u\in H^1(\Omega_\epsilon)$ (weakly) satisfies
\[ \Delta u=0 \quad \text{ in } \Omega_\epsilon; \qquad u\big|_{\p\Sigma_\epsilon} = \overline u(s).\]

In general, the eigenvalues $\lambda^{\rm l}$ of $\mc{L}_\epsilon^{-1}$ in the Laplace setting are given by 
\begin{equation}\label{SBlap_BVP}
\begin{aligned}
\Delta u &= 0 \hspace{1.5cm} \text{ in } \Omega_{\epsilon} \\
 \int_0^{2\pi} (\nabla u \cdot\bm{n})  \, \mathcal{J}_{\epsilon}(s,\theta) \, d\theta &= \lambda^{\rm l} u(s) \qquad \text{ on }\Gamma_\epsilon \\
u \to 0 \, &\text{ as } \abs{\bx} \to \infty.
\end{aligned}
\end{equation}
The spectral properties of $\mc{L}_\epsilon^{-1}$ in the Laplace setting are nearly identical to the Stokes case. Since all the Laplace analysis is contained within this section, we use the same notation without confusion.  \\

As in the Stokes setting, the eigenfunctions when $\Sigma_\epsilon=\mc{C}_\epsilon$ are simply the exponentials $e^{i\pi k z}$, $\abs{k}=1,2,3,\dots$, and we may calculate a simple closed-form expression for the eigenvalues $\lambda_k^{\rm l}$. 
\begin{proposition}[Laplace PDE spectrum]\label{Laplace_spec}
The eigenvalues of the Laplace slender body PDE operator $(\mc{L}_\epsilon)^{-1}$ \eqref{lap_map} along the straight, periodic cylinder $\mc{C}_\epsilon$ are given by
\begin{equation}\label{true_evals}
\lambda_k^{\rm l} = 2\pi^2 \epsilon \abs{k}\frac{K_1(\pi\epsilon \abs{k})}{K_0(\pi\epsilon \abs{k})}, \quad \abs{k}=1,2,3,\dots .
\end{equation}
where each $K_0$ and $K_1$ are, respectively, zeroth and first order modified Bessel functions of the second kind. Furthermore, the eigenvalues $\lambda_k^{\rm l}$ satisfy the growth rate
\begin{equation}\label{linear}
2\pi^2\epsilon \abs{k} \le 2\pi^2\epsilon \abs{k}\frac{\sqrt{(\pi\epsilon k)^2+\pi\epsilon \abs{k} + 1}+1}{\pi\epsilon \abs{k}+1} \le \lambda_k^{\rm l} \le 2\pi^2\epsilon \abs{k} + \pi, \qquad \abs{k}=1,2,3,\dots.
\end{equation}
\end{proposition}
The derivation of the eigenvalue expression \eqref{true_evals} is considerably simpler than in the Stokes setting, and is contained in Section \ref{sec:lap_calc}. \\

Now, in the Laplace analogue to slender body theory, we approximate the scalar-valued harmonic potential $u(\bx)$ outside of $\Sigma_\epsilon$ by the expression
\begin{equation}\label{SB_laplace}
u^{\SB}(\bx) = \frac{1}{4\pi} \int_{-1}^1 \frac{f(s')}{\abs{\bm{R}}} ds'.
\end{equation}
Here, as in the Stokes setting, the Green's function for the Laplace equation is integrated along the slender body centerline. Note that no Laplace analogue to the Stokes doublet correction is necessary since \eqref{SB_laplace} already satisfies the fiber integrity condition to leading order in $\epsilon$. \\

The analogous expression to \eqref{SBT_expr} for the potential along the centerline of a periodic fiber is given by 
\begin{equation}\label{gotz2}
4\pi u^\SB_{\rm C}(s) = \mc{L}_\epsilon^\SB[f](s) := -2\log (\pi\epsilon/8) f(s) + \frac{\pi}{2}\int_{\T} \frac{f(s') - f(s)}{\abs{\sin\big(\pi(s-s')/2\big)}} \, ds'.
\end{equation}

The eigenvalues of $\mc{L}_\epsilon^\SB$ along $\mc{C}_\epsilon$ in the Laplace setting are essentially calculated in \cite{gotz2000interactions}. We reiterate this calculation in Appendix \ref{sec:gotz}. 
\begin{proposition}\label{KRop_eigsL}
Consider the Laplace slender body approximation $\mc{L}^\SB_\epsilon$ \eqref{gotz2} along $\mc{C}_\epsilon$. The eigenvalues of $\mc{L}^\SB_\epsilon$ are given by
\begin{equation}\label{KR_eigs_L}
\frac{1}{\lambda_k^{\SB,{\rm l}}} = -\frac{1}{2\pi} \big(\log(\pi \epsilon \abs{k}/2) + \gamma \big). 
\end{equation}
Here $\gamma\approx 0.5772$ is the Euler gamma.
\end{proposition}
Again, $\lambda_k^{\rm l}$ and $\lambda^{\SB,{\rm l}}_k$ deviate wildly as $\abs{k}$ crosses $\frac{2e^{-\gamma}}{\pi\epsilon}$, since the eigenvalues $\lambda^{\SB,{\rm l}}_k$ blow up and jump to negative values while each $\lambda_k^{\rm l}$ satisfies the growth estimate \eqref{linear}. However, the two expressions are nearly identical at low wavenumbers (see Figure \ref{fig:Compare1}). \\

We consider the same two regularizations for $\mc{L}^\SB_\epsilon$ as in the Stokes setting. We first define $(\mc{L}_\epsilon^\SB)^{-1}_N$ to be the $N$-term partial sum of the Fourier series for $(\mc{L}_\epsilon^\SB)^{-1}$ in the Laplace setting:
\begin{equation}\label{LambdaN}
(\mc{L}_\epsilon^\SB)^{-1}_N [v](z) := \sum_{\abs{k}=1}^N \lambda^{\SB,{\rm l}}_k \wh v_k e^{i \pi k z}.
\end{equation}

We also consider the Laplace version of the $\delta$-regularized slender body approximation, given by the periodization of
\begin{equation}\label{reg_lap1}
 u^\SB_{\rm C}(z) = \mc{L}_\epsilon^\delta[f](z) := \frac{1}{4\pi} \bigg(2\log(\delta)f(z) + \int_{-1}^1 \frac{f(z')}{\sqrt{(z-z')^2+\delta^2\epsilon^2}} \, dz' \bigg). 
 \end{equation}
Here $\delta>1$ is the analogous threshold needed to ensure that \eqref{reg_lap1} yields a second kind integral equation for $f$. \\

Given $u\in H^1(\T)$, we show the following error estimate for the difference between the slender body Laplace PDE solution $\mc{L}_\epsilon^{-1}[u]$ and the truncated slender body approximation $(\mc{L}_\epsilon^\SB)^{-1}_N [u]$. 
\begin{theorem}[Truncated Laplace approximation error estimate]\label{thm:err}
Let $\mc{L}_\epsilon^{-1}$ be as in \eqref{lap_map} for the Laplace slender body PDE on $\p\mc{C}_\epsilon$, and let $(\mc{L}_\epsilon^\SB)^{-1}_N$ be as in \eqref{LambdaN} for $N=C_0/\epsilon$, $C_0\le 9/(20\pi)$. Then for any $u\in H^1(\T)$, we have
\begin{equation}\label{L2est}
\norm{\mc{L}_\epsilon^{-1} [u] - (\mc{L}_\epsilon^\SB)^{-1}_N [u]}_{L^2(\T)} \le C\epsilon \norm{u}_{H^1(\T)}.
\end{equation}
Furthermore, if $u\in H^2(\T)$, we obtain 
\begin{equation}\label{L2est2}
\norm{\mc{L}_\epsilon^{-1} [u] - (\mc{L}_\epsilon^\SB)^{-1}_N [u]}_{L^2(\T)} \le C\epsilon^2 \norm{u}_{H^2(\T)}.
\end{equation}
\end{theorem}
The proof of Theorem \ref{thm:err} is given in Section \ref{lap_err}. Using Theorem \ref{thm:err}, we may again show the $\epsilon$-dependence in the following well-posedness estimate for the slender body Laplace PDE: 
\begin{theorem}[Laplace well-posedness estimate]\label{thm:wellpo}
Let $\mc{L}_\epsilon^{-1}$ be as in \eqref{lap_map} for the Laplace slender body PDE on $\p\mc{C}_\epsilon$. For $u\in H^1(\T)$, we have that $\mc{L}_\epsilon^{-1} [u]$ satisfies the bound
\begin{equation}\label{wellpo}
\norm{\mc{L}_\epsilon^{-1} [u]}_{L^2(\T)} \le \frac{C}{\abs{\log\epsilon}}\norm{u}_{H^1(\T)}. 
\end{equation}
\end{theorem} 

As in the Stokes setting, a similar error estimate to Theorem \ref{thm:err} holds for the $\delta$-regularized slender body Laplace approximation $(\mc{L}_\epsilon^\delta)^{-1}$.  
\begin{theorem}[$\delta$-regularized Laplace approximation error estimate]\label{thm:errREG}
Let $\mc{L}_\epsilon^{-1}$ be as in \eqref{lap_map} for the slender body Laplace PDE on $\p\mc{C}_\epsilon$, and let $(\mc{L}_\epsilon^\delta)^{-1}$ be the $\delta$-regularized slender body Laplace approximation \eqref{reg_lap1}. Given $\delta>1$, for any $u\in H^1(\T)$ we have 
\begin{equation}\label{L2estREG}
\norm{\mc{L}_\epsilon^{-1} [u] - (\mc{L}_\epsilon^\delta)^{-1} [u]}_{L^2(\T)} \le C_\delta\, \epsilon \norm{u}_{H^1(\T)}.
\end{equation}
Furthermore, if $u\in H^2(\T)$, we have
\begin{equation}\label{L2estREG2}
\norm{\mc{L}_\epsilon^{-1} [u] - (\mc{L}_\epsilon^\delta)^{-1} [u]}_{L^2(\T)} \le C_\delta\, \epsilon^2 \norm{u}_{H^2(\T)}.
\end{equation}
Here $C_\delta= C_1\delta^2(1+\log\delta) + \frac{C_2}{\log\delta}$ for $C_1,C_2$ independent of $\delta$.
\end{theorem}
Again, Theorem \ref{thm:errREG} gives the same order of convergence as $\epsilon \to 0$ as Theorem \ref{thm:err}, but holds for all wavenumbers. In addition, there exists an optimal $\delta$, $1< \delta <\infty$, which minimizes the constant $C_\delta$, determined by $\delta^2\log(\delta)^2(3+2\log\delta) = C_2/C_1$. For $0.1 \le C_2/C_1 \le 10$, the optimal $\delta$ lies in the interval $1.1 \le \delta\le 2.1$. Carrying through with the constants in the proof of Theorem \ref{thm:errREG} yields $C_2/C_1\approx 2$ in the $H^1$ estimate and $C_2/C_1\approx 4.5$ in the $H^2$ estimate, but we make no claims of optimality regarding $C_1$ and $C_2$. The proof of Theorem \ref{thm:errREG} is given in Section \ref{sec:REGerr}.  \\

To prove Theorem \ref{thm:err} in the Laplace setting, we will need the following bound for the difference $\abs{\lambda_k^{\rm l} - \lambda^{\SB,{\rm l}}_k}$ for small $\abs{k}$. 
\begin{lemma}[Difference in Laplace eigenvalues]\label{lam_lamSB}
Let $\lambda_k^{\rm l}$ be as in \eqref{true_evals} and $\lambda^{\SB,{\rm l}}_k$ as in \eqref{SBT_spec_per}. For $\abs{k}\le \frac{9}{20\pi\epsilon}$, we have
\begin{equation}
\abs{\lambda_k^{\rm l} - \lambda^{\SB,{\rm l}}_k} \le C\epsilon^2k^2.
\end{equation}
\end{lemma}
The proof of this lemma is much simpler than in the Stokes setting and appears in Section \ref{lap_diff}.  \\

Similarly, the proof of Theorem \ref{thm:errREG} for the $\delta$-regularized Laplace approximation will rely on the following eigenvalue bound, which holds for all $k$. 
\begin{lemma}[Difference in $\delta$-regularized Laplace eigenvalues]\label{lam_lamREG}
Let $\delta>1$. The eigenvalues $\lambda^{\delta,{\rm l}}_k$ of the $\delta$-regularized Laplace slender body approximation $(\mc{L}_\epsilon^\delta)^{-1}$ are given by
\begin{equation}\label{lambdaREG}
\lambda^{\delta,{\rm l}}_k = \frac{2\pi}{\log\delta + K_0(\delta\pi\epsilon \abs{k})} , \quad \abs{k}=1,2,3,\dots. 
 \end{equation}
For each $k$, $\lambda^{\delta,{\rm l}}_k$ satisfies
\begin{equation}\label{Lreg_diff0}
\abs{\lambda_k^{\rm l} - \lambda^{\delta,{\rm l}}_k} \le 2\pi\bigg(\frac{1}{2}+\frac{1}{\log\delta}+\pi \epsilon \abs{k} \bigg),
\end{equation}
where $\lambda_k^{\rm l}$ are the eigenvalues \eqref{true_evals} of the Laplace slender body PDE. In addition, for $\abs{k}\le\frac{2}{5\pi\epsilon}$, $\lambda^{\delta,{\rm l}}_k$ satisfies the refined bound 
\begin{equation}\label{Lreg_diff1}
\abs{\lambda_k^{\rm l} - \lambda^{\delta,{\rm l}}_k} \le 82\pi^3 \delta^2(1+\log\delta)\epsilon^2 k^2.
\end{equation}
\end{lemma}
Again, the proof of Lemma \ref{lam_lamREG} is simpler than the analogous Stokes case, and appears in Section \ref{sec:REG}.
Note that $\delta>1$ ensures that $\lambda^{\delta,{\rm l}}_k$ are positive and bounded as $\abs{k}\to\infty$, yielding a true second kind integral equation \eqref{reg_lap1} for $f(z)$. 


\subsection{Preliminaries: Bessel function bounds}\label{bessel_bds}
Due to the form of the eigenvalues of the slender body PDE operator $\mc{L}_\epsilon^{-1}$ (see Propositions \ref{Laplace_spec} and \ref{Stokes_spec}), the proofs of Lemmas \ref{lam_lamSB}, \ref{tan_eig_diff}, and \ref{nor_eig_diff} rely on refined upper and lower bounds for the growth rate of the following ratio of modified Bessel functions. Although the bound in Lemma \ref{lem:bessel} is similar to the inequalities derived in \cite{baricz2013turan, gronwall1932inequality}, we were not able to find the results necessary for our purposes in the literature. Here we prove a novel version of these Bessel function inequalities which will be used throughout the paper. 

\begin{lemma}[Ratio of modified Bessel functions]\label{lem:bessel}
Let $K_0(z)$ and $K_1(z)$ denote second-kind modified Bessel functions of zeroth and first order, respectively. For $z>0$, the ratio $K_1/K_0$ satisfies the bounds
\begin{equation}\label{bessel}
\frac{\sqrt{z^2+z+1}+1}{z+1} < \frac{K_1(z)}{K_0(z)} < 1 + \frac{1}{2z}.
\end{equation}
\end{lemma}

\begin{proof}
We begin by considering the function
\begin{equation}\label{Bdef}
B(z) = z \frac{K_1(z)}{K_0(z)} = -z \frac{K_0'(z)}{K_0(z)}
\end{equation}
for $z\in \R_+$, which satisfies the ODE
\begin{equation}\label{BODE}
\frac{dB}{dz} = \frac{1}{z}(B^2- z^2), \; B(0)=0.
\end{equation}

Furthermore, by standard properties of the modified Bessel functions $K_0$ and $K_1$, we have that $B$ satisfies 
\begin{equation}\label{taylorB}
B(z) = z + \frac{1}{2} - \frac{1}{8z} + \frac{1}{8z^2} + \cdots \;  \text{ as }z\to \infty.
\end{equation}

Now, suppose that some function $g(z)$ satisfies
\begin{equation}\label{Bcond1}
  \left\{ \begin{array}{ll}
    \frac{dg}{dz} > \frac{1}{z}(g^2-z^2) \quad \text{ for } z>0, \\
    \exists z_0>0 \text{ such that } g(z) < B(z) \text{ for } z\ge z_0.
    \end{array}  \right. 
 \end{equation}
Then $g(z)<B(z)$ for all $z>0$, which can be seen by considering solutions $y(z)$ to the ODE 
\[ \frac{dy}{dz} = -\frac{1}{z}(y^2 - z^2).  \]
Note that, evaluating $dy/dz$ at the curve $y=g(z)$, we have
\[ \frac{dy}{dz}\bigg|_{y=g(z)} = - \frac{1}{z}(g^2-z^2) > -\frac{dg}{dz} \] 
for every $z>0$. Thus if $y>-g$ for some $z>0$, then $y>-g$ for every $z>0$. This is exactly the second condition of \eqref{Bcond1}. \\

Likewise, if 
\begin{equation}\label{Bcond2}
  \left\{ \begin{array}{ll}
    \frac{dg}{dz} < \frac{1}{z}(g^2-z^2) \quad \text{ for } z>0, \\
    \exists z_0>0 \text{ such that } g(z) > B(z) \text{ for } z\ge z_0,
    \end{array}  \right. 
 \end{equation}
then $g(z)>B(z)$ for all $z>0$. \\

Let $g(z) = \frac{z(\sqrt{z^2+z+1}+1)}{z+1}$. Then 
\begin{align*}
 \frac{dg}{dz} - \frac{1}{z}(g^2-z^2) &= \frac{-z^3+\frac{1}{2}z^2-\frac{1}{2}z+1+(z^2-z+1)\sqrt{z^2+z+1}}{(z+1)^2\sqrt{z^2+z+1}}\\
 &> \frac{z^2+\frac{1}{2}z+\frac{3}{2}-z\sqrt{z^2+z+1}}{(z+1)^2\sqrt{z^2+z+1}} \\
 &= \frac{9z^2+6z+9}{\big(4z^2+ 2z+ 6+ 4z\sqrt{z^2+z+1} \big) (z+1)^2\sqrt{z^2+z+1}}>0 
 \end{align*}
for all $z>0$. Here we have used that $\sqrt{z^2+z+1}> \sqrt{z^2+z+\frac{1}{4}}=z+\frac{1}{2}$. Also, by comparison with \eqref{taylorB}, we have $g(2)=\frac{2(1+\sqrt{7})}{3}<\frac{39}{16} < B(2)$, and thus $g$ satisfies \eqref{Bcond1}, so $B(z)>\frac{z\sqrt{z^2+z+1}}{z+1}$ for all $z>0$. \\

Now let $g(z) = z+ \frac{1}{2}$. Then
\[ \frac{dg}{dz} - \frac{1}{z}(g^2-z^2) = 1 - \frac{1}{z}\bigg(z+\frac{1}{4} \bigg) = -\frac{1}{4z} < 0\]
for $z>0$. Using the expansion \eqref{taylorB}, we note that $g(1/10) > B(1/10)$, and therefore $g$ satisfies \eqref{Bcond2}. Thus $B(z)<z+\frac{1}{2}$ for all $z>0$. \\

Dividing through by $z$ then yields \eqref{bessel}.
\end{proof}

In addition to Lemma \ref{lem:bessel}, which holds for all $z>0$, we will require bounds for modified Bessel functions which hold for small $z$. In particular, we will make use of the following proposition. 
\begin{proposition}[Bessel function bounds for small $z$]\label{bessel_smallz}
Let $K_0(z)$ and $K_1(z)$ denote second-kind modified Bessel functions of zeroth and first order, respectively.
For $0<z<1$, the following lower bounds hold: 
\begin{align}
\label{K0_bd}
K_0(z) \ge -\log z \\
\label{K1_bd}
zK_1(z) \ge 1- z^2(1+\abs{\log z})
\end{align}
\end{proposition}

\begin{proof}
The modified Bessel functions $K_0(z)$ and $K_1(z)$ enjoy the following asymptotic expansions near $z=0$:
\begin{align*}
K_0(z) &= \log 2-\gamma - \log z + \frac{1}{4} (1+\log 2-\gamma - \log z) z^2 +\frac{1}{128}(3+2(\log2-\gamma) - 2\log z)z^4+ \dots \\ 
 zK_1(z) &= 1 + \frac{1}{4} (-1 - 2(\log 2-\gamma)+2\log z)z^2 + \frac{1}{64}(-5-4(\log 2-\gamma) + 4\log z)z^4+\dots  
\end{align*}
Since $\log z<0$ for $0<z<1$ and $0<\log 2-\gamma\approx 0.1159 < \frac{1}{8}$, the bound \eqref{K0_bd} follows immediately from the expansion for $K_0$. Furthermore, we have
\[ zK_1(z)-1 \ge 2\bigg(-\frac{1}{4}\bigg(1+\frac{1}{8} - 2\log z\bigg)z^2 \bigg) \ge -(1+\abs{\log z})z^2. \]
\end{proof}

\subsection{Calculation of Laplace spectrum}\label{sec:lap_calc}
We show that the spectrum of the Laplace slender body PDE operator $\mc{L}_\epsilon^{-1}$, given by \eqref{lap_map}, is of the form \eqref{true_evals} in Proposition \ref{Laplace_spec}. To do so, we solve for $u$ satisfying 
\begin{equation}\label{laplace_eval}
\begin{aligned}
\Delta u &=0 \quad \text{in }(\R^2\times \T)\backslash\overline{\mc{C}_\epsilon}\\
u|_{\p \mc{C}_\epsilon}&=e^{\pi i kz}, \quad \abs{k}=1,2,3,\dots 
\end{aligned}
\end{equation}

We compute the eigenvalues $\lambda$ satisfying 
\[ \int_0^{2\pi} \frac{\p u}{\p \nu}\bigg|_{\p \mc{C}_\epsilon} \epsilon d\theta= \int_0^{2\pi} -\p_r u \big|_{\p \mc{C}_\epsilon} \epsilon d\theta = \lambda e^{\pi i kz},\]
where $\nu$ is the unit normal to $\p \mc{C}_\epsilon$ pointing into the cylinder. \\

First, due to the symmetry of \eqref{laplace_eval} with respect to $\theta$, we look for a radial solution $u(r,z)$ to \eqref{laplace_eval}, which we write as 
\[ u(r,z) = U(r)e^{\pi i kz}; \quad U(\epsilon)=1, \, U(r)\to 0 \text{ as }r\to\infty.\]

Using this ansatz in \eqref{laplace_eval}, rewritten in cylindrical coordinates, we have that $U$ satisfies the ODE 
\[ \p_{rr}U + \frac{1}{r} \p_rU - \pi^2k^2U = 0.\]
Therefore
\[ U(r) = c_1 I_0(\pi r\abs{k}) + c_2K_0(\pi r\abs{k}),\]
where $I_0$ and $K_0$ are modified Bessel functions of the first and second kind, respectively. Using the boundary conditions for $U$, we thus obtain
\[ U(r) = \frac{K_0(\pi r\abs{k})}{K_0(\pi\epsilon \abs{k})}.\]

Then
\[ \int_0^{2\pi} -\p_r u \big|_{\p \mc{C}_\epsilon} \epsilon d\theta= \int_0^{2\pi} \pi \abs{k}\frac{K_1(\pi\epsilon \abs{k})}{K_0(\pi\epsilon \abs{k})} e^{\pi ikz} \epsilon d\theta = 2\pi^2 \epsilon \abs{k}\frac{K_1(\pi\epsilon \abs{k})}{K_0(\pi\epsilon \abs{k})}e^{\pi ikz}. \]

Thus the eigenvalues $\lambda_k^{\rm l}$ of the map $\mc{L}_\epsilon^{-1}$ have the form \eqref{true_evals}. \\

The linear growth rate \eqref{linear} of $\lambda_k^{\rm l}$ follows immediately from Lemma \ref{bessel}.

\subsection{Proof of Lemma \ref{lam_lamSB}: difference in Laplace eigenvalues}\label{lap_diff}
To show that the eigenvalues $\lambda^{\rm l}_k$ and $\lambda^{\SB,{\rm l}}_k$ of the Laplace slender body PDE and approximation, respectively, agree to order $\epsilon^2k^2$ for sufficiently low wavenumbers $\abs{k}$, we make use of the ODEs satisfied by the continuous functions along which the eigenvalues \eqref{true_evals} and \eqref{KR_eigs_L} lie.
 
\begin{proof}
We again consider the function $B(z)$ as in \eqref{Bdef}, and recall the ODE \eqref{BODE} satisfied by $B$. In addition, we consider the function
\begin{align*}
B^\SB(z) = -\frac{1}{\log(z/2) + \gamma},
\end{align*}
which also satisfies an ODE of the form
\begin{equation}\label{BSB_ODE}
\frac{dB^\SB}{dz} = \frac{1}{z}( B^\SB)^2, \; B^\SB(0) = 0.
\end{equation}

The two functions $B$ and $B^\SB$ are plotted in Figure \ref{fig:Compare1}. \\
\begin{figure}[!h]
\centering
\includegraphics[scale=0.5]{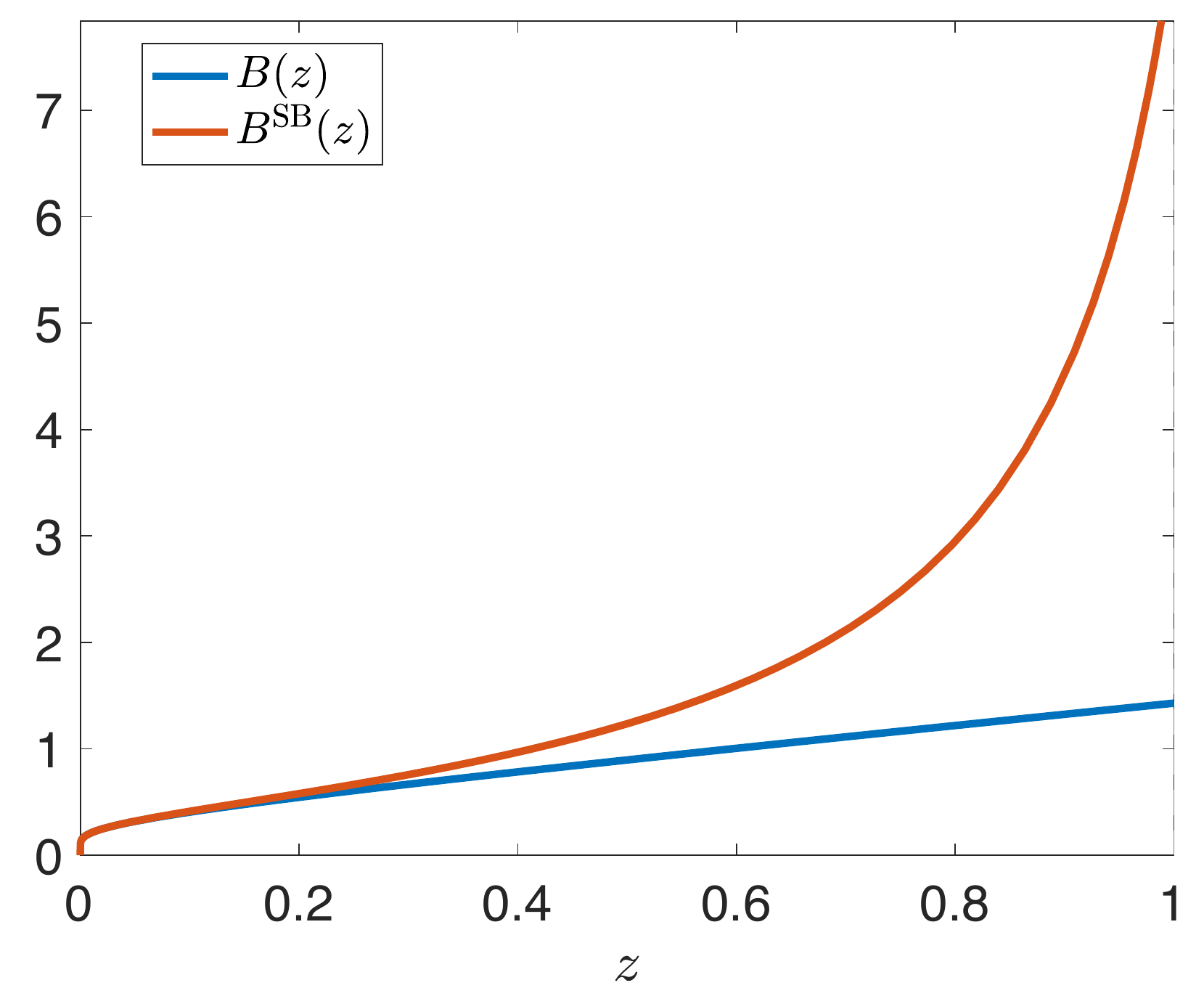}\\
\caption{Plot of the functions $B(z)$ (blue line) and $B^\SB(z)$ (red line). We can see how $B$ and $B^\SB$ correspond closely for small $z$, but diverge quickly as $z$ increases. Note that the eigenvalues of $\mc{L}_\epsilon^{-1}$ and $(\mc{L}_\epsilon^\SB)^{-1}$ correspond to $B$ and $B^\SB$, respectively, by $\lambda_k^{\rm l}=2\pi B(\pi\epsilon k)$ and $\lambda^{\SB,{\rm l}}_k=2\pi B^\SB(\pi\epsilon k)$. }
\label{fig:Compare1}
\end{figure}

Since $B$ and $B^\SB$ clearly deviate as $z\to 2e^{-\gamma}\approx 1.1229$, we are interested in comparing $B$ and $B^\SB$ only up to some $z<2e^{-\gamma}$; in particular, we will consider only $z\le 9/20$. The reason for this cutoff is the following. Note that $B^\SB+B$ is strictly increasing for $0<z<2e^{-\gamma}$, and for $z\le 9/20$, we have $B^\SB+B\le c_B:=B^\SB(9/20)+B(9/20) \approx 1.9339$. The Gr\"onwall argument we will use relies on $c_B<2$, so we take $z=9/20$ as a (relatively round) cutoff. \\

Now, due to the form of the ODEs \eqref{BODE} and \eqref{BSB_ODE}, the difference $\overline B = B^\SB-B$ is positive for $0<z<2e^{-\gamma}$ and satisfies
\begin{equation}\label{diff_ineq}
\frac{d\overline B}{dz} = \frac{1}{z}(B^\SB+B)\overline B + z \le \frac{c_B}{z}\overline B + z, \quad \overline B(0)=0.
\end{equation}

Using a Gr\"onwall inequality, we have 
\begin{equation}\label{diff}
\overline B(z) \le \int_0^z e^{c_B(\log z - \log s)} s \, ds = z^{c_B}\int_0^z s^{1-c_B} \, ds = \frac{1}{2-c_B} z^2. 
\end{equation}

Taking $z=\pi\epsilon \abs{k}$, $\abs{k}=1,2,3,\dots$, \eqref{diff} implies that the difference $\lambda^{\SB,{\rm l}}_k - \lambda_k^{\rm l}$ satisfies 
\begin{equation}\label{lamSB_lam}
\abs{\lambda^{\SB,{\rm l}}_k - \lambda_k^{\rm l}} \le \frac{2\pi^3}{2-c_B}\epsilon^2k^2.
\end{equation}
\end{proof}

\subsection{Proof of Theorems \ref{thm:err} and \ref{thm:wellpo}: Laplace error estimate and well-posedness}\label{lap_err}
With Lemma \ref{lam_lamSB}, we are equipped to prove the error bound between the Laplace slender body PDE operator $\mc{L}_\epsilon^{-1}$ and the truncated Laplace slender body approximation $(\mc{L}_\epsilon^\SB)^{-1}_N$ stated in Theorem \ref{thm:err}. The proof of Theorem \ref{thm:err} will in turn allow us to prove the $\epsilon$-dependence in the well-posedness estimate of Theorem \ref{thm:wellpo}. 
\begin{proof}[Proof of Theorem \ref{thm:err}:]
For $u \in H^1(\T)$, we have
\begin{equation}\label{intermed_est}
\begin{aligned}
\norm{\mc{L}_\epsilon^{-1} [u] - (\mc{L}_\epsilon^\SB)^{-1}_N [u]}_{L^2(\T)}^2 &= \sum_{\abs{k}=1}^N (\lambda_k^{\rm l} - \lambda^{\SB,{\rm l}}_k)^2 \abs{\wh u_k}^2 + \sum_{\abs{k}=N+1}^\infty (\lambda_k^{\rm l})^2\abs{\wh u_k}^2 \\
&\le C\sum_{\abs{k}=1}^N \epsilon^4 k^4 \abs{\wh u_k}^2 + C\sum_{\abs{k}=N+1}^\infty (1+\epsilon \abs{k})^2 \abs{\wh u_k}^2 \\
&\le C\epsilon^4 N^2 \sum_{\abs{k}=1}^N k^2 \abs{\wh u_k}^2 + C\sum_{\abs{k}=N+1}^\infty \bigg(\frac{1}{k^2}+\epsilon^2 \bigg) k^2\abs{\wh u_k}^2 \\
&\le C\bigg(\epsilon^4 N^2 + \frac{1}{N^2} + \epsilon^2\bigg) \norm{u}_{H^1(\T)}^2,
\end{aligned}
\end{equation}
where we have used Lemma \ref{lam_lamSB} and the growth rate \eqref{linear} in the second line. If $u\in H^2(\T)$, we instead have
\begin{equation}\label{intermed_est2}
\begin{aligned}
\norm{\mc{L}_\epsilon^{-1} [u] - (\mc{L}_\epsilon^\SB)^{-1}_N [u]}_{L^2(\T)}^2 &\le C\epsilon^4\sum_{\abs{k}=1}^N k^4 \abs{\wh u_k}^2 +  C\sum_{\abs{k}=N+1}^\infty \bigg(\frac{1}{k^4}+\frac{\epsilon^2}{k^2} \bigg) k^4\abs{\wh u_k}^2 \\
&\le C\bigg(\epsilon^4 + \frac{1}{N^4} + \frac{\epsilon^2}{N^2}\bigg) \norm{u}_{H^2(\T)}^2.
\end{aligned}
\end{equation}
In both \eqref{intermed_est} and \eqref{intermed_est2}, taking $N=C/\epsilon$ for some constant $C$ will yield the optimal dependence on $\epsilon$ as $\epsilon\to 0$. Due to Lemma \ref{lam_lamSB}, we have that $C$ must satisfy $C\le 9/(20\pi)$, yielding Theorem \ref{thm:err}. 
\end{proof}

Using Theorem \ref{thm:err}, we can now show the estimate of Theorem \ref{thm:wellpo}.
\begin{proof}[Proof of Theorem \ref{thm:wellpo}:]
The proof will rely on the proof of Theorem \ref{thm:err}. Begin by noting that if we instead choose $N= \floor{1/\sqrt{\epsilon}}$ in the definition \eqref{LambdaN} of $(\mc{L}_\epsilon^\SB)^{-1}_N$, where $\floor{q}$ denotes the nearest integer less than or equal to $q$, we then have that, for any $u\in L^2(\T)$, $(\mc{L}_\epsilon^\SB)^{-1}_N [u]$ satisfies 
\begin{equation}\label{1st_est}
\begin{aligned}
\norm{(\mc{L}_\epsilon^\SB)^{-1}_N [u]}_{L^2(\T)}^2 &= \sum_{\abs{k}=1}^N \abs{\lambda^{\SB,{\rm l}}_k}^2\abs{\wh u_k}^2 \le C\sum_{\abs{k}=1}^N \frac{1}{\log(\epsilon\abs{k})^2} \abs{\wh u_k}^2 \\
&\le C\frac{1}{\log(\epsilon^{1/2})^2}\sum_{\abs{k}=1}^N \abs{\wh u_k}^2 \le \frac{C}{\abs{\log\epsilon}^2}\norm{u}_{L^2(\T)}^2.
\end{aligned}
\end{equation}
Here we have used the definition \eqref{SBT_spec_per} of $\lambda^{\SB,{\rm l}}_k$ in the first inequality. \\

Furthermore, using $N=\floor{1/\sqrt{\epsilon}}$ in estimate \eqref{intermed_est} from the proof of Theorem \ref{thm:err}, we have
\begin{equation}\label{2nd_est}
\norm{\mc{L}_\epsilon^{-1}[ u] - (\mc{L}_\epsilon^\SB)^{-1}_N [u]}_{L^2(\T)} \le C\bigg( \epsilon^2 N + \frac{1}{N} + \epsilon\bigg) \norm{u}_{H^1(\T)} = C\sqrt{\epsilon}\norm{u}_{H^1(\T)}.
\end{equation}

Combining estimates \eqref{1st_est} and \eqref{2nd_est}, we then obtain 
\begin{align*}
\norm{\mc{L}_\epsilon^{-1} [u]}_{L^2(\T)} &\le \norm{\mc{L}_\epsilon^{-1} [u] - (\mc{L}_\epsilon^\SB)^{-1}_N [u]}_{L^2(\T)} + \norm{(\mc{L}_\epsilon^\SB)^{-1}_N [u]}_{L^2(\T)} \le \frac{C}{\abs{\log\epsilon}}\norm{u}_{H^1(\T)}. 
\end{align*}
\end{proof}

\subsection{Proof of Lemma \ref{lam_lamREG}: difference in $\delta$-regularized Laplace eigenvalues}\label{sec:REG}
 In this section, we prove Lemma \ref{lam_lamREG} bounding the difference between the Laplace slender body PDE eigenvalues $\lambda^{\rm l}_k$ \eqref{true_evals} and the eigenvalues $\lambda^{\delta,{\rm l}}_k$ of the $\delta$-regularized Laplace approximation \eqref{reg_lap1}. As in the proof of Lemma \ref{lam_lamSB}, we rely on the ODE satisfied by the continuous versions of the expressions for $\lambda^{\rm l}_k$ and $\lambda^{\delta,{\rm l}}_k$. 

\begin{proof}[Proof of Lemma \ref{lam_lamREG}]
We first calculate the form of $1/\lambda^{\delta,{\rm l}}_k$, the eigenvalues of the (forward) operator $\mc{L}_\epsilon^\delta$ mapping $f$ to $u$. The integral term of \eqref{reg_lap1} may be written as a convolution with the kernel  
\begin{equation}\label{kernel}
 \mc{K}(z) = \frac{1}{4\pi}\frac{1}{\sqrt{z^2+ \delta^2\epsilon^2} };  
 \end{equation}
therefore, the spectrum of $\mc{L}_\epsilon^\delta$ is given by 
\begin{align*}
 \frac{1}{\lambda^{\delta,{\rm l}}_k} = \frac{1}{2\pi} \log\delta +  \int_{-1}^1 \mc{K}(z) e^{\pi i k z} \, dz = \frac{1}{2\pi}\big(\log\delta+ K_0(\delta\pi\epsilon \abs{k}) \big), \; \abs{k}=1,2,3,\dots. 
 \end{align*}
 
For $\delta> 1$, we have that $\lambda^{\delta,{\rm l}}_k= \frac{2\pi}{\log\delta+ K_0(\delta \pi\epsilon \abs{k})}$ is strictly positive and approaches the positive constant $\frac{2\pi}{\log(\delta)}$ as $\abs{k}\to\infty$ (see Figure \ref{fig:3Bs}). Thus, due to the linear growth of the slender body PDE eigenvalues $\lambda_k^{\rm l}$ (equation \eqref{linear} in Proposition \ref{true_evals}), we immediately obtain the bound \eqref{Lreg_diff0}. \\

 \begin{figure}[!h]
\centering
\includegraphics[scale=0.5]{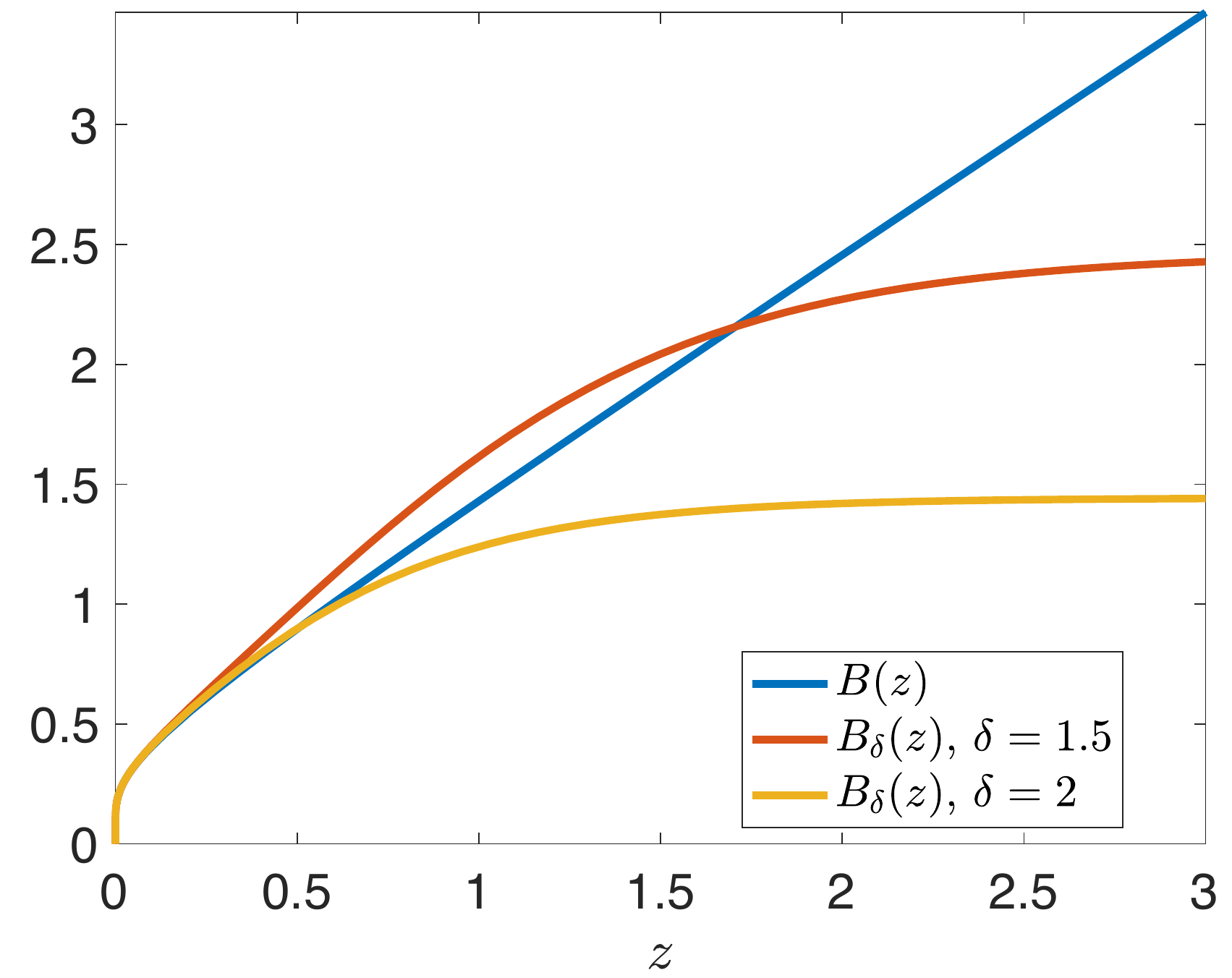}\\
\caption{Plot of the functions $B(z)$ (blue), $B_\delta(z)$ with $\delta=1.5$ (red), and $B_\delta(z)$ with $\delta=2$ (yellow) for $0\le z\le 3$ (note the different scale from Figure \ref{fig:Compare1}). For $\delta>1$, we can see that $B_\delta(z)$ closely aligns with $B(z)$ near $z=0$ and diverges linearly as $z\to\infty$. Recall that $\lambda^{\rm l}_k = 2\pi B(\pi\epsilon k)$ and $\lambda^{\delta,{\rm l}}_k = 2\pi B_\delta(\pi\epsilon k)$. }
\label{fig:3Bs}
\end{figure}
 
To obtain the refined low wavenumber bound \eqref{Lreg_diff1}, we consider the function
 \begin{align*}
 B_\delta(z) = \frac{1}{\log\delta + K_0(\delta z)}
 \end{align*}
 on the interval $0\le z \le \frac{2}{5}$. Note that $B_\delta$ satisfies $2\pi B_\delta(\pi\epsilon k)=\lambda^{\delta,{\rm l}}_k$ as well as the ODE
 \begin{equation}\label{BdeltODE}
 \frac{d B_\delta}{dz} = \delta K_1(\delta z) B_\delta^2(z), \quad B_\delta(0)=0.
 \end{equation}
 
Recalling the definition \eqref{Bdef} of $B(z)$ along with the ODE \eqref{BODE}, we have that the absolute value of the difference $\wt B(z) := B_\delta(z) - B(z)$ satisfies
 \begin{align*}
 \frac{d |\wt B|}{dz} = \bigg(\frac{1}{z}(B_\delta+B)\wt B(z) + \bigg(\delta K_1(\delta z)- \frac{1}{z} \bigg) B_\delta^2+ z\bigg){\rm sgn}(\wt B), \qquad \wt B(0)=0.
 \end{align*}

Now, on the interval $0\le z \le \frac{2}{5}$, by Proposition \ref{bessel_smallz}, equation \eqref{K0_bd}, we have that 
\begin{align*}
B_\delta(z)\le \frac{1}{\log\delta - \log(\delta z)} \le \frac{1}{\abs{\log z}}. 
\end{align*}
 Furthermore, by Proposition \ref{bessel_smallz}, equation \eqref{K1_bd}, we have 
\begin{equation}\label{delK1_bd}
 \abs{\delta K_1(\delta z)- \frac{1}{z}}\le \delta^2 z (1+\log\delta + \abs{\log z}). 
 \end{equation}
 Note that although Proposition \ref{bessel_smallz} is stated for $\delta z<1$, both \eqref{K0_bd} and \eqref{K1_bd} are trivially true in the case $\delta z\ge 1$ and thus there is no further restriction on the choice of $\delta$ due to these bounds. \\
 
We also define the constant 
\begin{equation}\label{cl2}
c_{\rm l, 2} := \frac{1}{\abs{\log(2/5)}} + B(2/5) \approx 1.8753
\end{equation}
and note that $(B_\delta+B)\le c_{\rm l, 2} <2$. \\

We then have
\begin{equation}\label{wtBODE}
 \begin{aligned}
 \frac{d |\wt B|}{dz} &\le \frac{c_{\rm l, 2}}{z} |\wt B(z)| + \delta^2 z (1+\log\delta + \abs{\log z})\frac{1}{\abs{\log z}^2} + z \\
&\le \frac{c_{\rm l, 2}}{z} |\wt B(z)|  + \frac{3}{2}\delta^2 z \bigg(\frac{5}{2}+ \log\delta \bigg) + z\\
&\le \frac{c_{\rm l, 2}}{z} |\wt B(z)|  + 5\delta^2\big(1+ \log\delta \big)  z , \qquad \wt B(0)=0.
 \end{aligned}
 \end{equation}
 Here we have also used that $\frac{1}{\abs{\log z}} \le \frac{1}{\abs{\log(2/5)}} \approx 1.0914 < \frac{3}{2}$. \\

We may again use a Gr\"onwall argument to show that
\begin{equation}\label{GronLapReg}
\begin{aligned}
\abs{\wt B(z)} &\le 5\delta^2\big(1+ \log\delta \big)\int_0^z e^{c_{\rm l,2}(\log z -\log s)}s \, ds \\
&=5\delta^2\big(1+ \log\delta \big) z^{c_{\rm l,2}}\int_0^z s^{1-c_{\rm l,2}} \, ds = \frac{5\delta^2(1+ \log\delta)}{2-c_{\rm l,2}} z^2 \le 41 \delta^2(1+ \log\delta)z^2.
\end{aligned}
\end{equation}
Here it is important that we have chosen our interval $0\le z\le \frac{2}{5}$ such that $c_{\rm l,2}<2$. Taking $z=\pi\epsilon \abs{k}$ yields \eqref{Lreg_diff1} and completes the proof of Lemma \ref{lam_lamREG}.
\end{proof}

\subsection{Proof of Theorem \ref{thm:errREG}: error in $\delta$-regularized Laplace expression}\label{sec:REGerr}
We may now use Lemma \ref{lam_lamREG} in the same way as Lemma \ref{lam_lamSB} was used in Section \ref{lap_err} to estimate the difference between $\mc{L}_\epsilon^{-1} [u]$ and $(\mc{L}_\epsilon^\delta)^{-1} [u]$. 

\begin{proof}[Proof of Theorem \ref{thm:errREG}]
We begin by denoting $N'=\floor{ \frac{2}{5\pi\epsilon}}$. For $u\in H^1(\T)$, we use both the high and low wavenumber bounds of Lemma \ref{lam_lamREG} to obtain 
\begin{equation}\label{REG_est0}
\begin{aligned}
&\norm{\mc{L}_\epsilon^{-1} [u] - (\mc{L}_\epsilon^\delta)^{-1} [u]}_{L^2(\T)}^2 \\
&\hspace{1cm} = \sum_{\abs{k}=1}^{N'} (\lambda_k^{\rm l} - \lambda^{\delta,{\rm l}}_k)^2 \abs{\wh u_k}^2 + \sum_{\abs{k}=N'+1}^\infty (\lambda_k^{\rm l} - \lambda^{\delta,{\rm l}}_k)^2 \abs{\wh u_k}^2 \\
&\hspace{1cm}  \le 82^2\pi^6\delta^4(1+\log\delta)^2\sum_{\abs{k}=1}^{N'} \epsilon^4 k^4 \abs{\wh u_k}^2 + 4\pi^2\sum_{\abs{k}=N'+1}^\infty \bigg(\frac{1}{2}+ \frac{1}{\log\delta}+ \pi\epsilon \abs{k}\bigg)^2 \abs{\wh u_k}^2 \\
&\hspace{1cm} \le 82^2\pi^6\delta^4(1+\log\delta)^2\epsilon^4 (N')^2 \sum_{\abs{k}=1}^{N'} k^2 \abs{\wh u_k}^2 + 4\pi^2\sum_{\abs{k}=N'+1}^\infty \bigg(\frac{(1+\log\delta)^2}{(\log\delta)^2}\frac{1}{k^2}+\epsilon^2 \bigg) k^2\abs{\wh u_k}^2 \\
&\hspace{1cm} \le \pi^2(1+\log\delta)^2\bigg(82^2\pi^4\delta^4\epsilon^4 (N')^2 + \frac{4}{(\log\delta)^2(N')^2} + 4\epsilon^2 \bigg)\bigg) \norm{u}_{H^1(\T)}^2.
\end{aligned}
\end{equation}

For $u\in H^2(\T)$, we may replace \eqref{REG_est0} with the bound
\begin{equation}\label{REG_est1}
\begin{aligned}
&\norm{\mc{L}_\epsilon^{-1} [u] - (\mc{L}_\epsilon^\delta)^{-1} [u]}_{L^2(\T)}^2 \\
&\hspace{1cm} \le 82^2\pi^6\delta^4(1+\log\delta)^2\epsilon^4 \sum_{\abs{k}=1}^{N'} k^4 \abs{\wh u_k}^2 + 4\pi^2\sum_{\abs{k}=N'+1}^\infty \bigg(\frac{(1+\log\delta)^2}{\log\delta^2}\frac{1}{k^4}+ \frac{\epsilon^2}{k^2} \bigg) k^4\abs{\wh u_k}^2 \\
&\hspace{1cm} \le \pi^2(1+\log\delta)^2\bigg(82^2\pi^4\delta^4\epsilon^4 + \frac{4}{\log\delta^2(N')^4} + \frac{4\epsilon^2}{(N')^4}\bigg) \norm{u}_{H^2(\T)}^2.
\end{aligned}
\end{equation}

Using that $N'=\floor{ \frac{2}{5\pi\epsilon}}$, estimates \eqref{REG_est0} and \eqref{REG_est1} then yield Theorem \ref{thm:errREG}. 
\end{proof}

\section{Slender body Stokes PDE}\label{sec:stokes}
We begin this section with the proofs of the key Lemmas \ref{tan_eig_diff} and  \ref{nor_eig_diff} bounding the difference between the eigenvalues of the Stokes slender body PDE operator $\mc{L}_\epsilon^{-1}$ and the slender body approximation $(\mc{L}^\SB_\epsilon)^{-1}$ in the tangential and normal directions, respectively. Section \ref{sec:tan_diff} is devoted to the tangential direction (Lemma \ref{tan_eig_diff} ) and Section \ref{sec:nor_diff} concerns the normal direction (Lemma \ref{nor_eig_diff}). The proofs rely on a comparison of the ODEs satisfied by (continuous versions of) the eigenvalues of both operators. Then in Section \ref{Stokes_err}, we use Lemmas \ref{tan_eig_diff} and \ref{nor_eig_diff} to prove Theorem \ref{thm:errS} bounding the difference $\mc{L}_\epsilon^{-1}[\bu] - (\mc{L}^\SB_\epsilon)^{-1}[\bu]$, and use Theorem \ref{thm:errS} to prove the $\epsilon$-dependence in the well-posedness estimate of Theorem \ref{thm:wellpoS}.

\subsection{Proof of Lemma \ref{tan_eig_diff}: difference in tangential eigenvalues}\label{sec:tan_diff}
In this section we show that the difference between the tangential direction eigenvalues $\lambda^{\rm t}_k$ of the Stokes slender body PDE operator \eqref{eigsT} and the corresponding eigenvalues $\lambda^{\SB,{\rm t}}_k$ of the slender body approximation \eqref{KR_eigs_St} are bounded by an expression proportional to $\epsilon^2k^2$ \eqref{tan_diff}.

\begin{proof}
We begin by considering the function 
\begin{equation}\label{Bt}
B_{\rm t}(z) = \frac{z K_1^2(z)}{2K_0(z)K_1(z) + z \big( K_0^2(z) - K_1^2(z) \big)}
\end{equation}
for $z\in \R_+$. Notice that $B_{\rm t}\ge0$ since, by Lemma \ref{bessel}, for $z\in \R_+$ we have
\begin{equation}\label{diffLB} 
0\le K_1^2-K_0^2\le \frac{K_0(K_1+K_0)}{2z} \le \frac{K_0K_1}{z}, 
\end{equation}
and therefore
\[ 2K_0K_1 + z(K_0^2-K_1^2) \ge K_0K_1\ge 0.\]
Furthermore, we have $\lambda^{\rm t}_k=4\pi B_{\rm t}(\pi\epsilon \abs{k})$, $\abs{k}=1,2,3,\dots$. \\

Differentiating \eqref{Bt}, we see that $B_{\rm t}(z)$ satisfies the ODE
\begin{equation}\label{BtODE}
\frac{d B_{\rm t}}{dz} = \frac{2}{z}(B_{\rm t})^2 - 2 \frac{K_0(z)}{K_1(z)} B_{\rm t}; \quad B_{\rm t}(0)=0.
\end{equation}

Using \eqref{BtODE}, we can show that $B_{\rm t}$ is monotone increasing for $z\in \R_+$. In particular, we have:
\begin{proposition}\label{Bt_bd}
For $z>0$, the function $B_{\rm t}(z)$ satisfies
\begin{equation}\label{eqBt_bd}
\frac{B_{\rm t}(z)}{z} > \frac{K_0(z)}{K_1(z)}.
\end{equation}
\end{proposition}
\begin{proof}
Proving \eqref{eqBt_bd} amounts to showing that
\begin{align*}
\bigg(\frac{K_1}{K_0}\bigg)^3 + z \bigg(\frac{K_1}{K_0}\bigg)^2 - 2\frac{K_1}{K_0} - z \ge 0.
\end{align*}
By Lemma \ref{lem:bessel}, we have
\begin{align*}
\bigg(\frac{K_1}{K_0}\bigg)^3 + z \bigg(\frac{K_1}{K_0}\bigg)^2 - 2\frac{K_1}{K_0} - z &\ge \frac{(z^2-z+2)\sqrt{1+z+z^2}-z^3+z^2+2}{(1+z)^3} \\
&\ge \frac{z^2+2z+6}{2(1+z)^3} >0.
\end{align*}
Here we have used that $\frac{1}{2}+z \le \sqrt{1+z+z^2} \le 1+z$. 
\end{proof}

In addition to $B_{\rm t}(z)$, we consider the function 
 \begin{equation}\label{BSBt}
B_{\rm t}^\SB(z) = -\frac{1}{1+2\log(z/2)+2\gamma},
\end{equation}
and note that the tangential eigenvalues $\lambda^{\SB,{\rm t}}_k$ \eqref{KR_eigs_St} of the slender body approximation are given by $4\pi B_{\rm t}^\SB(\pi\epsilon \abs{k})$, $\abs{k}=1,2,3,\dots$. The function $B_{\rm t}^\SB(z)$ may be compared to $B_{\rm t}(z)$ in Figure \ref{fig:Compare_t}.  \\

\begin{figure}[!h]
\centering
\includegraphics[scale=0.5]{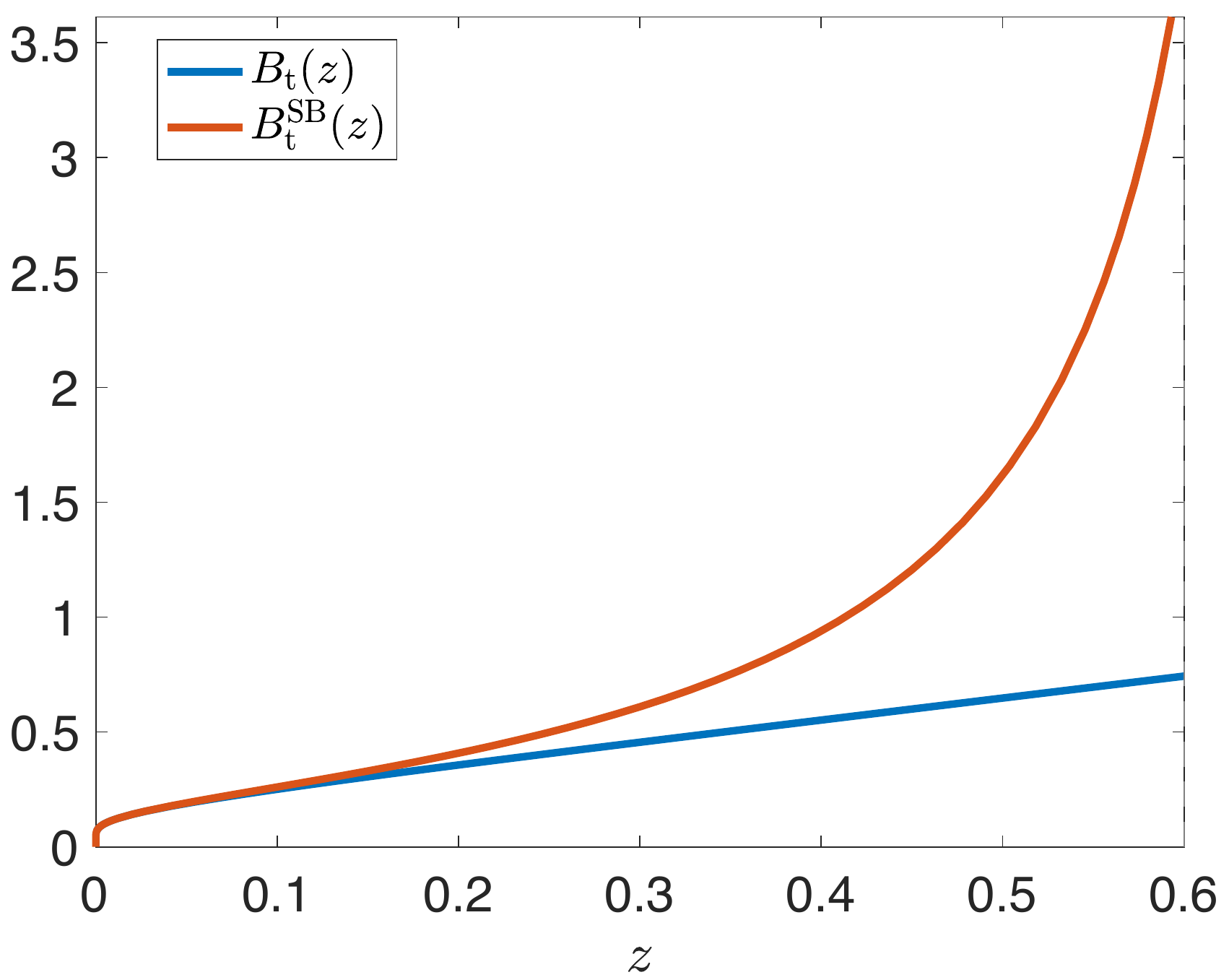}\\
\caption{Plot of the functions $B_{\rm t}(z)$ (blue) and $B_{\rm t}^\SB(z)$ (red). As in the Laplace setting, $B_{\rm t}$ and $B_{\rm t}^\SB$ are nearly identical for small $z$, but diverge quickly as $z$ increases. The functions $B_{\rm t}$ and $B_{\rm t}^\SB$ correspond to the tangential eigenvalues of $\mc{L}_\epsilon^{-1}$ and $(\mc{L}_\epsilon^\SB)^{-1}$, respectively, via $\lambda^{\rm t}_k=4\pi B_{\rm t}(\pi\epsilon k)$ and $\lambda^{\SB,{\rm t}}_k=4\pi B_{\rm t}^\SB(\pi\epsilon k)$. }
\label{fig:Compare_t}
\end{figure}

Similarly to the Laplace setting (Section \ref{lap_diff}), the function $B_{\rm t}^\SB(z)$ satisfies the ODE
\begin{equation}\label{BSBt_ODE}
\frac{d B_{\rm t}^\SB}{dz} = \frac{2}{z}(B_{\rm t}^\SB)^2; \quad B_{\rm t}^\SB(0)=0.
\end{equation}

Due to the form of the ODEs \eqref{BtODE} and \eqref{BSBt_ODE}, the difference $\overline B_{\rm t} = B_{\rm t}^\SB- B_{\rm t}$ is positive for all $0< z< 2e^{-\gamma-1/2}$ and satisfies
\begin{align*}
\frac{d \overline B_{\rm t}}{dz} = \frac{2}{z}(B^\SB_{\rm t}+B_{\rm t})\overline B_{\rm t} + 2\frac{K_0}{K_1}B_{\rm t}; \quad \overline B_{\rm t}(0) =0.
\end{align*}

Now, by Lemma \ref{lem:bessel}, for $z\in \R_+$ we have that 
\begin{equation}\label{KBt_bd}
2\frac{K_0}{K_1}B_{\rm t} = \frac{z}{1+\frac{z}{2}\big(\frac{K_0}{K_1} - \frac{K_1}{K_0} \big)} \le \frac{z}{1+\frac{z}{2}\big(\frac{2z}{2z+1} - \frac{2z+1}{2z} \big)} = \frac{2z(4z+2)}{4z+3} \le 2z.
\end{equation}

Noting that $B^\SB_{\rm t}$ and $B_{\rm t}$ are both monotone increasing for $0\le z< 2e^{-\gamma-1/2}\approx 0.681$, let 
\begin{equation}\label{cBt}
c_{\rm t} = \max_{0\le z\le 1/4} \big(B_{\rm t}^\SB(z) + B_{\rm t}(z) \big)= B_{\rm t}^\SB(1/4) + B_{\rm t}(1/4) \approx 0.905.
\end{equation}
Here the range $0\le z\le \frac{1}{4}$ is chosen such that $c_{\rm t} <1$. Then for $0\le z\le \frac{1}{4}$, we have that $\overline B_{\rm t}$ satisfies
\begin{align*}
\frac{d \overline B_{\rm t}}{dz} \le \frac{2c_{\rm t} }{z}\overline B_{\rm t} + 2z; \quad \overline B_{\rm t}(0) =0.
\end{align*}

Using a Gr\"onwall inequality, we thus obtain
\begin{equation}\label{barBt_ineq}
\overline B_{\rm t} \le 2\int_0^z e^{2c_{\rm t} (\log z - \log s)} s \, ds = 2z^{2c_{\rm t}} \int_0^z s^{1-2c_{\rm t} } \, ds = \frac{z^2}{1-c_{\rm t} }.
\end{equation}

Plugging $z=\pi\epsilon \abs{k}$ into \eqref{barBt_ineq}, we obtain
\[ \abs{\lambda^{\rm t}_k - \lambda^{\SB,{\rm t}}_k} \le \frac{4\pi^3}{1-c_{\rm t} } \epsilon^2k^2.\] 
\end{proof}

\subsection{Proof of Lemma \ref{nor_eig_diff}: difference in normal eigenvalues}\label{sec:nor_diff}
In this section we prove that the difference $\lambda^{\rm n}_k - \lambda^{\SB,{\rm n}}_k$ between the normal direction eigenvalues of the Stokes slender body PDE \eqref{eigsN} and the slender body approximation \eqref{KR_eigs_Sn} is bounded by an expression proportional to $\epsilon^2k^2$ \eqref{nor_diff}.

\begin{proof}
We consider the function
\begin{equation}\label{Bn}
B_{\rm n}(z) =\frac{4zK_1^2K_2+z^2 K_1(K_1^2-K_0K_2)}{2K_0K_1K_2 + z \big(K_1^2(K_0+K_2)-2K_0^2K_2 \big)},
\end{equation}
where $K_j=K_j(z)$, $j=0,1,2$. Note that $\lambda^{\rm n}_k = 2\pi B_{\rm n}(\pi\epsilon \abs{k})$. We have that
\begin{equation}\label{Bn_32z}
B_{\rm n}(z)>\frac{3}{2}z, \qquad z>0,
\end{equation}
which is equivalent to the lower bound of \ref{Ngrowth} and shown in Appendix \ref{app_nor_spec}. \\

Differentiating \eqref{Bn}, we find that $B_{\rm n}(z)$ satisfies the ODE
\begin{equation}\label{BnODE}
\frac{d B_{\rm n}}{dz} = \frac{1}{2z}(B_{\rm n})^2 - h(z); \quad B_{\rm n}(0)=0,
\end{equation}
where
\begin{equation}\label{h_z}
\begin{aligned} 
h(z) &= \frac{z}{8}\bigg(\frac{4 z^4 K_0^6 - 16 z^3 K_0^5 K_1 - 
       z^2 (120 + 11 z^2) K_0^4 K_1^2 + 
       4 z (-40 + 3 z^2) K_0^3 K_1^3}{(z^2 K_0^3 + 
       z K_0^2 K_1 - (2 + z^2) K_0 K_1^2 - z K_1^3)^2} \\
 &\qquad  +\frac{ 2 (-16 + 66 z^2 + 5 z^4) K_0^2 K_1^4 + 
       4 z (32 + z^2) K_0 K_1^5 - 
       3 z^2 (8 + z^2) K_1^6}{ (z^2 K_0^3 + 
       z K_0^2 K_1 - (2 + z^2) K_0 K_1^2 - z K_1^3)^2} \bigg).
\end{aligned}
\end{equation}

We can show that $h(z)$ satisfies the following bound. 
\begin{proposition}\label{h_bd}
For $z>0$, the function $h(z)$ defined in \eqref{h_z} satisfies 
\begin{equation}\label{hbound}
\abs{h(z)} < \frac{9}{8}z.
\end{equation}
\end{proposition}
Due to the expression for $h(z)$, the proof of Proposition \ref{h_bd} is more complicated than the analogous bounds in the Laplace and tangential Stokes cases, and thus appears in Appendix \ref{app_hbd}. Note that by the ODE \eqref{BnODE}, Proposition \ref{h_bd} along with the lower bound \eqref{Bn_32z} implies that $B_{\rm n}$ is monotone increasing on $\R_+$. \\

In addition to $B_{\rm n}(z)$, we consider the function
\begin{equation}\label{BSBn}
B_{\rm n}^\SB(z) = \frac{4}{1-2\log(z/2)-2\gamma},
\end{equation}
which corresponds to the normal direction eigenvalues for the slender body approximation \eqref{KR_eigs_Sn} via $\lambda^{\SB,{\rm n}}_k = 2\pi B_{\rm n}^\SB(\pi\epsilon \abs{k})$. The two functions $B_{\rm n}(z)$ and $B_{\rm n}^\SB(z)$ may be compared in Figure \ref{fig:Compare_n}. \\

\begin{figure}[!h]
\centering
\includegraphics[scale=0.5]{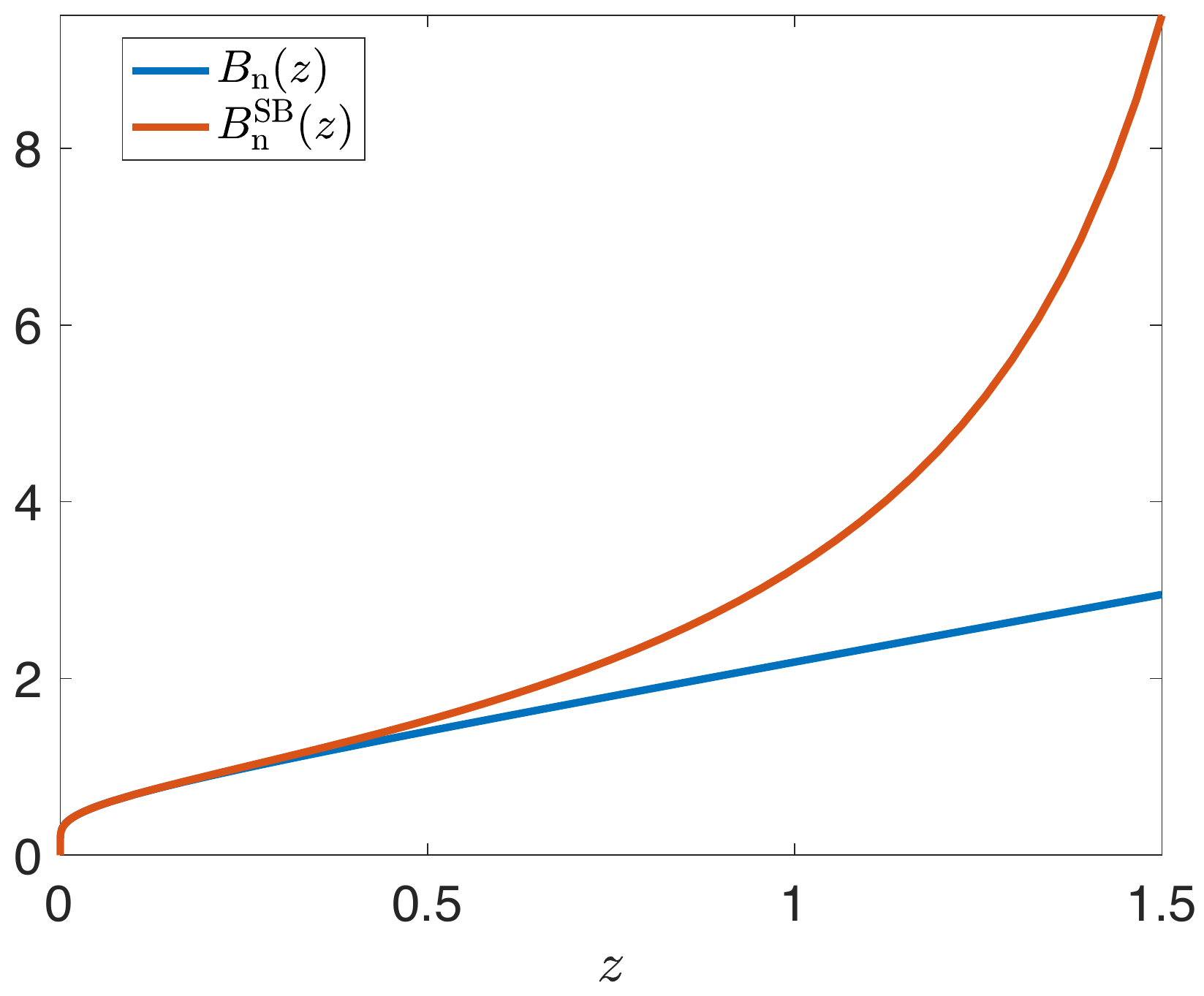}\\
\caption{Plot of the functions $B_{\rm n}(z)$ (blue) and $B_{\rm n}^\SB(z)$ (red). As in the Laplace and tangential Stokes setting, $B_{\rm t}$ and $B_{\rm t}^\SB$ are nearly identical for small $z$ and diverge for larger $z$. Note that the normal direction eigenvalues of $\mc{L}_\epsilon^{-1}$ and $(\mc{L}_\epsilon^\SB)^{-1}$, respectively, are given by $\lambda^{\rm n}_k=2\pi B_{\rm n}(\pi\epsilon k)$ and $\lambda^{\SB,{\rm n}}_k=2\pi B_{\rm n}^\SB(\pi\epsilon k)$. }
\label{fig:Compare_n}
\end{figure}

The function $B_{\rm n}^\SB(z)$ satisfies the ODE
\begin{equation}\label{BSBn_ODE}
\frac{d B_{\rm n}^\SB}{dz} = \frac{1}{2z}(B_{\rm n}^\SB)^2; \quad B_{\rm n}^\SB(0)=0.
\end{equation}

Now, unlike the Laplace and tangential Stokes cases, we don't always have $B_{\rm n}^\SB>B_{\rm n}$, since $h(z)$ is not strictly positive. Thus we consider the absolute value of the difference $\overline B_{\rm n}(z) := B_{\rm n}^\SB(z) - B_{\rm n}(z)$, which satisfies
\begin{equation}\label{barBn_ODE}
\frac{d \abs{\overline B_{\rm n}} }{dz} = \bigg(\frac{1}{2z}(B_{\rm n}^\SB+B_{\rm n})\overline B_{\rm n} + h(z)\bigg)\rm{sgn}(\overline B_{\rm n}); \quad \overline B_{\rm n}(0)=0.
\end{equation}

Using that both $B_{\rm n}(z)$ and $B_{\rm n}^\SB(z)$ are monotone increasing for $0\le z< 2 e^{\frac{1-2\gamma}{2}}\approx 1.851$, we define
\begin{equation}\label{cBn}
c_{\rm n} = \max_{0\le z\le \frac{73}{100}} \big( B_{\rm n}^\SB(z) +B_{\rm n}(z) \big) = B_{\rm n}^\SB(73/100)+B_{\rm n}(73/100) \approx 3.916.
\end{equation}
Here the range $0\le z\le \frac{73}{100}$ is chosen such that $c_{\rm n} < 4$. We then have that for $0\le z\le \frac{73}{100}$, $\abs{\overline B_{\rm n}(z)}$ satisfies 
\begin{align*}
\frac{d \abs{\overline B_{\rm n}}}{dz} \le \frac{c_{\rm n}}{2z}\abs{\overline B_{\rm n}} + \frac{9}{8}z; \quad \abs{\overline B_{\rm n}(0)}=0.
\end{align*}

Using a Gr\"onwall argument as in the tangential case yields
\begin{equation}\label{barBn_ineq}
\abs{\overline B_{\rm n}} \le \frac{9}{8} \int_0^z e^{\frac{c_{\rm n}}{2} (\log z - \log s)} s \, ds = \frac{9}{8}z^{\frac{c_{\rm n}}{2}}  \int_0^z s^{1- \frac{c_{\rm n}}{2}} \, ds = \frac{9z^2}{4(4-c_{\rm n})}.
\end{equation}

Finally, taking $z=\pi\epsilon \abs{k}$ in \eqref{barBn_ineq}, we obtain
\[ \abs{\lambda^{\rm n}_k - \lambda_k^{\SB,{\rm n}} } \le \frac{9\pi^3}{2(4-c_{\rm n})} \epsilon^2k^2.\] 
\end{proof}

\subsection{Proof of Theorems \ref{thm:errS} and \ref{thm:wellpoS}: Stokes error estimate and well-posedness}\label{Stokes_err} 
As in the Laplace setting, the error estimate in Theorem \ref{thm:errS} follows from a direct application of Lemmas \ref{tan_eig_diff} and \ref{nor_eig_diff}. The bound for $(\mc{L}_\epsilon)^{-1}$ of Theorem \ref{thm:wellpoS} then follows using the proof of Theorem \ref{thm:errS}. 

\begin{proof}[Proof of Theorem \ref{thm:errS}]
Recall the definition \eqref{LambdaNM} of the truncated slender body approximation $(\mc{L}_\epsilon^\SB)^{-1}_{NM}$. Given a slender body velocity $\bu\in H^1(\T)$, we have that the difference between the slender body PDE operator $\mc{L}_\epsilon^{-1}[\bu]$ and the truncated slender body approximation $(\mc{L}_\epsilon^\SB)^{-1}_{NM} [\bu]$ satisfies
\begin{equation}\label{errScomp0}
\begin{aligned}
&\norm{\mc{L}_\epsilon^{-1}[\bu] -(\mc{L}_\epsilon^\SB)^{-1}_{NM}[ \bu]}_{L^2(\T)}^2 \\
&\hspace{1cm} = \sum_{\abs{k}=1}^N \big(\lambda^{\rm t}_k-\lambda^{\SB,{\rm t}}_k\big)^2 \abs{\wh u_{z,k}}^2 + \sum_{\abs{k}=1}^M \big(\lambda^{\rm n}_k-\lambda^{\SB,{\rm n}}_k\big)^2 \big(\abs{\wh u_{x,k}}^2 + \abs{\wh u_{y,k}}^2 \big) \\
&\hspace{2cm} + \sum_{\abs{k}=N+1}^\infty (\lambda^{\rm t}_k)^2 \abs{\wh u_{z,k}}^2 + \sum_{\abs{k}=M+1}^\infty (\lambda^{\rm n}_k)^2 \big(\abs{\wh u_{x,k}}^2 + \abs{\wh u_{y,k}}^2 \big) \\
&\hspace{1cm} \le C\bigg(\sum_{\abs{k}=1}^N \epsilon^4k^4 \abs{\wh u_{z,k}}^2 + \sum_{\abs{k}=1}^M \epsilon^4k^4 \big(\abs{\wh u_{x,k}}^2 + \abs{\wh u_{y,k}}^2 \big) \bigg) \\
&\hspace{2cm} + C\bigg(\sum_{\abs{k}=N+1}^\infty (1+\epsilon \abs{k})^2 \abs{\wh u_{z,k}}^2 + \sum_{\abs{k}=M+1}^\infty (1+\epsilon \abs{k})^2 \big(\abs{\wh u_{x,k}}^2 + \abs{\wh u_{y,k}}^2 \big) \bigg) \\
&\hspace{1cm} \le C\epsilon^4\bigg(N^2\sum_{\abs{k}=1}^N k^2 \abs{\wh u_{z,k}}^2 + M^2 \sum_{\abs{k}=1}^M k^2 \big(\abs{\wh u_{x,k}}^2 + \abs{\wh u_{y,k}}^2 \big) \bigg) \\
&\hspace{2cm}+ C\bigg( \sum_{\abs{k}=N+1}^\infty \bigg(\frac{1}{k^2}+\epsilon^2 \bigg)k^2 \abs{\wh u_{z,k}}^2 + \sum_{\abs{k}=M+1}^\infty \bigg(\frac{1}{k^2}+\epsilon^2\bigg)k^2 \big(\abs{\wh u_{x,k}}^2 + \abs{\wh u_{y,k}}^2 \big) \bigg) \\
&\hspace{1cm}\le C\bigg( \epsilon^4(N^2 + M^2) + \frac{1}{N^2} + \frac{1}{M^2} + \epsilon^2\bigg) \norm{\bu}_{H^1(\T)}^2.
\end{aligned}
\end{equation}
In the first inequality we have used Lemmas \ref{tan_eig_diff} and \ref{nor_eig_diff} as well as the upper bounds on $\lambda^{\rm t}_k$ \eqref{Tgrowth} and $\lambda^{\rm n}_k$ \eqref{Ngrowth}. In the case that we actually have $\bu\in H^2(\T)$, the difference $\mc{L}_\epsilon^{-1}[\bu] -(\mc{L}_\epsilon^\SB)^{-1}_{NM}[ \bu]$ then satisfies 
\begin{equation}\label{errScompH2}
\begin{aligned}
&\norm{\mc{L}_\epsilon^{-1}[\bu] -(\mc{L}_\epsilon^\SB)^{-1}_{NM}[ \bu]}_{L^2(\T)}^2 \\
&\hspace{1cm} \le C\epsilon^4 \bigg(\sum_{\abs{k}=1}^N k^4 \abs{\wh u_{z,k}}^2 + \sum_{\abs{k}=1}^M k^4 \big(\abs{\wh u_{x,k}}^2 + \abs{\wh u_{y,k}}^2 \big) \bigg) \\
&\hspace{2cm}+ C\bigg( \sum_{\abs{k}=N+1}^\infty \bigg(\frac{1}{k^4}+\frac{\epsilon^2}{k^2} \bigg)k^4 \abs{\wh u_{z,k}}^2 + \sum_{\abs{k}=M+1}^\infty \bigg(\frac{1}{k^4}+\frac{\epsilon^2}{k^2} \bigg)k^4 \big(\abs{\wh u_{x,k}}^2 + \abs{\wh u_{y,k}}^2 \big) \bigg) \\
&\hspace{1cm}\le C\bigg( \epsilon^4 + \frac{1}{N^4} + \frac{1}{M^4} + \frac{\epsilon^2}{N^2} + \frac{\epsilon^2}{M^2}\bigg) \norm{\bu}_{H^1(\T)}^2.
\end{aligned}
\end{equation}

 As in the Laplace setting, we see that choosing $N=C_1/\epsilon$ and $M=C_2/\epsilon$ in both \eqref{errScomp0} and \eqref{errScompH2} for constants $C_1$, $C_2$ will yield the best rate of convergence as $\epsilon\to 0$. By Lemmas \ref{tan_eig_diff} and \ref{nor_eig_diff}, we must have $C_1\le 1/(4\pi)$ and $C_2\le 73/(100\pi)$, and thus we obtain Theorem \ref{thm:errS}. 
\end{proof}

We next use the proof of Theorem \ref{thm:errS} to derive the well-posedness estimate for $\mc{L}_\epsilon^{-1}$ stated in Theorem \ref{thm:wellpoS}.
\begin{proof}[Proof of Theorem \ref{thm:wellpoS}]
We begin by recalling the form of the spectrum of the slender body approximation $(\mc{L}_\epsilon^\SB)^{-1}$ in the tangential \eqref{KR_eigs_St} and normal \eqref{KR_eigs_Sn} directions; in particular, $\lambda^{\SB,{\rm t}}\le C\abs{\log(\epsilon \abs{k})}^{-1}$ and $\lambda^{\SB,{\rm n}}\le C\abs{\log(\epsilon \abs{k})}^{-1}$ for $\abs{k}$ sufficiently small (i.e. $\abs{k}<\frac{1}{4\pi\epsilon}$ and $\abs{k}<\frac{73}{100\pi\epsilon}$, respectively). \\

We choose $N=M=\floor{1/\sqrt{\epsilon}}$ in the definition \eqref{LambdaNM} of the truncated slender body approximation $(\mc{L}_\epsilon^\SB)^{-1}_{NM}$, where, again, the notation $\floor{q}$ denotes the nearest integer less than or equal to $q$. Then for any $\bu\in L^2(\T)$, we have 
\begin{equation}\label{errScomp1}
\begin{aligned}
\norm{(\mc{L}_\epsilon^\SB)^{-1}_{NM} [\bu]}_{L^2(\T)}^2 &= \sum_{\abs{k}=1}^N (\lambda^{\SB,{\rm t}}_k)^2 \abs{\wh u_{z,k} }^2 + \sum_{\abs{k}=1}^M  (\lambda^{\SB,{\rm n}}_k)^2 \big(\abs{\wh u_{x,k} }^2 + \abs{\wh u_{y,k} }^2 \big) \\
&\le C \sum_{\abs{k}=1}^{N} \frac{1}{\log(\epsilon k)^2} \big(\abs{\wh u_{x,k} }^2 + \abs{\wh u_{y,k} }^2 + \abs{\wh u_{z,k} }^2\big) \\
&\le \frac{C}{\log(\epsilon^{1/2})^2} \sum_{\abs{k}=1}^{N} \abs{\wh{\bu}_k}^2 \le \frac{C}{\abs{\log\epsilon}^2} \norm{\bu}_{L^2(\T)}^2.
\end{aligned}
\end{equation}

Furthermore, taking $N=M=\floor{1/\sqrt{\epsilon}}$ in estimate \eqref{errScomp0} in the proof of Theorem \ref{thm:errS}, we have
\begin{equation}\label{errScomp2}
\begin{aligned}
\norm{\mc{L}_\epsilon^{-1}[\bu] -(\mc{L}_\epsilon^\SB)^{-1}_{NM} [\bu]}_{L^2(\T)}^2 &\le C\bigg( \epsilon^4(N^2 + M^2) + \frac{1}{N^2} + \frac{1}{M^2} + \epsilon^2\bigg) \norm{\bu}_{H^1(\T)}^2 \le C\epsilon \norm{\bu}_{H^1(\T)}^2.
\end{aligned}
\end{equation}
Finally, combining \eqref{errScomp1} and \eqref{errScomp2}, we obtain
\begin{align*}
\norm{\mc{L}_\epsilon^{-1}[\bu]}_{L^2(\T)} \le \norm{\mc{L}_\epsilon^{-1}[\bu] -(\mc{L}_\epsilon^\SB)^{-1}_{NM} [\bu]}_{L^2(\T)} +\norm{(\mc{L}_\epsilon^\SB)^{-1}_{NM} [\bu]}_{L^2(\T)} \le \frac{C}{\abs{\log\epsilon}} \norm{\bu}_{H^1(\T)}.
\end{align*}
\end{proof}

\subsection{Proof of Lemma \ref{diff_REGS}: difference in $\delta$-regularized Stokes eigenvalues}\label{sec:REGS}
In this section we consider the spectrum of the $\delta$-regularized Stokes slender body approximation $\mc{L}_\epsilon^\delta$, defined in \eqref{KR_reg} for $\Sigma_\epsilon=\mc{C}_\epsilon$. We prove Lemma \ref{diff_REGS} bounding the difference between the eigenvalues of $(\mc{L}_\epsilon^\delta)^{-1}$ and the Stokes slender body PDE operator $\mc{L}_\epsilon^{-1}$ in both the tangential and normal directions.

\begin{proof}[Proof of Lemma \ref{diff_REGS}]
We first note that the integral operator $\bm{K}_\delta$ in \eqref{KR_reg} may be written as a convolution with the same kernel \eqref{kernel} as in the Laplace setting, yielding the forms of the tangential \eqref{lamREG_tS} and normal eigenvalues \eqref{lamREG_nS} of $(\mc{L}_\epsilon^\delta)^{-1}$ via a similar calculation. \\

Furthermore, taking the regularization parameter $\delta>\sqrt{e}$ ensures that both $\lambda^{\delta,{\rm t}}_k$ and $\lambda^{\delta,{\rm n}}_k$ are strictly positive for all $k$, and we have that 
\[ \lim_{\abs{k}\to\infty} \lambda^{\delta,{\rm t}}_k = \frac{4\pi}{-1+2\log\delta}, \quad \lim_{\abs{k}\to\infty} \lambda^{\delta,{\rm n}}_k = \frac{8\pi}{1+2\log\delta}.\] 
See Figure \ref{fig:BdeltaS} for a depiction with $\delta=2$ and $\delta=3$. Due to the boundedness of $\lambda^{\delta,{\rm t}}_k$ and $\lambda^{\delta,{\rm n}}_k$ and the linear growth of both $\lambda^{\rm t}_k$ and $\lambda^{\rm n}_k$ (by Proposition \ref{Stokes_spec}, estimates \eqref{Tgrowth} and \eqref{Ngrowth}, respectively), we immediately obtain the bounds \eqref{REG_diff0} and \eqref{REG_diff00}. \\

 \begin{figure}[!h]
\centering
\includegraphics[scale=0.45]{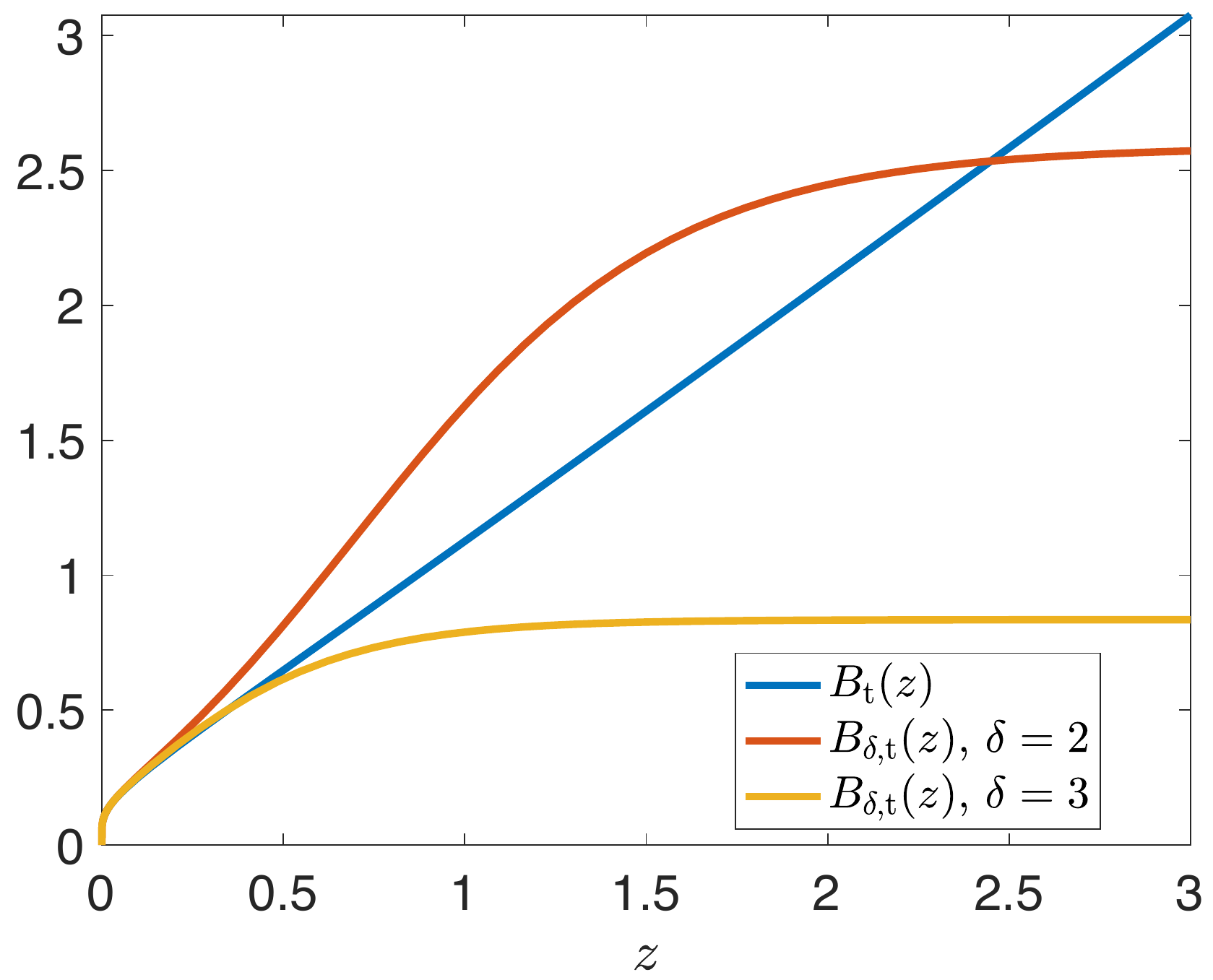}
\includegraphics[scale=0.45]{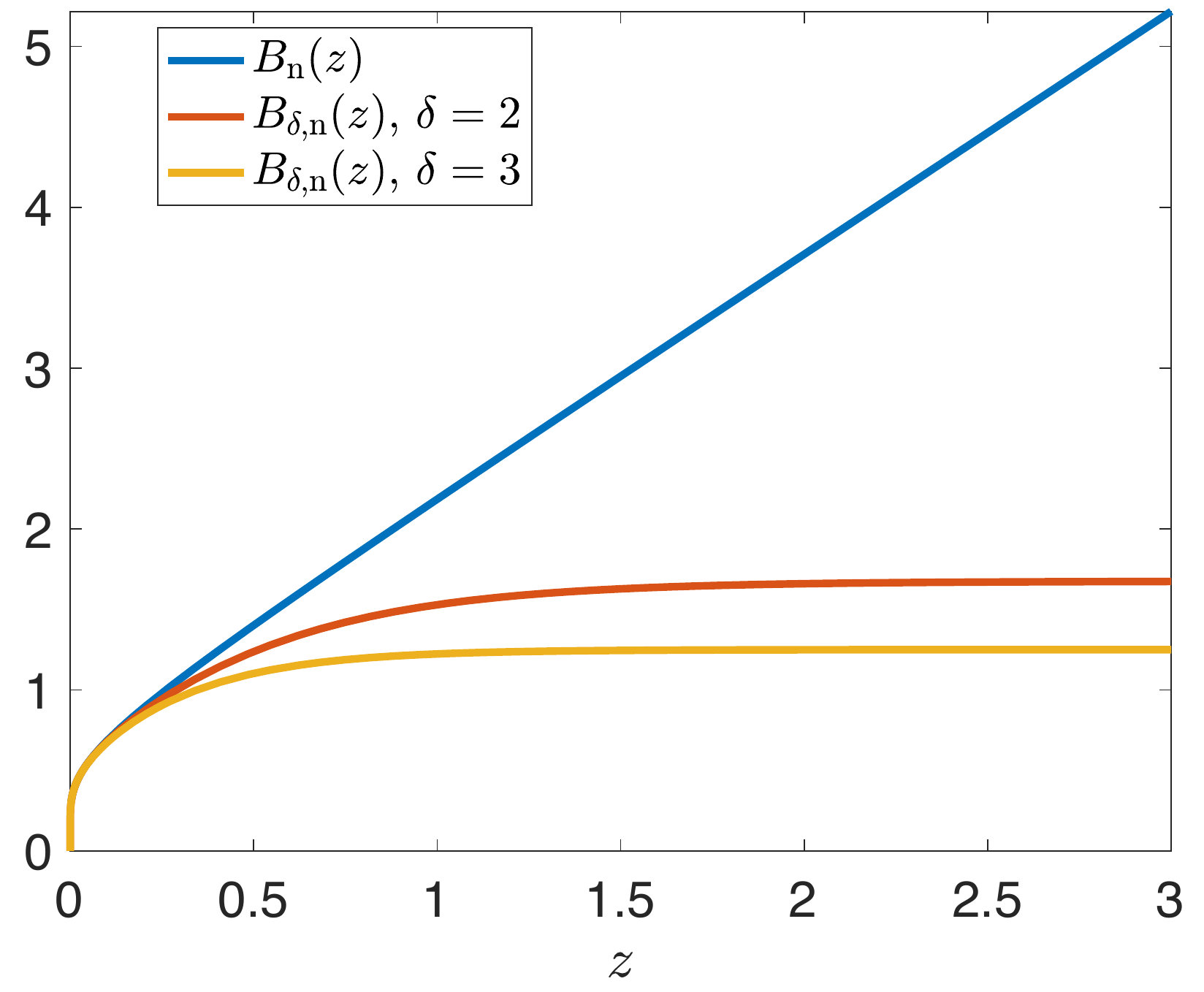}\\
\caption{Plot of the functions $B_q(z)$ (blue), $B_{\delta,q}(z)$ with $\delta=2$ (red), and $B_{\delta,q}(z)$ with $\delta=3$ (yellow) for $q={\rm t}$ (left) and $q={\rm n}$ (right). Note the different scale from Figures \ref{fig:Compare_t} and \ref{fig:Compare_n}. As in the Laplace setting, we can see that $B_{\delta,q}(z)$ agrees closely with $B_q(z)$ near $z=0$ for both $q={\rm t},{\rm n}$ and diverges only linearly as $z\to\infty$. }
\label{fig:BdeltaS}
\end{figure}

It remains to show the refined estimates \eqref{REG_diff1} and \eqref{REG_diff2} for small $k$. We begin with the tangential direction \eqref{REG_diff1}. \\

We consider the function
\begin{align*}
B_{\delta,{\rm t}}(z) = \frac{1}{ -1+ 2\log(\delta) + 2 K_0(\delta z)}
\end{align*}
on the interval $0\le z \le \frac{1}{4}$ and note that $\lambda^{\delta,{\rm t}}_k = 4\pi B_{\delta,{\rm t}}(\pi\epsilon \abs{k})$. Now, $B_{\delta,{\rm t}}(z)$ satisfies a nearly identical ODE to $B_\delta(z)$ \eqref{BdeltODE} in the Laplace setting (Section \ref{sec:REG}). In particular,
\begin{align*}
\frac{d B_{\delta,{\rm t}}}{dz} = 2\delta K_1(\delta z) B_{\delta,{\rm t}}^2(z), \qquad B_{\delta,{\rm t}}(0) = 0.
\end{align*}

Recalling the definition of $B_{\rm t}(z)$ \eqref{Bt} and its ODE \eqref{BtODE}, the absolute value of the difference $\wt B_{\rm t} := B_{\delta,{\rm t}}(z) - B_{\rm t}(z)$ satisfies
\begin{align*}
\frac{d |\wt B_{\rm t}|}{dz} = \bigg(\frac{2}{z}(B_{\delta,{\rm t}} + B_{\rm t})\wt B_{\rm t} + 2\bigg(\delta K_1(\delta z)-\frac{1}{z} \bigg) B_{\delta,{\rm t}}^2 + 2\frac{K_0(z)}{K_1(z)}B_{\rm t}\bigg){\rm sgn}(\wt B_{\rm t}), \qquad \wt B_{\rm t}(0)=0. 
\end{align*}

Now, for $0\le z\le \frac{1}{4}$, using Proposition \ref{bessel_smallz}, equation \eqref{K0_bd}, we have
\begin{align*}
B_{\delta,{\rm t}}(z) \le \frac{1}{-1 - 2\log z} \le \frac{1}{\big(-1 - \frac{3}{4}\log(\frac{1}{4})\big) - \frac{5}{4}\log z} \le \frac{4}{5\abs{\log z}}.
\end{align*}
Again, we note that Proposition \ref{bessel_smallz} is stated for $\delta z<1$ but clearly holds for $\delta z\ge 1$; therefore we do not need an upper bound on our choice of $\delta$. \\

We define the constant 
\begin{equation}\label{ct2}
c_{\rm t, 2} := \frac{4}{5\abs{\log(1/4)}} + B_{\rm t}(1/4) \approx 0.9835,
\end{equation}
and note in particular that $B_{\delta,{\rm t}}(z)+B_{\rm t}(z) \le c_{\rm t, 2} <1$. \\

Then, using the bounds \eqref{delK1_bd} from Section \ref{sec:REG} and \eqref{KBt_bd} from Section \ref{sec:tan_diff}, we obtain
\begin{equation}\label{BtProc}
\begin{aligned}
\frac{d |\wt B_{\rm t}|}{dz} &\le \frac{2c_{{\rm t},2}}{z} |\wt B_{\rm t}| + 2\delta^2z(1+\log\delta+ \abs{\log z}) \frac{16}{25\abs{\log z}^2} + 2z \\
&\le  \frac{2c_{{\rm t},2}}{z} |\wt B_{\rm t}| + 4\delta^2z(1+\log\delta) + 2z \\
&\le  \frac{2c_{{\rm t},2}}{z} |\wt B_{\rm t}| + 6\delta^2(1+\log\delta)z, \qquad \wt B_{\rm t}(0)=0. 
\end{aligned}
\end{equation}
Here we have also used that $\frac{1}{\abs{\log z}} \le \frac{1}{\abs{\log(1/4)}} \approx 0.7213 <1$. \\

We may then use a Gr\"onwall inequality in the exact same way as the Laplace setting (equation \eqref{GronLapReg}) to obtain
\begin{equation}\label{wtBt}
\abs{\wt B_{\rm t}(z)} \le \frac{6 \delta^2(1+\log\delta)}{1-c_{\rm t,2}} z^2.
\end{equation}
Again, choosing the interval $0\le z\le \frac{1}{4}$ ensures that $c_{\rm t,2}<1$ so that the Gr\"onwall argument is valid. Using $z=\pi\epsilon \abs{k}$ in \eqref{wtBt} then yields \eqref{REG_diff1}. \\

The proof of the normal direction bound \eqref{REG_diff2} proceeds similarly. We consider
\begin{align*}
B_{\delta,{\rm n}}(z) &= \frac{4}{1+ 2\log\delta +2K_0(\delta z)}
\end{align*}
on the interval $0\le z\le \frac{2}{3}$ (see Figure \ref{fig:BdeltaS}) and note that $\lambda^{\delta,{\rm n}}_k=2\pi B_{\delta,{\rm n}}(\pi\epsilon \abs{k})$. As in the Laplace and tangential Stokes settings, we have
\begin{align*}
\frac{d B_{\delta,{\rm n}}}{dz} = \frac{1}{2}\delta K_1(\delta z) B_{\delta,{\rm n}}^2(z), \qquad B_{\delta,{\rm n}}(0) = 0.
\end{align*}
Furthermore, for $0\le z\le \frac{2}{3}$, by Proposition \ref{bessel_smallz}, equation \eqref{K0_bd}, we have
\begin{align*}
B_{\delta,{\rm n}}(z) \le \frac{4}{1+2\abs{\log z}}.
\end{align*}

Recalling the form of $B_{\rm n}(z)$ \eqref{Bn} and its ODE \eqref{BnODE}, the difference $\wt B_{\rm n}(z) := B_{\delta,{\rm n}}(z) - B_{\rm n}(z)$ satisfies
\begin{align*}
\frac{d |\wt B_{\rm n}|}{dz} = \bigg(\frac{1}{2z}(B_{\delta,{\rm n}} + B_{\rm n})\wt B_{\rm n} + \frac{1}{2}\bigg(\delta K_1(\delta z)-\frac{1}{z} \bigg) B_{\delta,{\rm n}}^2 + h(z)\bigg){\rm sgn}(\wt B_{\rm n}), \qquad \wt B_{\rm n}(0)=0,
\end{align*}
where $h(z)$ is as in \eqref{h_z}. \\

Then, defining
\begin{equation}\label{cn2}
c_{\rm n, 2} := \frac{4}{1+2\abs{\log(2/3)}} + B_{\rm n}(2/3) \approx 3.8765,
\end{equation}
and noting $B_{\delta,{\rm n}}(z)+B_{\rm n}(z) \le c_{\rm n, 2} <4$ for $0\le z\le \frac{2}{3}$, we may use the bounds \eqref{delK1_bd} and Proposition \ref{h_bd} to obtain
\begin{align*}
\frac{d |\wt B_{\rm n}|}{dz} &\le \frac{c_{\rm n, 2}}{2z}|\wt B_{\rm n}| + \frac{1}{2}\delta^2z(1+\log\delta+\abs{\log z}) \frac{16}{(1+2\abs{\log z})^2}+ \frac{9}{8}z\\
&\le \frac{c_{\rm n, 2}}{2z}|\wt B_{\rm n}| + 10\delta^2(1+\log\delta)z, \qquad \wt B_{\rm n}(0)=0.
\end{align*}
Here we have also used that $\frac{1}{1+2\abs{\log z}} \le \frac{1}{1+2\abs{\log(2/3)}} \approx 0.5522 <1$. \\

We may then use a Gr\"onwall argument as in equation \eqref{GronLapReg} to yield
\begin{equation}\label{wtBn}
\abs{\wt B_{\rm n}(z)}\le \frac{20\delta^2(1+\log\delta)}{4-c_{\rm n, 2}} z^2,
\end{equation}
which is valid since $c_{\rm n, 2}<4$. Taking $z=\pi\epsilon \abs{k}$ in \eqref{wtBn} then yields \eqref{REG_diff2}, thereby completing the proof of Lemma \ref{diff_REGS}.
\end{proof}

\subsection{Proof of Theorem \ref{thm:errREGS}: error in $\delta$-regularized Stokes expression}\label{sec:REGSerr}
Finally, we use Lemma \ref{diff_REGS} to prove the error estimate of Theorem \ref{thm:errREGS}. The proof follows the same steps as the proof of Theorem \ref{thm:errS} in Section \ref{Stokes_err}, with Lemma \ref{diff_REGS} replacing Lemmas \ref{tan_eig_diff} and \ref{nor_eig_diff}.

\begin{proof}[Proof of Theorem \ref{thm:errREGS}]
We begin by taking $N'=\floor{\frac{1}{4\pi\epsilon}}$ and $M'=\floor{\frac{2}{3\pi\epsilon}}$. For vector-valued $\bu\in H^1(\T)$, we recall the notation prior to \eqref{LambdaNM} for the Fourier coefficients of $\bu$. Then, using Lemma \ref{diff_REGS}, we have
\begin{equation}\label{errScomp0}
\begin{aligned}
&\norm{\mc{L}_\epsilon^{-1}[\bu] -(\mc{L}_\epsilon^\delta)^{-1}[\bu]}_{L^2(\T)}^2 \\
&\hspace{1cm} = \sum_{\abs{k}=1}^{N'} \big(\lambda^{\rm t}_k-\lambda^{\delta,{\rm t}}_k\big)^2 \abs{\wh u_{z,k}}^2 + \sum_{\abs{k}=1}^{M'} \big(\lambda^{\rm n}_k-\lambda^{\delta,{\rm n}}_k \big)^2 \big(\abs{\wh u_{x,k}}^2 + \abs{\wh u_{y,k}}^2 \big) \\
&\hspace{2cm} + \sum_{\abs{k}=N'+1}^\infty \big(\lambda^{\rm t}_k-\lambda^{\delta,{\rm t}}_k\big)^2 \abs{\wh u_{z,k}}^2 + \sum_{\abs{k}=M'+1}^\infty \big(\lambda^{\rm n}_k-\lambda^{\delta,{\rm n}}_k\big)^2  \big(\abs{\wh u_{x,k}}^2 + \abs{\wh u_{y,k}}^2 \big) \\
&\hspace{1cm} \le C\delta^4(1+\log\delta)^2\bigg(\sum_{\abs{k}=1}^{N'} \epsilon^4k^4 \abs{\wh u_{z,k}}^2 + \sum_{\abs{k}=1}^{M'} \epsilon^4k^4 \big(\abs{\wh u_{x,k}}^2 + \abs{\wh u_{y,k}}^2 \big) \bigg) \\
&\hspace{2cm} + (4\pi)^2 \sum_{\abs{k}=N'+1}^\infty \bigg(\frac{1}{2}+\frac{1}{-1+2\log\delta}+ \pi\epsilon \abs{k} \bigg)^2 \abs{\wh u_{z,k}}^2 \\
&\hspace{2cm}+ (3\pi)^2\sum_{\abs{k}=M'+1}^\infty \bigg(1+ \frac{8}{3(1+2\log\delta)} + \pi\epsilon \abs{k}\bigg)^2 \big(\abs{\wh u_{x,k}}^2 + \abs{\wh u_{y,k}}^2 \big)  \\
&\hspace{1cm} \le C\epsilon^4\delta^4(1+\log\delta)^2\bigg((N')^2\sum_{\abs{k}=1}^{N'} k^2 \abs{\wh u_{z,k}}^2 + (M')^2 \sum_{\abs{k}=1}^{M'} k^2 \big(\abs{\wh u_{x,k}}^2 + \abs{\wh u_{y,k}}^2 \big) \bigg) \\
&\hspace{2cm}+ C\bigg( \sum_{\abs{k}=N'+1}^\infty \bigg(\frac{1}{k^2}+\frac{1}{(-1+2\log\delta)^2}\frac{1}{k^2}+\epsilon^2 \bigg)k^2 \abs{\wh u_{z,k}}^2 \\
&\hspace{3cm}+ \sum_{\abs{k}=M'+1}^\infty \bigg(\frac{1}{k^2}+\frac{1}{(1+2\log\delta)^2}\frac{1}{k^2}+\epsilon^2\bigg)k^2 \big(\abs{\wh u_{x,k}}^2 + \abs{\wh u_{y,k}}^2 \big) \bigg) \\
&\hspace{1cm}\le \bigg(C \delta^4(1+\log\delta)^2\epsilon^4((N')^2 + (M')^2) + \frac{C}{(N')^2}+ \frac{C}{(-1+2\log\delta)^2(N')^2} \\
&\hspace{5cm}+ \frac{C}{(M')^2}+ \frac{C}{(1+2\log\delta)^2(M')^2} + C\epsilon^2\bigg) \norm{\bu}_{H^1(\T)}^2.
\end{aligned}
\end{equation}

If $\bu\in H^2(\T)$, we instead obtain 
\begin{equation}\label{errScomp1}
\begin{aligned}
&\norm{\mc{L}_\epsilon^{-1}[\bu] -(\mc{L}_\epsilon^\delta)^{-1}[\bu]}_{L^2(\T)}^2 \\
&\hspace{1cm} \le C\epsilon^4\delta^4(1+\log\delta)^2\bigg(\sum_{\abs{k}=1}^{N'} k^4 \abs{\wh u_{z,k}}^2 + \sum_{\abs{k}=1}^{M'} k^4 \big(\abs{\wh u_{x,k}}^2 + \abs{\wh u_{y,k}}^2 \big) \bigg) \\
&\hspace{2cm}+ C\bigg( \sum_{\abs{k}=N'+1}^\infty \bigg(\frac{1}{k^4} + \frac{1}{(-1+2\log\delta)^2}\frac{1}{k^4} +\frac{\epsilon^2}{k^2} \bigg)k^4 \abs{\wh u_{z,k}}^2 \\
&\hspace{2cm} + \sum_{\abs{k}=M'+1}^\infty \bigg(\frac{1}{k^4} + \frac{1}{(1+2\log\delta)^2}\frac{1}{k^4} +\frac{\epsilon^2}{k^2}\bigg)k^4 \big(\abs{\wh u_{x,k}}^2 + \abs{\wh u_{y,k}}^2 \big) \bigg) \\
&\hspace{1cm}\le \bigg(C \delta^4(1+\log\delta)^2\epsilon^4 + \frac{C}{(N')^4} + \frac{C}{(M')^4} + \frac{C}{(-1+2\log\delta)^2(N')^4}  \\
&\hspace{5cm} +\frac{C}{(1+2\log\delta)^2(M')^4} + \frac{C\epsilon^2}{(N')^2} +\frac{C\epsilon^2}{(M')^2}\bigg) \norm{\bu}_{H^2(\T)}^2.
\end{aligned}
\end{equation}

Using that $N'=\floor{\frac{1}{4\pi\epsilon}}$ and $M'=\floor{\frac{2}{3\pi\epsilon}}$ in \eqref{errScomp0} and \eqref{errScomp1}, we obtain Theorem \ref{thm:errREGS}.
\end{proof}

\appendix
\section{Spectrum of the slender body approximation}\label{sec:gotz}
Here we reiterate the derivation of the spectrum of the slender body approximation $\mc{L}^\SB_\epsilon$ about $\mc{C}_\epsilon$ (Proposition \ref{KRop_eigs}), which was studied in depth by G\"otz \cite{gotz2000interactions} and later by \cite{shelley2000stokesian,tornberg2004simulating}.\\

 G\"otz essentially considers the Laplace slender body approximation \eqref{gotz2} for a non-periodic fiber with straight centerline and radius $\epsilon$. Ignoring endpoint effects at $s=\pm 1$, the expression becomes 
\begin{equation}\label{gotz1}
4\pi \overline u(s) = L(s) f(s) + \int_{-1}^1 \frac{f(s') - f(s)}{\abs{s-s'}} \, ds',
\end{equation}
where $L(s) = \log\big( \frac{4(1-s^2)}{\epsilon^2}\big)$. G\"otz actually arrives at an expression of the form \eqref{gotz1} by considering the Stokes approximation \eqref{SBT_expr} for a fiber with straight, non-periodic centerline in cross flow (unidirectional fluid velocity perpendicular to the slender body centerline). In this case, the integral operator is the same as in \eqref{gotz1} but the local term is given by $(1+L(s))f(s)$, and the left hand side is scaled by $8\pi$ rather than $4\pi$. \\

G\"otz then studies properties of the integral operator also known as the $S$-transform, 
\begin{equation}\label{S_op}
S[\varphi](s) := \int_{-1}^1  \frac{\varphi(s') - \varphi(s)}{\abs{s-s'}} \, ds',
\end{equation}
first introduced by Tuck in \cite{tuck1964some}. G\"otz and Tuck show that the operator $S$ is diagonalizable by the Legendre polynomials $\{P_k\}$, which form an orthogonal basis for $L^2(-1,1)$ and satisfy the recurrence relation
\begin{equation}\label{reccurence_rel}
P_0(t)=1, \, P_1(t)= t, \, P_{k+1}=\frac{2k+1}{k+1}t P_k(t) - \frac{k}{k+1}P_{k-1}(t).
\end{equation}
In particular, the operator $S$ satisfies 
\[ S[P_k]= -\mu_k P_k, \quad k=0,1,2,3,\dots \]
where 
\[ \mu_k := 2\sum_{j=1}^k \frac{1}{j}; \quad \mu_0=0.\]

By property of harmonic series, we have 
\begin{equation}\label{harm_ser} 
\lim_{k\to\infty} ( \mu_k - 2\log k) = 2 \gamma ,
\end{equation}
where $\gamma\approx0.5772$ is the Euler constant. \\

For a periodic filament, the expression \eqref{gotz1} must be amended to account for periodicity. Using the relation \eqref{periodic},
we may rewrite \eqref{gotz1} as our periodic Laplace slender body approximation \eqref{gotz2}. \\

In \cite{shelley2000stokesian}, Shelley and Ueda perform a spectral calculation similar to that of G\"otz to show that the eigenvalues $\mu_k^\text{per}$ of the periodic integral operator in \eqref{gotz2} satisfy 
\begin{align*}
\mu_k^\text{per} = 4\sum_{j=1}^{\abs{k}} \frac{1}{2j-1}, \, \abs{k}=1,2,3,\dots.
\end{align*}
Using the asymptotic relation \eqref{harm_ser}, we have that in the periodic setting, the eigenvalues of the forward slender body operator $\mc{L}^\SB_\epsilon$ are approximated by
\begin{equation}\label{SBT_spec_per0}
\frac{1}{\lambda^{\SB}_k} = -\frac{1}{2\pi}(\log(\pi\epsilon \abs{k}/2) +\gamma),
\end{equation}
and thus the eigenvalues of the inverse operator $(\mc{L}^\SB_\epsilon)^{-1}$ are given by 
\begin{equation}\label{SBT_spec_per}
\lambda^{\SB}_k = -\frac{2\pi}{\log(\pi\epsilon \abs{k}/2) +\gamma}.
\end{equation}

The asymptotic formula \eqref{SBT_spec_per} coincides almost exactly with the sum formula at each $\abs{k}=1,2,3,\dots$, and all subsequent analysis of the spectrum of the inverse operator $(\mc{L}^\SB_\epsilon)^{-1}$ will use the formula \eqref{SBT_spec_per}. \\

In the Stokes setting about $\mc{C}_\epsilon$, the slender body approximation \eqref{SBT_expr} becomes
\begin{equation}\label{Ceps_SBT}
 \mc{L}_\epsilon^\SB[\bm{f}](s) = \frac{1}{8\pi}\bigg[\big[({\bf I}- 3\be_z\be_z^{\rm T}){\bm f}(s) - ({\bf I}+\be_z\be_z^{\rm T}) \bigg(2 \log(\pi\epsilon/4){\bm f}(s) - \int_{\T} \frac{{\bm f}(s')-{\bm f}(s) }{|\sin (\pi(s-s'))/\pi|} \, ds' \bigg) \bigg].
\end{equation}

Dotting \eqref{Ceps_SBT} with $\be_z$ and using the above Laplace analysis, we obtain the form \eqref{KR_eigs_St} of $\lambda^{\SB,{\rm t}}_k$. Similarly, dotting \eqref{Ceps_SBT} with $\be_x$ or $\be_y$ and using the Laplace analysis yields $\lambda^{\SB,{\rm n}}_k$ as in \eqref{KR_eigs_Sn}. 

\section{Eigenvalues of the Stokes slender body PDE}\label{app_eval}
In this appendix, we calculate the eigenvalues of the operator $(\mc{L}_\epsilon)^{-1}$ for the slender body Stokes PDE about $\mc{C}_\epsilon$. The procedure is similar to the Laplace setting (Section \ref{sec:lap_calc}), but involves solving more complicated ODEs for the components of the fiber velocity. The calculations are also related to those of \cite{rummler1997eigenfunctions} for cylindrical interior domains. The tangential eigenvalues, given by \eqref{eigsT}, are calculated in Section \ref{app_tan_spec}, and the normal eigenvalues, given by \eqref{eigsN}, are calculated in Section \eqref{app_nor_spec}. 

\subsection{Calculation of the tangential spectrum}\label{app_tan_spec}
In this section we calculate the form \eqref{eigsT} of the tangential eigenvalues $\lambda^{\rm t}_k$ of the slender body PDE operator $\mc{L}_\epsilon^{-1}$. We consider the boundary value problem
\begin{equation}\label{z_flow}
\begin{aligned}
-\Delta \bu +\nabla p &=0, \quad \dive\,\bu=0 \quad \text{in } (\R^2\times \T)\backslash \overline{\mc{C}_\epsilon}   \\
\bu(z) &= e^{i\pi kz}\be_z \hspace{1.4cm} \text{on }\p \mc{C}_\epsilon.
\end{aligned}
\end{equation}
We aim to solve for $\lambda$ satisfying 
\[ \bm{f}(z) := \int_0^{2\pi} (\bm{\sigma n}) \epsilon \, d\theta = \lambda \bu(z) \]
along $\p\mc{C}_\epsilon$ by first solving \eqref{z_flow} for $(\bu,p)$ in the exterior of $\mc{C}_\epsilon$. We consider $(\bu,p)$ of the form 
\[\bu(r,z) = \begin{pmatrix} 
U_r(r) \\
0 \\
U_z(r)
\end{pmatrix} e^{i \pi kz}, \quad p(r,z) = \overline p(r) e^{i \pi kz} \]
with $U_r(\epsilon)=0$ and $U_z(\epsilon)=1$. Since the pressure is harmonic, we have that $\overline p$ satisfies 
\begin{equation}\label{press_z}
\p_{rr}\overline p +\frac{1}{r}\p_r \overline p - \pi^2 k^2 \overline p =0,
\end{equation}
and therefore, due to decay at infinity, $\overline p= c_{\rm p} K_0(\pi r\abs{k})$, a zeroth order modified Bessel function of the second kind. Using the form of $\overline p$ in the momentum equations, we have that $U_r$ and $U_z$ satisfy
\begin{align}
\label{UrEq}
\p_{rr} U_r + \frac{1}{r}\p_r U_r - \bigg(\pi^2k^2+ \frac{1}{r^2} \bigg)U_r = - c_{\rm p} \pi \abs{k} K_1(\pi r\abs{k}) \\
\label{UzEq}
\p_{rr} U_z + \frac{1}{r}\p_r U_z - \pi^2k^2 U_z = i c_{\rm p} \pi k K_0(\pi r\abs{k}).
\end{align}
Here $K_1$ is a first order modified Bessel function of the second kind. The solutions of \eqref{UrEq} and \eqref{UzEq} have the form 
\begin{align}
\label{Ur}
U_r(r) &= c_1 K_1(\pi r \abs{k}) + \frac{c_{\rm p} r}{2} K_0(\pi r \abs{k}) \\
\label{Uz}
U_z(r) &=c_0K_0(\pi r \abs{k}) - \frac{i c_{\rm p} r}{2} K_1(\pi r \abs{k}) {\rm sgn}(k),
\end{align}
and, using the boundary conditions at $r=\epsilon$, we obtain
\begin{equation}\label{c1c0}
c_1 = -\frac{c_{\rm p}\epsilon K_0(\pi\epsilon \abs{k})}{2 K_1(\pi\epsilon \abs{k})}, \quad c_0 = \frac{1}{K_0(\pi\epsilon \abs{k})} + \frac{i c_{\rm p} \epsilon K_1(\pi\epsilon \abs{k})}{2K_0(\pi\epsilon \abs{k})}{\rm sgn}(k).
\end{equation}

Finally, plugging \eqref{Ur} and \eqref{Uz} into the incompressibility condition
\begin{equation}\label{div_freez}
\p_r U_r + \frac{1}{r}U_r + i \pi k U_z =0,
\end{equation}
we find that 
\begin{equation}\label{cpz}
c_{\rm p} = \frac{- i \, 2\pi k K_1(\pi\epsilon \abs{k})}{2K_0(\pi\epsilon \abs{k})K_1(\pi\epsilon \abs{k}) + \pi\epsilon \abs{k} \big( K_0^2(\pi\epsilon \abs{k}) - K_1^2(\pi\epsilon \abs{k}) \big) }.
\end{equation}

The force density $\bm{f}$ along $\mc{C}_\epsilon$ is given by 
\begin{align*}
\bm{f}(z) &= \int_0^{2\pi}(\bm{\sigma}\bm{n}) \, \epsilon \, d\theta \\
&= \int_0^{2\pi}\bigg(-\frac{\p \bu}{\p r}- \bigg(\frac{\p\bu}{\p r}\cdot\be_r\bigg)\be_r - \frac{1}{\epsilon} \bigg(\frac{\p\bu}{\p\theta}\cdot\be_r \bigg)\be_{\theta} - \bigg(\frac{\p\bu}{\p z}\cdot\be_r \bigg)\be_z + p \be_r\bigg)\epsilon \, d\theta \\
&= -\int_0^{2\pi}\bigg((2\p_r U_r - \overline p)\be_r +(\p_r U_z + i\pi k U_r)\be_z \bigg) e^{i\pi kz} \, \epsilon \, d\theta \\
&= -2\pi \epsilon (\p_r U_z(\epsilon) + i\pi k U_r(\epsilon)) e^{i\pi kz} \be_z,
\end{align*}
where we have used that $\be_r=\cos\theta\be_x+\sin\theta \be_y$ and thus integrates to zero in $\theta$. Now, using \eqref{div_freez} and \eqref{UrEq}, we have
\begin{align*}
\p_r U_z(r) &= \frac{i}{\pi k}\bigg(\p_{rr}U_r+\frac{1}{r}\p_rU_r-\frac{1}{r^2}U_r\bigg) \\
&= i\pi kU_r(r) - i c_{\rm p} K_1(\pi r\abs{k}){\rm sgn}(k).
\end{align*}

Since $U_r(\epsilon)=0$, the force density $\bm{f}$ becomes
\begin{align*}
 \bm{f}(z) &= i c_{\rm p} \, 2\pi \epsilon K_1(\pi \epsilon \abs{k})\, {\rm sgn}(k) \, \bu(z),
 \end{align*}
and the tangential eigenvalues are thus given by \eqref{eigsT}. \\

We now show that the eigenvalues $\lambda^{\rm t}_k$ given by \eqref{eigsT} satisfy the linear growth bounds \eqref{Tgrowth}.
\begin{proof}[Proof of the growth rate \eqref{Tgrowth}]
To show the growth bound \eqref{Tgrowth}, we recall the definition of the function $B_{\rm t}(z)$ \eqref{Bt}, which satisfies $\lambda^{\rm t}_k = 4\pi B_{\rm t}(\pi\epsilon \abs{k})$. By Proposition \ref{Bt_bd} in Section \ref{sec:tan_diff} that 
\[ B_{\rm t}(z)>z \frac{K_0(z)}{K_1(z)} > z \]
 for $z>0$, which implies the lower bound of \eqref{Tgrowth}.  \\

Additionally, we have that 
\begin{align*}
B_{\rm t}(z) < \frac{zK_1(z)}{K_0(z)}
\end{align*}
for all $z> 0$, since 
\begin{align*}
\frac{zK_1(z)}{K_0(z)} - B_{\rm t}(z) &= \frac{zK_1(z)}{K_0(z)}\bigg(\frac{K_0(z)K_1(z)+z \big(K_0^2(z)-K_1^2(z) \big)}{2K_0(z)K_1(z)+z \big(K_0^2(z)-K_1^2(z) \big)} \bigg) > 0 
\end{align*}
by \eqref{diffLB}. The upper bound of \eqref{Tgrowth} then follows from Lemma \ref{bessel} using $z=\pi\epsilon \abs{k}$.
\end{proof}

\subsection{Calculation of the normal spectrum}\label{app_nor_spec}
Here we calculate the form of the normal eigenvalues $\lambda^{\rm n}_k$ \eqref{eigsN} of the slender body PDE operator $(\mc{L}_\epsilon)^{-1}$, given in Proposition \ref{Stokes_spec}. We consider here the $\be_x$ direction and note that the calculation for the $\be_y$ direction is identical. We thus consider the boundary value problem
\begin{equation}\label{x_flow}
\begin{aligned}
-\Delta \bu +\nabla p &=0, \quad \dive\,\bu=0 \qquad \hspace{3.7cm} \text{in } (\R^2\times \T)\backslash \overline{\mc{C}_\epsilon} \\
\bu(z) &=  e^{i\pi kz}\be_x = \cos\theta e^{i\pi kz}\be_r - \sin\theta e^{i\pi kz}\be_\theta \qquad \text{on }\p \mc{C}_\epsilon.
\end{aligned}
\end{equation}
We wish to solve for $\lambda$ satisfying 
\[ \bm{f}(z) := \int_0^{2\pi} (\bm{\sigma n}) \epsilon \, d\theta = \lambda \bu(z) \]
along $\p\mc{C}_\epsilon$. To do so, we first solve \eqref{x_flow} for $(\bu,p)$ in the exterior of $\mc{C}_\epsilon$. Due to the boundary conditions on $\p \mc{C}_\epsilon$, we look for $(\bu,p)$ of the form
\[ \bu(r,\theta,z) = \begin{pmatrix}
U_r(r) \cos\theta \\
-U_\theta(r) \sin\theta \\
U_z(r) \cos\theta
\end{pmatrix} e^{i\pi kz}, \quad p(r,\theta,z) = \overline p(r) \cos\theta e^{i\pi kz}. \]

Again, since $p$ is harmonic, we have 
\begin{equation}\label{press_x}
\p_{rr}\overline p +\frac{1}{r}\p_r \overline p - \bigg(\pi^2k^2+ \frac{1}{r^2}\bigg) \overline p =0, 
\end{equation}
which, along with the boundary condition $\overline p\to 0$ as $r\to\infty$, yields $\overline p=c_{\rm p} K_1(\pi r\abs{k})$, where $K_1$ is a first-order modified Bessel function of the second kind.  \\

Next, from the momentum equation $\Delta\bu=\nabla p$, using the form of the pressure \eqref{press_x}, we have that $U_r$, $U_\theta$, and $U_z$ satisfy the following ODEs:
\begin{align}
\label{Ur1}
\p_{rr}U_r +\frac{1}{r}\p_r U_r + \frac{2}{r^2}(U_\theta-U_r) - \pi^2k^2U_r &= - c_{\rm p} \pi \abs{k} K_0(\pi r\abs{k}) - \frac{c_{\rm p}}{r} K_1(\pi r\abs{k})  \\
\label{Uthe2}
\p_{rr}U_\theta +\frac{1}{r}\p_r U_\theta + \frac{2}{r^2}(U_r-U_\theta) - \pi^2 k^2U_\theta &= \frac{c_{\rm p}}{r} K_1(\pi r\abs{k}) \\
\label{Uz3}
\p_{rr}U_z +\frac{1}{r}\p_r U_z - \bigg(\pi^2 k^2+\frac{1}{r^2}\bigg)U_z &= i\pi k c_{\rm p} K_1(\pi r\abs{k}).
\end{align}

We may immediately solve the $U_z$ equation to obtain 
\begin{equation}\label{Uz3Sol}
U_z(r) = c_1 K_1(\pi r\abs{k}) - \frac{i c_{\rm p} r}{2 } K_0(\pi r\abs{k}){\rm sgn}(k), 
\end{equation}
where, using the boundary condition $U_z(\epsilon)=0$, we have
\begin{equation}\label{c1}
c_1 = \frac{i c_{\rm p}\epsilon K_0(\pi\epsilon\abs{k})}{2 K_1(\pi\epsilon \abs{k})}{\rm sgn}(k).
\end{equation}

To solve for $U_r$ and $U_\theta$, we let $U^{-}=U_r-U_\theta$ and $U^{+}=U_r+U_\theta$. Then, using the Bessel function identity 
 \begin{equation}\label{bessel_ID}
K_2(z) = K_0(z) + 2\frac{K_1(z)}{z} ,
\end{equation}
the equations \eqref{Ur1} and \eqref{Uthe2} can be decoupled as 
\begin{align}
\label{Um}
\p_{rr} U^{-} + \frac{1}{r}\p_rU^{-} - \bigg(\pi^2 k^2+\frac{4}{r^2} \bigg) U^{-} &= - c_{\rm p} \pi \abs{k} K_2(\pi r\abs{k})  \\
 \label{Up}
\p_{rr} U^{+} + \frac{1}{r}\p_r U^{+} - \pi^2k^2 U^{+} &= - c_{\rm p}\pi \abs{k} K_0(\pi r\abs{k}), 
\end{align}
along with the boundary conditions $U^+\to 0$ and $U^- \to 0$ as $r\to \infty$ while $U^{-}(\epsilon)=0$ and $U^{+}(\epsilon)=2$. \\

The general solutions to \eqref{Um} and \eqref{Up} are given by 
\begin{align}
\label{UmSol}
U^-(r) &= c_2 K_2(\pi r\abs{k})+ \frac{c_{\rm p} r}{2 } K_1(\pi r\abs{k}) \\
\label{UpSol}
U^+(r) &= c_0 K_0(\pi r\abs{k})+ \frac{c_{\rm p} r}{2 } K_1(\pi r\abs{k}) .
\end{align}
Using the boundary conditions for $U^+$ and $U^-$, we find that
\begin{equation}\label{c0c2}
c_0 = \frac{2}{K_0(\pi\epsilon \abs{k})} - \frac{c_{\rm p}\epsilon K_1(\pi\epsilon \abs{k})}{2K_0(\pi\epsilon \abs{k})}, \qquad c_2 = -\frac{c_{\rm p}\epsilon K_1(\pi\epsilon \abs{k})}{2K_2(\pi\epsilon \abs{k})}.
\end{equation}

Finally, we use the incompressibility condition $\dive\,\bu=0$ to solve for the constant $c_{\rm p}$. Plugging $U_z$,  $U_r = (U^+ + U^-)/2$, and $U_\theta = (U^+ - U^-)/2$ into the resulting equation 
\begin{equation}\label{div_freex}
\p_r U_r + \frac{1}{r}(U_r - U_\theta) + i\pi k U_z =0,
\end{equation}
we obtain 
\begin{equation}\label{cpN}
c_{\rm p} = \frac{4\pi \abs{k} K_1 K_2}{2K_0K_1K_2 + \pi\epsilon \abs{k} \big( K_1^2(K_0 + K_2) -2K_0^2K_2 \big)},
\end{equation}
where each $K_j$, $j=0,1,2$, is evaluated at $\pi\epsilon \abs{k}$. \\

Now, the force density $\bm{f}$ along $\mc{C}_\epsilon$ is given by
\begin{align*}
\bm{f}(z) &= \int_0^{2\pi}(\bm{\sigma}\bm{n}) \, \epsilon \, d\theta \\
&= \int_0^{2\pi}\bigg(-\frac{\p \bu}{\p r}- \bigg(\frac{\p\bu}{\p r}\cdot\be_r\bigg)\be_r - \frac{1}{\epsilon} \bigg(\frac{\p\bu}{\p\theta}\cdot\be_r \bigg)\be_{\theta} - \bigg(\frac{\p\bu}{\p z}\cdot\be_r \bigg)\be_z + p \, \be_r\bigg)\epsilon \, d\theta \\
&= -\int_0^{2\pi}\bigg( \big(2\p_r U_r(\epsilon)-\overline p(\epsilon) \big) \cos\theta \,\be_r - \big(\p_rU_\theta(\epsilon)+ \frac{1}{\epsilon}(U_r(\epsilon)-U_\theta(\epsilon)) \big)\sin\theta \, \be_\theta \\
&\hspace{7cm} + \big(\p_rU_z(\epsilon)+ i\pi k U_r(\epsilon) \big)\cos\theta \, \be_z \bigg) e^{i\pi kz} \, \epsilon \, d\theta \\
&= -\pi \epsilon \big(2\p_r U_r(\epsilon)+ \p_rU_\theta(\epsilon) -\overline p(\epsilon) \big) e^{i\pi kz} \be_x.
\end{align*}
Here we have used that $\be_r=\cos\theta\be_x+\sin\theta\be_y$ and $\be_\theta=-\sin\theta\be_x+\cos\theta\be_y$. By \eqref{div_freex} and the boundary conditions at $\epsilon$, we may rewrite 
\begin{align*} 
2\p_r U_r(\epsilon)+ \p_rU_\theta(\epsilon)&=\p_r U^+(\epsilon) - \frac{1}{r}U^-(\epsilon) - i\pi k U_z(\epsilon) = \p_r U^+(\epsilon).
\end{align*}

Thus $\bm{f}$ may be written as
\begin{align*}
\bm{f}(z) &= -\pi \epsilon \big( \p_rU^+(\epsilon) - c_{\rm p}K_1(\pi\epsilon \abs{k})  \big) \bu(z) \\
&= \pi \bigg( \pi \epsilon \abs{k} c_0 K_1(\pi\epsilon \abs{k})+ \frac{c_{\rm p} \epsilon}{2} \big(\pi \epsilon \abs{k} K_0(\pi\epsilon \abs{k}) + 2 K_1(\pi\epsilon \abs{k})\big)\bigg)\bu(z).
\end{align*}

Using \eqref{c0c2} and \eqref{cpN}, we then have that the normal direction eigenvalues are given by \eqref{eigsN}. \\


We next show that the eigenvalues $\lambda^{\rm n}_k$ given by \eqref{eigsN} satisfy the linear growth bounds \eqref{Ngrowth}.
\begin{proof}[Proof of the growth rate \eqref{Ngrowth}]
To show the growth rate \eqref{Ngrowth} of $\lambda^{\rm n}_k$, we begin by recalling the definition \eqref{Bn} of the function $B_{\rm n}(z)$, and recall that $\lambda^{\rm n}_k= 2\pi B_{\rm n}(\pi \epsilon \abs{k})$. \\

We will first show that
\begin{equation}\label{Bn_lower}
B_{\rm n}(z) > \frac{3}{2}z, \qquad z>0,
\end{equation}
which then implies the lower bound of \eqref{Ngrowth}. Define
\begin{equation}\label{A_z}
A(z) = \frac{K_0(z)}{K_1(z)}, 
\end{equation}
and note that by Lemma \ref{lem:bessel}, we have
\begin{equation}\label{Az_bound}
\frac{2z}{2z+1} < A(z) < \frac{z+1}{\sqrt{z^2+z+1}+1}.
\end{equation}

Using the Bessel function identity \eqref{bessel_ID} along with the definition of $A(z)$ \eqref{A_z}, we rewrite $B_{\rm n}$ as 
\begin{align*}
B_{\rm n}(z) &= \frac{N_1(z)}{D_1(z)}, \\
N_1(z) &:= -z A^2 + 2zA  + (8+z^2)  \\
D_1(z) &:= -2z A^3 - 2 A^2+ \bigg(\frac{4}{z} + 2z\bigg) A  + 2.
\end{align*}

Proving the bound \eqref{Bn_lower} is then equivalent to showing
\begin{align*}
\frac{2}{z}N_1(z) - 3D_1(z) > 0.
\end{align*}

Using the upper and lower bounds \eqref{Az_bound} on $A(z)$, we have 
\begin{align*}
\frac{2}{z}N_1(z) - 3D_1(z) &= \frac{2}{z} \bigg(3 z^2 A^3 + 3zA^2 +2zA + (8+z^2) - z^2A^2 - (6 + 3 z^2)A -3z \bigg) \\
&> \frac{2}{z} \bigg( \frac{24 z^5}{(1 + 2 z)^3} + \frac{12 z^3}{(1 + 2 z)^2} + \frac{4 z^2}{1 + 2 z}+8+ z^2 \\
&\hspace{2cm} - \frac{z^2 (1 + z)^2}{(1 + \sqrt{1 + z + z^2})^2} - \frac{3(1 + z) (2 + z^2)}{1 + \sqrt{1 + z + z^2}}  - 3 z \bigg)\\
&\ge \frac{2}{z(1 + \sqrt{1 + z + z^2})^2} \bigg( 24 z^9+ z^8 (12 - 24 \sqrt{1 + z + z^2}) + 86 z^7 \\
&\qquad + z^6 (71 - 74 \sqrt{1 + z + z^2}) + \frac{587}{2}z^5 + 664 z^4 + \frac{2121}{2}z^3+ 915 z^2+ 387 z +60  \bigg) \\
&=: \frac{2}{z(1 + \sqrt{1 + z + z^2})^2} g_0(z).
\end{align*}

Now, $g_0(z)$ satisfies
\begin{align*}
g_0(z) &\ge 12z^7\bigg(2z^2 + z+1 - 2z\sqrt{1+z+z^2} \bigg) + 74 z^5\bigg(z^2 + \frac{71}{74}z + \frac{587}{148} - z\sqrt{1 + z + z^2} \bigg) \\
&= \frac{12z^7(z^2+2z+1)}{2z^2 + z+1 + 2z\sqrt{1+z+z^2}} + \frac{z^5( 20128 z^3+ 172012 z^2+ 166708 z+ 344569)}{2(148z^2 + 142z + 587 + 148 z\sqrt{1 + z + z^2})} \\
&\ge 0,
\end{align*}
and thus $\frac{2}{z}N_1(z) - 3D_1(z)>0$, proving \eqref{Bn_lower} and therefore the lower bound of \eqref{Ngrowth}. \\

To show the upper bound of \eqref{Ngrowth}, we must show that 
\begin{equation}\label{Bn_upper}
B_{\rm n}(z) < \frac{3}{2}(z+1), \qquad z>0.
\end{equation}

Using the identity \eqref{bessel_ID} for $K_2(z)$, we have
\begin{align*}
\frac{3}{2}(z+1) - B_{\rm n}(z) &= \frac{N_2(z)}{D_2(z)}; \\
N_2(z) &:=  -3z^2(z+1) A^3 + z(z^2-3z-3)A^2 + (3z^3 +z^2+6z+6)A -z(z^2 - 3z + 5),  \\
D_2(z) &:= 2z(1+zA)(1-A^2) +4A. 
\end{align*}
Clearly $D_2(z)> 0$ for $z> 0$, since $0< A(z)<1$ by \eqref{Az_bound}. Thus to show \eqref{Bn_upper}, it remains to show that $N_2(z)> 0$ for all $z> 0$. Using \eqref{Az_bound}, we have 
\begin{align*}
N_2(z) &> \frac{z}{(1+2z)^2(1+\sqrt{1+z+z^2})^3} \wt N_2(z); \\
\wt N_2(z) &:= -12 z^7 + 12 z^6 (-2 + \sqrt{1+z+z^2}) + 9 z^5 (-5 + 2 \sqrt{1+z+z^2}) \\
&\qquad +  z^4 (-27 + 35 \sqrt{1+z+z^2})  + 3 z^3 (9 + 11 \sqrt{1+z+z^2}) +  z^2 (65 + 37 \sqrt{1+z+z^2}) \\
&\qquad  + z (73 + 62 \sqrt{1+z+z^2}) +25 (1 + \sqrt{1+z+z^2})  \big).
\end{align*}

It suffices to show that $\wt N_2(z)>0$ for $z>0$. Using that $\sqrt{1+z+z^2}>\frac{1}{2}+z$, we have
\begin{align*}
\wt N_2(z) &> 75 + 258z + 291z^2 + 161z^3 + 47z^4 - 2z^5 - 24z^7 + 12 z^6 (-1 + 2\sqrt{1 +z +z^2}).
\end{align*}
If $0<z\le 1$, then
\begin{align*}
\wt N_2(z) &\ge 75-2-24-12 =37 >0.
\end{align*}
If $z>1$, then
\begin{align*}
\wt N_2(z) &\ge 2z^5\big(-1-6z - 12z^2 + 12z\sqrt{1+z+z^2}\big) \\
&= 2z^5 \frac{-1-12z+84z^2}{1+6z + 12z^2 + 12z\sqrt{1+z+z^2}} > 0.
\end{align*}
Thus \eqref{Bn_upper} holds, which implies the upper bound of \eqref{Ngrowth}. 
\end{proof}

\section{Proof of Proposition \ref{h_bd}: bound on $h(z)$}\label{app_hbd}
Here we prove Proposition \ref{h_bd} bounding the function $h(z)$ \eqref{h_z} from the ODE \eqref{BnODE} for $B_{\rm n}(z)$. The function $B_{\rm n}(z)$ in turn corresponds to the normal spectrum of the Stokes slender body PDE. 

\begin{proof}[Proof of Proposition \ref{h_bd}]
Recall the definition of the ratio $A(z)$ \eqref{A_z} as well as the bounds \eqref{Az_bound} satisfied by $A(z)$. We use \eqref{A_z} to rewrite $h(z)$ as 
\begin{align*}
h(z) &= \frac{z}{8}\frac{N_3(z)}{D_3(z)}, \\
N_3(z) &:= 4z^4A^6 - 16 z^3 A^5 - z^2 (120 + 11 z^2) A^4 + 4 z (-40 + 3 z^2) A^3 \\
       &\qquad  +2 (-16 + 66 z^2 + 5 z^4) A^2 + 4 z (32 + z^2) A - 3 z^2 (8 + z^2) \\
D_3(z) &:= (z^2 A^3 +  z A^2 - (2 + z^2) A - z )^2.
\end{align*}

To prove that $h(z)$ satisfies \eqref{hbound}, it suffices to show that both
\begin{align}
\label{upperh}
9D_3(z) -N_3(z) &> 0  \\
\label{lowerh}
\text{and} \qquad 9D_3(z) +N_3(z) &> 0.
\end{align}

To show \eqref{upperh}, we begin by rewriting $9D_3(z) -N_3(z)$ as 
\begin{equation}\label{9D3N3}
\begin{aligned}
9D_3(z) &-N_3(z) \\
&= z^4 A^2 \big(\sqrt{5} A^2 + \sqrt{2}\big)^2 + z^3 A \big(\sqrt{34} A^2 - \sqrt{14}\big)^2 + \big(\sqrt{93} A^2 - \sqrt{33}\big)^2 \\
&\hspace{2cm} +124z A^3+68A^2 +3z^4(1-A^2) \\
&\quad -\bigg((7-2\sqrt{10})z^4 A^4+(48-4\sqrt{119})z^3 A^3 +(114-6\sqrt{341})z^2A^2+92 z A \bigg).
\end{aligned}
\end{equation}

We first consider $z \le \frac{3}{2}$. Using the upper and lower bounds of \eqref{Az_bound}, we have
\begin{align*}
9D_3(z) -N_3(z) &> A^2(124z A+68) -\bigg((7-2\sqrt{10})z^4 A^4+(48-4\sqrt{119})z^3 A^3 +(114-6\sqrt{341})z^2A^2+92 z A \bigg)\\
&\ge -z \bigg(92 + \frac{7}{10}\frac{ z^3(1 + z)^3}{(1 + \sqrt{1 + z + z^2})^3} + \frac{22}{5}\frac{z^2 (1 + z)^2}{(1 + \sqrt{1 + z + z^2})^2} + \frac{18}{5}\frac{ z (1 + z)}{1 + \sqrt{1 + z + z^2}}\bigg) \\
&\qquad  + \frac{8z(17+34z+62z^2)}{(1+2z)^2}.
\end{align*}
Now, since $z\le \frac{3}{2}<\frac{5}{3}$, we have that
\begin{equation}\label{5_3bound}
1+\sqrt{1+z+z^2} \ge \frac{5}{3} + z,
\end{equation}
and therefore the denominators of the negative terms above may be replaced and simplified, yielding
\begin{align*}
9D_3(z) -N_3(z) &\ge \frac{z}{10(5 + 3z)^3 (1 + 2 z)^2}g_1(z), \\
g_1(z) &:= 55000 - 23700 z - 15320 z^2+ 118251 z^3 + 71109 z^4 - 28971 z^5 \\
&\qquad  - 30789 z^6 - 7776 z^7 - 756 z^8.
\end{align*}

For $z\le 1$, we have 
\begin{align*}
g_1(z) 
&\ge 15980 + z^3 (49959 + 71109 z) >0. 
\end{align*}
For $1\le z\le \frac{5}{4}$, we have
\begin{align*}
g_1(z) 
&\ge \frac{1}{2}\bigg(2875 + \frac{14287131}{128}z^3 + \frac{78831}{2}z^4 \bigg) \ge \frac{19700315}{256} >0 .
\end{align*}
Finally, for $\frac{5}{4}\le z\le \frac{3}{2}$, we have
\begin{align*}
g_1(z) 
&\ge -1270 + \frac{34389}{4} z^3 + \frac{2817}{2} z^4 \ge \frac{9707635}{512} >0 .
\end{align*}
Thus for $0<z\le \frac{3}{2}$, we have
\begin{align*}
9D_3(z) -N_3(z) &\ge \frac{z}{10(5 + 3z)^3 (1 + 2 z)^2}g_1(z) > 0.
\end{align*}

When $z\ge \frac{3}{2}$, we may again use \eqref{Az_bound} with \eqref{9D3N3} to obtain
\begin{align*}
9D_3(z) -N_3(z) &\ge z^4 A^2 \big(\sqrt{5} A^2 + \sqrt{2}\big)^2 +124z A^3+68A^2  \\
&\qquad -\bigg((7-2\sqrt{10})z^4 A^4+(48-4\sqrt{119})z^3 A^3 +(114-6\sqrt{341})z^2A^2+92 z A \bigg) \\
&\ge z \bigg( \frac{992 z^3}{(1 + 2 z)^3} + \frac{272 z}{(1 + 2 z)^2} + \frac{8 z^5 (5 + 40 z + 182 z^2 + 408 z^3 + 528 z^4)}{5 (1 + 2 z)^6}\\
&\hspace{1cm}  - \frac{7 z^3 (1 + z)^4}{10 (1 + \sqrt{1 + z + z^2})^4} - \frac{22 z^2 (1 + z)^3}{5 (1 + \sqrt{1 + z + z^2})^3} \\
&\hspace{2cm} - \frac{18 z (1 + z)^2}{5 (1 + \sqrt{1 + z + z^2})^2} - \frac{92 (1 + z)}{1 + \sqrt{1 + z + z^2}} \bigg).
\end{align*}

Noting that
\begin{equation}\label{3_2bound}
1+\sqrt{1+z+z^2} \ge 1+\sqrt{\frac{1}{4} + z+z^2} = \frac{3}{2}+z,
\end{equation}
we may replace the denominators of the negative terms above and simplify to yield
\begin{align*}
9D_3(z) -N_3(z) &\ge \frac{8}{5(1+2z)^6(3+2z)^4} \bigg( -3105 - 32886 z - 162858 z^2 - 483254 z^3 - 891642 z^4 \\
&\hspace{2.5cm}  - 966847 z^5 - 508816 z^6 - 12905 z^7 + 109108 z^8 + 93644 z^9 \\
&\hspace{3.5cm} + 134944 z^{10} + 122960 z^{11} + 51264 z^{12} + 8000 z^{13} \bigg) =: g_2(z).
\end{align*}
Since each of the terms of $g_2(z)$ containing the highest powers of $z$ are positive, once $g_2(z^*)>0$ for some $z^*\ge 1$, we will have $g_2(z)>0$ for all $z\ge z^*$. In particular, since  
\begin{align*}
\textstyle g_2(\frac{3}{2}) = \frac{646907}{163840} \approx 3.948 >0,
\end{align*}
we must have $g_2(z)>0$ for $z\ge \frac{3}{2}$. This establishes \eqref{upperh}. \\

Similarly, to show \eqref{lowerh}, we rewrite
\begin{equation}\label{lowerh3}
\begin{aligned}
9D_3(z) + N_3(z)  &= z^4 A^2(\sqrt{13}A^2-\sqrt{19})^2+ (2\sqrt{247}-29)z^4A^4 +2z^3A(A^2-\sqrt{11})^2 \\
&\hspace{2cm}+150z^2A^2 +164 zA +4A^2 \\
&\qquad - \bigg( 3z^4 + 2(12-2\sqrt{11})z^3A^3+ 147z^2A^4 + 15z^2 +196zA^3 \bigg).
\end{aligned} 
\end{equation}

When $z\ge 1$, we may use the upper and lower bounds \eqref{Az_bound} for $A(z)$ along with \eqref{3_2bound} and the fact that  $\sqrt{11}>13/4$ to obtain 
\begin{align*}
9D_3(z) + N_3(z)  &\ge \frac{z}{(1+2z)^6(3+2z)^4}\bigg(-4704 - 49399 z - 218100 z^2 - 514979 z^3 - 636860 z^4 \\
&\hspace{2cm} - 132064 z^5 + 944672 z^6 + 1784464 z^7 + 1643712 z^8 + 834432 z^9 \\
&\hspace{3cm} + 211456 z^{10} + 18432 z^{11} \bigg) =: g_3(z).
\end{align*}
Again, since the terms of $g_3(z)$ containing the highest powers of $z$ are all positive and 
\[ g_3(1) = \frac{3881062}{455625} >0, \]
we have $g_3(z)>0$ for $z\ge 1$. \\

For $z<1$, we first note that $A(z)$ is monotone increasing by \eqref{Az_bound}, and thus
\[A(z) \le A(1) \approx 0.6995 < 0.7.  \]
Therefore we have that 
\begin{align*}
4zA\big(41-49A^2 \big) &\ge 4zA \frac{1699}{100} = \frac{1699}{25}zA,
\end{align*}
and thus we may rewrite \eqref{lowerh3} as the bound
\begin{align*}
9D_3(z) + N_3(z)  &\ge 150z^2A^2 + \frac{1699}{25}zA +4A^2  - \bigg( 3z^4 + 2(12-2\sqrt{11})z^3A^3+ 147z^2A^4 + 15z^2 \bigg).
\end{align*}
Then, using the bounds \eqref{Az_bound} on $A(z)$ along with \eqref{5_3bound} and $\sqrt{11}>13/4$, we obtain
\begin{align*}
9D_3(z) + N_3(z) &\ge \frac{z^2}{25(1 + 2 z)^2(5 + 3 z)^4} (1841700 + 6025975 z + 12933975 z^2 + 14989020 z^3 \\
&\hspace{3cm} + 7067328 z^4 - 280899 z^5 - 1175175 z^6 - 275400 z^7 - 24300 z^8) \\
&\ge \frac{85926 z^2}{25(1 + 2 z)^2(5 + 3 z)^4} >0,
\end{align*}
since $z<1$. Thus we have that $9D_3(z) + N_3(z)>0$ for all $z$, yielding \eqref{lowerh} and completing the proof of Proposition \ref{h_bd}. 
\end{proof}


\bibliographystyle{abbrv} 
\bibliography{bib0}


\end{document}